\documentclass[reqno]{amsart}

\usepackage[usenames,dvipsnames]{pstricks}
\usepackage{epsfig}

\usepackage{color}
\date{\today}

\usepackage[latin1]{inputenc}
\usepackage[T1]{fontenc}
\usepackage{amsmath, amsthm}
\usepackage[english]{babel}
\usepackage{amssymb}
\usepackage{latexsym}
\usepackage{graphicx}
\usepackage[all]{xy}

\newtheorem{theorem}{Theorem}[section]
\newtheorem{lemma}[theorem]{Lemma}
\newtheorem{proposition}[theorem]{Proposition}
\newtheorem{corollary}[theorem]{Corollary}
\theoremstyle{definition}
\newtheorem{definition}[theorem]{Definition}
\newtheorem{example}[theorem]{Example}

\newtheorem{remark}[theorem]{Remark}
\newcommand{\ot}{\otimes}
\newcommand{\co}{\circ}

\begin{document}

\begin{center}

{\huge{\bf Multiplication alteration by two-cocycles. The non-associative version}}

\end{center}

\ \\
\begin{center}
{\bf J.N. Alonso \'Alvarez$^{1}$, J.M. Fern\'andez Vilaboa$^{2}$, R.
Gonz\'{a}lez Rodr\'{\i}guez$^{3}$}
\end{center}

\ \\
\hspace{-0,5cm}$^{1}$ Departamento de Matem\'{a}ticas, Universidad
de Vigo, Campus Universitario Lagoas-Marcosende, E-36280 Vigo, Spain
(e-mail: jnalonso@ uvigo.es)
\ \\
\hspace{-0,5cm}$^{2}$ Departamento de \'Alxebra, Universidad de
Santiago de Compostela.  E-15771 Santiago de Compostela, Spain
(e-mail: josemanuel.fernandez@usc.es)
\ \\
\hspace{-0,5cm}$^{3}$ Departamento de Matem\'{a}tica Aplicada II,
Universidad de Vigo, Campus Universitario Lagoas-Mar\-co\-sen\-de, E-36310
Vigo, Spain (e-mail: rgon@dma.uvigo.es)
\ \\

{\bf Abstract} In this paper  we introduce the theory of multiplication alteration by two-cocycles for non-associative structures like non-associative bimonoids with left (right) division.  Also we explore the connections between Yetter-Drinfeld modules for Hopf quasigroups, projections of Hopf quasigroups, skew pairings, and quasitriangular structures, obtaining the non-associative version of the main results proved by Doi and Takeuchi for Hopf algebras.

\vspace{0.5cm}



\section{Introduction and preliminaries}

Let $R$ be a commutative ring  with a unit and denote the tensor product over $R$ by $\ot$. In \cite{SW} we can find one of the first interesting examples of multiplication alteration by  a two-cocycle for $R$-algebras. In this case Sweedler proved that, if $U$ is an associative unitary $R$-algebra with a commutative subalgebra $A$ and $\sigma=\sum a_{i}\ot b_{i}\ot c_{i}\in A \ot A\ot A$ is an Amistur two-cocycle, then $U$ admits a new associative and unitary product defined by $u\bullet v=\sum a_{i}\;ub_{i}\;vc_{i}$ for all $u,v\in U$. Moreover, if $U$ is central separable, $U$ with the new product is still central separable  and  is isomorphic to the Rosenberg-Zelinsky central separable algebra obtained from the two-cocycle $\sigma^{-1}$ (see \cite{RZ}). Later,  Doi discovered in \cite{Doi1} a new contruction to modify the algebra structure of a bialgebra $A$ over a field ${\Bbb F}$ using an invertible two-cocycle $\sigma$ in $A$. In this case if $\sigma: A\ot A\rightarrow {\Bbb F}$ is the two-cocycle, the new product on $A$ is defined by 
$$a\ast b=\sum \sigma(a_{1}\ot b_{1})a_{2}b_{2}\sigma^{-1}(a_{3}\ot b_{3})$$
for  $a,b\in A$.  With the new algebra structure and the original coalgebra structure, $A$ is a new bialgebra denoted by $A^{\sigma},$ and if $A$ is a Hopf algebra with antipode $\lambda_{A}$, so is $A^{\sigma}$ with   antipode  given by 
$$\lambda_{A^{\sigma}}(a)=\sum \sigma(a_{1}\ot \lambda_{A}(a_{2}))\lambda_{A}(a_{3})\sigma^{-1}( \lambda_{A}(a_{4})\ot a_{5}).$$
for $a\in A$. One of the main remarkable examples of this construction is the Drinfeld double of a Hopf algebra $H$. If $H^{\ast}$ is the dual of $H$ and $A=H^{\ast cop}\ot H$, the Drinfeld double $D(H)$ can be obtained as $A^{\sigma}$ where $\sigma$ is defined by  $\sigma((x\ot g)\ot (y\ot h))=x(1_{H})y(g)\varepsilon_{H}(h)$ for $x,y\in H^{\ast}$ and $g,h\in H$. As was pointed by Doi and Takeuchi in \cite{D-T} "this will be the the shortest description of the multiplication  of $D(H)$". 

A particular case of alterations of products by two-cocycles are provided by invertible skew pairings on bialgebras. If $A$ and $H$ are bialgebras and $\tau:A\ot H\rightarrow {\Bbb F}$ is an invertible skew pairing, Doi and Takeuchi defined in \cite{D-T} a new biagebra $A\bowtie_{\tau} H$ in the following way: The morphism $\omega:A\ot H\ot A\ot H\rightarrow {\Bbb F}$ defined by $\omega((a\ot g)\ot (b\ot h))=\varepsilon_{A}(a)\varepsilon_{H}(h)\tau(b\ot g)$, for $a,b\in A$ and $g,h\in H$, is a  two-cocycle  in $A\ot H$ and $A\bowtie_{\tau} H=(A\ot H)^{\omega}$. The construction of $A\bowtie_{\tau} H$ also generalizes the Drinfeld double because $H^{\ast cop}$ and $H$ are skew-paired. Moreover, $A\bowtie_{\tau} H$ is an example of  Majid's double crossproduct $A\bowtie H$ (see \cite{MajidPA}, \cite{M}) where the left $H$-module structure of $A$, denoted by $\varphi_{A}$,  and the right $A$-module structure of $H$, denoted by $\phi_{H}$,  are defined by 
$$\varphi_{A}(h\ot a)=\sum \tau(a_{1}\ot h_{1})a_{2}\tau^{-1}(a_{3}\ot h_{2}), \;\;\; \phi_{H}(h\ot a)=\sum \tau(a_{1}\ot h_{1})h_{2}\tau^{-1}(a_{2}\ot h_{3})$$
for $a\in A$ and $h\in H$. 

On the other hand, a relevant class of Hopf algebras are quasitriangular Hopf algebras. This kind of Hopf algebraic objects were introduced by Drinfeld \cite{Drin} and provide solutions of the quantum Yang-Baxter equation: If $H$ is quasitriangular with morphism $R:{\Bbb F}\rightarrow H\ot H$ and $M$ is a left $H$-module with action $\varphi_{M}$, the endomorphism $T:M\ot M\rightarrow M\ot M$ defined by $T(m\ot m^{\prime})=\sum \varphi_{M}(R^{1}\ot m)\ot \varphi_{M}(R^{2}\ot m^{\prime})$ is a solution of the quantum Yang-Baxter equation. If moreover, for a Hopf algebra $A$ there exists an invertible  skew paring  $\tau:A\ot H\rightarrow {\Bbb F}$, by Proposition 2.5 of \cite{D-T}, we have that $g:A\bowtie_{\tau} H\rightarrow H$, defined by $g(a\ot h)=\sum \tau(a\ot R^{1})R^{2}h$ for $a\in A$, $h\in H$,  is a Hopf algebra projection. Therefore skew pairings and quasitriangular structures give special cases of Hopf algebra projections. These kind of projections were  completely described  by Radford in \cite{RAD}, who gave equivalent conditions for the tensor product of two
Hopf algebras $A$ and $H$ (equipped with smash product algebra and
smash coproduct coalgebra) to be a Hopf algebra, and characterized
such objects via bialgebra projections. Later, Majid in \cite{MAJ2}
interpreted this result in the modern context of braided categories
and, using the bosonization process, stated that there is a one
to one correspondence between Hopf algebras in the category of
left-left Yetter-Drinfeld modules, denoted by $\;^{H}_{H}{\mathcal
Y} {\mathcal D}$, and Hopf algebras $B$ with a projection,
i.e., with Hopf algebra morphisms $f:H\rightarrow B$,
$g:B\rightarrow H$ such that $g\co f=id_{H}$.  Therefore, if we came back to the Hopf algebra projection induced by two Hopf algebras $A$ and $H$, such that $H$ is quasitriangular, and a skew pairing $\tau:A\ot H\rightarrow {\Bbb F}$, we obtain by the Majid's bijection a Hopf algebra in $\;^{H}_{H}{\mathcal
Y} {\mathcal D}$. As was proved in \cite{P-M}, this Hopf algebra (or braided Hopf algebra) is $A$ with a modified product and antipode. If $A\rtimes H$ denotes the bosonization of $A$,  $A\rtimes H$ and $A\bowtie_{\tau} H$ are isomorphic as Hopf algebras.

An interesting generalization of Hopf algebras are non-associative Hopf
algebras. This notion in a category of vector spaces was introduced by P\'erez-
Izquierdo in \cite{PI07} in order to construct the universal enveloping algebra for Sabinin algebras, prove a Poincar\'e-Birkhoff-Witt Theorem for Sabinin algebras and give
a non-associative version of the Milnor-Moore theorem. Later, Klim and Majid
in \cite{Majidesfera} considered the case where in addition an antipode is defined in order
to understand the structure and relevant properties of the algebraic 7-sphere.
Moreover, non-associative Hopf algebras arise naturally related with other structures
in various non-associative contexts (\cite{TV}, \cite{BMP-I12}, \cite{our1}, \cite{our2}), and in recent years, interesting research
about its specific structure has been developed (see for example \cite{PIS}). As in the
quasi-Hopf setting, non-associative Hopf algebras are not associative, but the lack
of this property is compensated by some axioms involving the division operation.
It includes the example of an enveloping algebra $U(L)$ of a Malcev algebra (see \cite{Majidesfera},
\cite{TV}) and, more generally, of an arbitrary Sabinin algebra (see \cite{PI07}) as well as the
notion of the loop algebra RL of a loop L (see \cite{BMP-I12}, \cite{MPIS14}). Then, non-associative Hopf
algebras unify Moufang loops and Malcev algebras, and, more generally, formal
loops and Sabinin algebras, in the same way that Hopf algebras unify groups and
Lie algebras.

The main motivation of this paper is to introduce the theory of alteration multiplication, in the sense of Doi, for non-associative algebraic structures in monoidal categories. An outline of the paper is as follows.  In Section 2 we recall some definitions and we prove some useful results for the next sections. In the third section  we prove that for a non-associative bimonoid $A$ with a left (right) division, if there exists an invertible two-cocycle $\sigma$, it is possible to define a new non-associative bimonoid $A^{\sigma}$ with a left (right) division. Then, if $A$ is a Hopf quasigroup, in the sense of Klim and Majid, $A^{\sigma}$ is a Hopf quasigroup, and if $A$ is a Hopf algebra we recover  the Doi's construction. In Section 4 we introduce the notion of skew pairing and  prove that, as in the associative Hopf algebra setting, if there exists a skew pairing, for two non-associative bimonoids $A$, $H$ with a left (right) division,  we can define a new non-associative bialgebra $A\bowtie_{\tau} H$ with a left (right) division such that $A\bowtie_{\tau} H=(A\ot H)^{\omega}$ for some two-cocycle $\omega$ induced by $\tau$. This implies a similar result for Hopf quasigroups and, as in the Hopf world,  we prove in Section 5 that $A\bowtie_{\tau} H$  can be described in terms of double crossproducts. Finally, using the theory of Hopf quasigroup projections developed in \cite{our1}, we show that for a Hopf quasigroup $A$ and a quasitriangular Hopf quasigroup $H$, if there exists an invertible  skew paring  $\tau$, it is possible to obtain a strong Hopf quasigroup  projection and as a consequence of the results proved in \cite{our1}, we obtain that $A$ admits a structure of Hopf quasigroup in the category $\;^{H}_{H}{\mathcal
Y} {\mathcal D}$ introduced in \cite{our1}. This last result provides a way to prove that there exist examples of "true" braided Hopf quasigroups (see Example \ref{thelastone}).

In this paper we will work in a monoidal setting. Following \cite{Mac}, recall that a monoidal category is a category ${\mathcal C}$ together with a functor $\ot :{\mathcal C}\times {\mathcal C}\rightarrow {\mathcal C}$, called tensor product, an object $K$ of ${\mathcal C}$, called the unit object, and  families of natural isomorphisms 
$$a_{M,N,P}:(M\ot N)\ot P\rightarrow M\ot (N\ot P),$$
$$r_{M}:M\ot K\rightarrow M, \;\;\; l_{M}:K\ot M\rightarrow M,$$
in ${\mathcal C}$, called  associativity, right unit and left unit constraints, respectively, satisfying the Pentagon Axiom and the Triangle Axiom, i.e.,
$$a_{M,N, P\ot Q}\co a_{M\ot N,P,Q}= (id_{M}\ot a_{N,P,Q})\co a_{M,N\ot P,Q}\co (a_{M,N,P}\ot id_{Q}),$$
$$(id_{M}\ot l_{N})\co a_{M,K,N}=r_{M}\ot id_{N},$$
where for each object $X$ in ${\mathcal 
C}$, $id_{X}$ denotes the identity morphism of $X$. A monoidal category is called strict if the associativity, right unit and left unit constraints are identities. It is a wellknown  fact (see for example \cite{Christian}) that every non-strict monoidal category is monoidal equivalent to a strict one. Then, in general, when we work with monoidal categories, we can assume without loss of generality, that the category is strict. This lets us to treat monoidal categories as if they were strict and, as a consequence, the results proved in this paper hold for every non-strict symmetric monoidal category, for example the category of vector spaces over a field ${\Bbb F}$, or the category of left modules over a commutative ring $R$.
For
simplicity of notation, given objects $M$, $N$, $P$ in ${\mathcal
C}$ and a morphism $f:M\rightarrow N$, we write $P\ot f$ for
$id_{P}\ot f$ and $f \ot P$ for $f\ot id_{P}$.

A braiding for a strict monoidal category ${\mathcal C}$ is a natural family of isomorphisms $t_{M,N}:M\ot N\rightarrow N\ot M$ subject to the conditions 
$$
t_{M,N\ot P}= (N\ot t_{M,P})\co (t_{M,N}\ot P),\;\;
t_{M\ot N, P}= (t_{M,P}\ot N)\co (M\ot t_{N,P}).
$$

A strict braided monoidal  category ${\mathcal C}$ is a strict monoidal category with a braiding. These categories were introduced by Joyal and Street in \cite{JS1} (see also  \cite{JS2}) motivated by the theory of braids and links in topology. Note that, as a consequence of the definition, the equalities $t_{M,K}=t_{K,M}=id_{M}$ hold, for all object  $M$ of ${\mathcal C}$.  
 
If the braiding satisfies that  $t_{N,M}\co t_{M,N}=id_{M\ot N},$
 for all $M$, $N$ in ${\mathcal C}$, we will say that ${\mathcal C}$  is symmetric. In this case, we call the braiding $t$ a symmetry for the category ${\mathcal C}$.

Throughout this paper $\mathcal C$ denotes a strict symmetric monoidal category with tensor product $\ot$, unit object $K$ and symmetry $c$. Following \cite{Bespa}, we also assume that in ${\mathcal C}$ every idempotent morphism splits, i.e., for any morphism $q:X\rightarrow X$ such that $q\co q=q$ there exist an object $Z$, called the image of $q$, and morphisms $i:Z\rightarrow X$, $p:X\rightarrow Z$ such that $q=i\co p$ and $p\co i=id_Z$. The morphisms $p$ and $i$ will be called a factorization of $q$. Note that $Z$, $p$ and $i$ are unique up to isomorphism. The categories satisfying this property constitute a broad class that includes, among others, the categories with epi-monic decomposition for morphisms and categories with equalizers or with coequalizers.  For example, complete bornological spaces is a symmetric monoidal closed category that is not  abelian, but it does have coequalizers (see \cite{Meyer}). On the other hand, let {\bf Hilb} be the category whose objects are complex Hilbert spaces and whose morphisms are the continuous linear maps. Then {\bf Hilb} is not an abelian and closed category but it is a symmetric monoidal category (see \cite{Kad}) with coequalizers.

A magma in ${\mathcal C}$ is a pair $A=(A,  
\mu_{A})$ where  $A$ is an object in ${\mathcal C}$ and $\mu_{A}:A\otimes A\rightarrow A$ (product) is a morphism in ${\mathcal C}$.
A unital magma  in ${\mathcal C}$ is a triple $A=(A, \eta_A, 
\mu_{A})$ where $(A,  
\mu_{A})$ is a magma in ${\mathcal C}$ and
 $\eta_{A}: K\rightarrow A$ (unit) is a morphism in ${\mathcal C}$ such that 
$\mu_{A}\circ (A\otimes \eta_{A})=id_{A}=\mu_{A}\circ
(\eta_{A}\otimes A)$. A monoid in ${\mathcal C}$ is a unital magma $A=(A, \eta_A, \mu_{A})$ in ${\mathcal C}$ satisfying   $\mu_{A}\circ (A\otimes
\mu_{A})=\mu_{A}\circ (\mu_{A}\otimes A)$, i.e., the product $\mu_{A}$ is associative. Given two unital magmas (monoids) $A$ and $B$,
$f:A\rightarrow B$ is a morphism of unital magmas (monoids) if  $f\circ \eta_{A}= \eta_{B}$ and $\mu_{B}\circ (f\otimes f)=f\circ \mu_{A}$. 

Also,
if $A$, $B$ are unital magmas (monoids) in ${\mathcal C}$, the object $A\otimes
B$ is a unital magma (monoid) in
 ${\mathcal C}$ where $\eta_{A\otimes B}=\eta_{A}\otimes \eta_{B}$
and $\mu_{A\otimes B}=(\mu_{A}\otimes \mu_{B})\circ (A\otimes c_{B,A}\otimes B)$. If $A=(A, \eta_A,  \mu_{A})$ is a unital magma also is $A^{op}=(A, \eta_A,  \mu_{A}\co c_{A,A})$.

A comagma in ${\mathcal C}$ is a pair ${D} = (D, \delta_{D})$ where $D$ is an object in ${\mathcal C}$ and $\delta_{D}:D\rightarrow D\otimes D$ (coproduct) is a morphism in
${\mathcal C}$. A counital comagma in ${\mathcal C}$ is a triple ${D} = (D,
\varepsilon_{D}, \delta_{D})$ where $(D, \delta_{D})$ is a comagma in ${\mathcal
C}$ and $\varepsilon_{D}: D\rightarrow K$ (counit) is a  morphism in
${\mathcal C}$ such that $(\varepsilon_{D}\otimes D)\circ
\delta_{D}= id_{D}=(D\otimes \varepsilon_{D})\circ \delta_{D}$.  A comonoid  in ${\mathcal C}$ is a counital comagma in ${\mathcal C}$ satisfying  $(\delta_{D}\otimes D)\circ \delta_{D}= (D\otimes \delta_{D})\circ \delta_{D}$, i.e., the coproduct $\delta_{D}$ is coassociative. If $D$ and
$E$ are counital comagmas (comonoids) in  ${\mathcal C}$,
$f:D\rightarrow E$ is a  morphism of counital comagmas (comonoids) if $\varepsilon_{E}\circ f
=\varepsilon_{D}$,  and $(f\otimes f)\circ
\delta_{D} =\delta_{E}\circ f$.  

Moreover, if $D$, $E$ are counital comagmas (comonoids) in ${\mathcal C}$,
the object $D\otimes E$ is a counital comagma (comonoid) in ${\mathcal C}$ where
$\varepsilon_{D\otimes E}=\varepsilon_{D}\otimes \varepsilon_{E}$
and $\delta_{D\otimes E}=(D\otimes c_{D,E}\otimes E)\circ(
\delta_{D}\otimes  \delta_{E})$. If $D=(D, \varepsilon_{D}, \delta_{D})$ is a counital comagma also is $D^{cop}=(D, \varepsilon_{D}, c_{D,D}\co\delta_{D})$.

If $A$ is a magma, $B$ a comagma and $f:B\rightarrow A$, $g:B\rightarrow A$ are morphisms, we define the convolution product by $f\ast g=\mu_{A}\circ
(f\otimes g)\circ \delta_{B}$. If $A$ is unital and $B$ counital, the morphism $f:B\rightarrow A$ is convolution invertible if there exists $f^{-1}:B\to A$ such that $f \ast f^{-1}=f^{-1}\ast f=\varepsilon_B\ot \eta_A$.

\section{non-associative bimonoids}

In this section we introduce the definition of non-associative bimonoid with left (right) division. We give some properties and establish the relation with left (right) Hopf quasigroups.

\begin{definition}
\label{naHadef}
{\rm A non-associative bimonoid in the category $\mathcal C$ is a unital magma $(H,\eta_H,\mu_H)$ and a comonoid $(H,\varepsilon_H,\delta_H)$ such that $\varepsilon_H$ and $\delta_H$ are morphisms of unital magmas (equivalently, $\eta_H$ and $\mu_H$ are morphisms of counital comagmas). Then the following identities hold:
\begin{equation}
\label{eta-eps}
\varepsilon_{H}\co \eta_{H}=id_{K},
\end{equation}
\begin{equation}
\label{mu-eps}
\varepsilon_{H}\co \mu_{H}=\varepsilon_{H}\ot \varepsilon_{H},
\end{equation}
\begin{equation}
\label{delta-eta}
\delta_{H}\co \eta_{H}=\eta_{H}\ot \eta_{H},
\end{equation}
\begin{equation}
\label{delta-mu}
\delta_{H}\co \mu_{H}=(\mu_{H}\ot \mu_{H})\co \delta_{H\ot H}.
\end{equation}

We say that $H$ has a left division if moreover there exists a morphism $l_{H}:H\ot H\to H$ in $\mathcal C$ (called the left division of $H$) such that
\begin{equation}
\label{leftdivision}
l_{H}\co(H\ot\mu_H)\co(\delta_H\ot H)=\varepsilon_H\ot H=\mu_H\co(H\ot l_{H})\co(\delta_H\ot H).
\end{equation}

A morphism $f:H\rightarrow B$ between non-associative bimonoids  $H$ and $B$ is a morphism of  unital magmas and comonoids.

We say that a non-associative bimonoid $H$ in the category $\mathcal C$  is cocommutative if $\delta_{H}=c_{H,H}\co \delta_{H}$.
}
\end{definition}

\begin{remark}
\label{rightnaHadef}
{\rm
We have the corresponding notion of  non-associative bimonoid with right division, replacing the left division $l_{H}$ by a right division $r_H:H\ot H\to H$ that, instead of (\ref{leftdivision}), satisfies:
\begin{equation}
\label{rightdivision}
r_H\co(\mu_H\ot H)\co(H\ot\delta_H)=H\ot\varepsilon_H=\mu_H\co(r_H\ot H)\co(H\ot\delta_H).
\end{equation}

Note that, if ${\mathcal C}$ is the category of vector spaces over a field ${\Bbb F}$, the notion of non-associative bimonoid with left and a right division is the one introduced by P\'erez-Izquierdo in \cite{PI07} with the name of  unital $H$-bialgebra.
}
\end{remark}

Now we give some properties about non-associative bimonoids with left division.

\begin{proposition}
\label{divisionunica}
Let $H$ be a non-associative bimonoid. There exists  a left division $l_{H}$ if and only if the morphism 
$h:H\ot H\rightarrow H\ot H$ defined as $h=(H\ot \mu_H)\co (\delta_H\ot H)$ is an isomorphism. As a consequence, a left division $l_{H}$ is uniquely determined.

Similarly, there exists  a right division $r_{H}$ if and only if the morphism 
$d:H\ot H\rightarrow H\ot H$ defined as $d=(\mu_H\ot H)\co (H\ot \delta_H)$ is an isomorphism. As a consequence, a right division $r_{H}$ is uniquely determined.

\end{proposition}

\begin{proof}
Let $l_{H}:H\ot H\rightarrow H$ be a left division. Define $h^{\prime}=(H\ot l_{H})\co (\delta_{H}\ot H)$. Then, by (\ref{leftdivision}) and the coassociativity of $\delta_{H}$, we obtain that $h^{\prime}$ is the inverse of $h$. 

On the other hand, if $h$ is an isomorphism, using the coassociativity of $\delta_{H}$, we obtain that 
$$(\delta_{H}\ot H)\co h^{-1}\co h=\delta_{H}\ot H=(H\ot (h^{-1}\co h))\co (\delta_{H}\ot H)=(H\ot h^{-1})\co (\delta_{H}\ot H)\co h$$ 
and the equality
\begin{equation}
\label{h-1tdelta}
(\delta_{H}\ot H)\co h^{-1}=(H\ot h^{-1})\co (\delta_{H}\ot H)
\end{equation}
holds. 

Then the morphism $l_{H}=(\varepsilon_H\ot H)\co h^{-1}$ is a left division for $H$. Indeed, trivially $l_{H}\co h=\varepsilon_H\ot H$ and, by (\ref{h-1tdelta}), we have that
$$\mu_H\co (H\ot l_{H})\co (\delta_H\ot H)=\mu_H\co (H\ot \varepsilon_H\ot H)\co (\delta_H\ot H)\co h^{-1}=(\varepsilon_H\ot H)\co h\co h^{-1}=\varepsilon_H\ot H.$$

The proof for the right side is similar and we leave the details to the reader. Note that in this case $d^{-1}=(r_{H}\ot H)\co (H\ot \delta_{H})$ and $r_{H}=(H\ot \varepsilon_{H})\co d^{-1}.$

\end{proof}

\begin{remark}
{\rm In the conditions of the previous proposition, if $h$ is an isomorphism, we obtain 
\begin{equation}
\label{h-1delta}
 h^{-1}\co \delta_{H}=H\ot \eta_{H},
\end{equation}
\begin{equation}
\label{h-1mu}
\mu_{H}\co h^{-1}=\varepsilon_{H}\ot H.
\end{equation}

In a similar way if $d$ is an isomorphism,
\begin{equation}
\label{h-1delta-r}
 d^{-1}\co \delta_{H}=\eta_{H}\ot H,
\end{equation}
\begin{equation}
\label{h-1mu-r}
\mu_{H}\co d^{-1}=H\ot \varepsilon_{H}.
\end{equation}

Also, composing with $\varepsilon_{H}\ot  H$ in (\ref{h-1delta}), 
\begin{equation}
\label{new-h-1delta}
 l_{H}\co \delta_{H}=\varepsilon_{H}\ot \eta_{H}.
\end{equation}

Composing with $H\otimes \eta_{H}$ in (\ref{leftdivision}), 
\begin{equation}
\label{new-h-2delta}
id_{H}\ast \lambda_{H}=\varepsilon_{H}\ot \eta_{H}
\end{equation}
for $\lambda_{H}=l_{H}\co (H\ot \eta_{H})$.  Similarly, composing with $\eta_{H}\ot H$ in (\ref{rightdivision}) we have 
\begin{equation}
\label{new-h-2delta-r}
\varrho_{H}\ast id_{H}=\varepsilon_{H}\ot \eta_{H}
\end{equation}
for $\varrho_{H}=r_{H}\co (\eta_{H}\ot H)$. 

On the other hand, by (\ref{eta-eps}), (\ref{delta-eta}), and (\ref{leftdivision})
\begin{equation}
\label{new-lh-eta}
 l_{H}\co (\eta_{H}\ot H)=id_{H}.
\end{equation}

Also, for right divisions we have 
\begin{equation}
\label{new-lh-eta-r}
 r_{H}\co (H\ot \eta_{H})=id_{H}.
\end{equation}

Finally, by (\ref{mu-eps}) and  (\ref{leftdivision})
\begin{equation}
\label{lh-vareps}
\varepsilon_{H}\co l_{H}=\varepsilon_{H}\ot \varepsilon_{H}.
\end{equation}

Therefore,  
\begin{equation}
\label{lambda-vareps}
\varepsilon_{H}\co \lambda_{H}=\varepsilon_{H},
\end{equation}
and 
\begin{equation}
\label{lambda-eta}
\lambda_{H}\co \eta_{H}=\eta_{H}.
\end{equation}

Of course, for a right division we have 
\begin{equation}
\label{lh-vareps-r}
\varepsilon_{H}\co r_{H}=\varepsilon_{H}\ot \varepsilon_{H},
\end{equation}
\begin{equation}
\label{lambda-vareps-r}
\varepsilon_{H}\co \varrho_{H}=\varepsilon_{H},
\end{equation}
\begin{equation}
\label{lambda-eta-r}
\varrho_{H}\co \eta_{H}=\eta_{H}.
\end{equation}
}
\end{remark}

The following result was proved in (\cite{PI07}, Proposition 6) for unital $H$-bialgebras. In this paper we give an alternative proof based in Proposition \ref{divisionunica}.

\begin{proposition}
\label{PI-adv-Prop 6}
Let $H$ be a non-associative bimonoid  with left division $l_{H}$. It holds that
\begin{equation}\label{alphadelta}
\delta_H\co l_{H}=(l_{H}\ot l_{H})\co(H\ot c_{H,H}\ot H)\co((c_{H,H}\co\delta_H)\ot\delta_H).
\end{equation}
As a consequence, if $\lambda_{H}=l_{H}\co (H\ot \eta_{H})$ we have  that $\lambda_{H}$ is anticomultiplicative, i.e.,
\begin{equation}\label{anticomul}
\delta_H\co \lambda_{H}=(\lambda_{H}\ot \lambda_{H})\co c_{H,H} \co \delta_H.
\end{equation}

If $r_{H}$ is a right division for $H$, the equality 
\begin{equation}\label{alphadelta-r}
\delta_H\co r_{H}=(r_{H}\ot r_{H})\co(H\ot c_{H,H}\ot H)\co(\delta_{H}\ot (c_{H,H}\co\delta_H))
\end{equation}
holds. Then, if $\varrho_{H}=r_{H}\co (\eta_{H}\ot H)$, we have that 
\begin{equation}\label{anticomul-r}
\delta_H\co \varrho_{H}=(\varrho_{H}\ot \varrho_{H})\co c_{H,H} \co \delta_H.
\end{equation}

\end{proposition}

\begin{proof} Indeed, if we compose in the first term of (\ref{alphadelta}) with the isomorphism $h=(H\ot \mu_H)\co (\delta_H\ot H)$, we obtain
$$\delta_{H}\co l_{H}\co h=\varepsilon_{H}\ot \delta_{H}$$
and, on the other hand, composing in the second term, 
\begin{itemize}
\item[ ]$\hspace{0.38cm}(l_{H}\ot l_{H})\co(H\ot c_{H,H}\ot H)\co((c_{H,H}\co\delta_H)\ot\delta_H)\co h $

\item[ ]$=(l_{H}\ot l_{H})\co(H\ot c_{H,H}\ot H)\co((c_{H,H}\co\delta_H)\ot((\mu_{H}\ot \mu_{H})\co \delta_{H\ot H}))\co (\delta_{H}\ot H) $
{\scriptsize ({\blue by (\ref{delta-mu})})}

\item[ ]$=((l_{H}\co h)\ot l_{H})\co(H\ot c_{H,H}\ot \mu_{H})\co (c_{H,H}\ot c_{H,H}\ot H)\co (((H\ot \delta_{H})\co \delta_{H})\ot \delta_{H})$ {\scriptsize ({\blue by  naturality of $c$ }}
\item[ ]\hspace{0.38cm} {\scriptsize {\blue  and  coassociativity})}

\item[ ]$=(H\ot l_{H})\co(c_{H,H}\ot \mu_{H})\co \delta_{H\ot H}$   {\scriptsize ({\blue by  naturality of $c$ and  properties of the counit})}

\item[ ]$=(H\ot (l_{H}\co h))\co(c_{H,H}\ot H)\co (H\ot \delta_{H})$   {\scriptsize ({\blue by  naturality of $c$})}

\item[ ]$=\varepsilon_{H}\ot \delta_{H}$ {\scriptsize ({\blue by  naturality of $c$})}

\end{itemize}

Therefore, (\ref{alphadelta})  holds. Finally, the equality (\ref{anticomul}) follows  by (\ref{alphadelta}) and  (\ref{delta-eta}). Similarly we can prove the identities (\ref{alphadelta-r}) and (\ref{anticomul-r}).

\end{proof}

\begin{example}
\label{exsabinin}
{\rm
An essential example of a non-associative bimonoid arises from Sabinin algebras. We take one of its possible definitions  directly from \cite{PI07}. A vector space $V$ over a field of characteristic zero is called a Sabinin algebra if it is endowed with multilinear operations
$$
\langle  x_1,x_2,\ldots, x_m ; y, z      \rangle, \hspace{0.1cm} m\geq 0,
$$
$$
\Phi(x_1,x_2,\ldots, x_m ; y_1, y_2,\ldots, y_n), \hspace{0.1cm} m\geq 1, n\geq 2,
$$
which satisfy the identities
$$
\langle  x_1,x_2,\ldots, x_m ; y, z      \rangle=
-\langle x_1,x_2,\ldots, x_m ; z, y      \rangle,
$$
\begin{itemize}
\item[]
$\langle x_1,x_2,\ldots,x_r, a, b, x_{r+1},\ldots, x_m ; y, z      \rangle
-
\langle x_1,x_2,\ldots,x_r, b, a, x_{r+1},\ldots, x_m ; y, z      \rangle$\\
\item[]
$
+
\sum_{k=0}^r \sum_{\omega}  \langle x_{\omega_1},\ldots, x_{\omega_k}; \langle x_{\omega_{k+1}}, \ldots x_{\omega_r}; a, b  \rangle,\ldots,x_m; y, z        \rangle=0,
$
\end{itemize}
$$
\sigma_{x,y,z}
\langle x_1,x_2,\ldots,x_r, x; y,z\rangle
+
\sum_{k=0}^r \sum_{\omega}  \langle x_{\omega_1},\ldots, x_{\omega_k}; \langle x_{\omega_{k+1}}, \ldots x_{\omega_r}; y, z  \rangle, x\rangle =0
$$
and
$$
\Phi( x_1,\ldots, x_m; y_1,\ldots, y_n)=\Phi (x_{\tau(1)},\ldots,x_{\tau(m)}; y_{\upsilon(1)},\ldots,y_{\upsilon(n)} ),
$$
where $\omega$ runs the set of all bijections of the type $\omega:\{1,2,\ldots,r\}\rightarrow \{1,2,\ldots,r\}$, $i\mapsto \omega_i$, $\omega_1<\omega_2<\ldots<\omega_k$, $\omega_{k+1}<\ldots <\omega_r$, $k=0,1,\ldots, r$, $r\geq0$, $\sigma_{x,y,z}$ denotes the cyclic sum by $x,y,z$; $\tau\in S_m$, $\upsilon\in S_n$ and $S_l$ is the symmetric group.

Given any Sabinin algebra $V$, it was explicitly constructed in \cite{PI07}  its universal enveloping algebra $U(V)$, and moreover, it was also proved that it can be provided with a cocommutative non-associative bimonoid structure with left and right division.

This is an interesting example because, as it is pointed out in \cite{PI07} and \cite{MPIS14}, the definition of Sabinin algebra includes an infinite set of independent operations, but when we take a finite set we recover many other common structures. For example, it includes, as  particular instances with $\Phi=0$, Lie, Malcev and Bol algebras. Specifically, Bol algebras have one binary and one ternary operation, Malcev algebras one binary operation, and Lie algebras are just a further particularization of Malcev algebras.

}

\end{example}

Now we introduce the notion of  left Hopf quasigroup.

\begin{definition}
\label{leftHopfqg} {\rm A left Hopf quasigroup $H$ in
${\mathcal C}$ is a non-associative bimonoid such that there exists a morphism $\lambda_{H}:H\rightarrow H$ in ${\mathcal C}$ (called the left antipode of $H$) satisfying:
\begin{equation}
\label{leftHqg}
\mu_H\circ (\lambda_H\ot \mu_H)\circ (\delta_H\ot H)=
\varepsilon_H\ot H=
\mu_H\circ (H\ot \mu_H)\circ (H\ot \lambda_H\ot H)\circ (\delta_H\ot H).
\end{equation}

Note that composing with $H\ot \eta_{H}$ in (\ref{leftHqg}) we obtain 
\begin{equation}
\label{primeradealpha}
\lambda_{H}\ast id_{H}=\varepsilon_H\ot \eta_H. 
\end{equation}

Obviously, there is a similar definition for the right side, i.e., $H$ is a right Hopf quasigroup if  there is a morphism $\varrho_H:H\rightarrow H$
in ${\mathcal C}$ (called the right antipode of $H$) such that
\begin{equation}
\label{rightHqg}
\mu_H\circ (\mu_H\ot H)\circ (H\ot \varrho_H\ot H)\circ (H\ot \delta_H)=
H\ot \varepsilon_H=
\mu_H\circ(\mu_H\ot \varrho_H)\circ (H\ot \delta_H). 
\end{equation}

Then, composing with $ \eta_{H}\ot H$ in (\ref{rightHqg}) we obtain 
\begin{equation}
\label{primeradealpha2}
id_{H}\ast \varrho_{H}=\varepsilon_H\ot \eta_H. 
\end{equation}

 }
\end{definition}

The above definition is a generalization of the notion of Hopf quasigroup (also called non-associative Hopf algebra with the inverse property, or non-associative IP Hopf algebra) introduced in \cite{Majidesfera} (in this case ${\mathcal C}$ is the category of vector spaces over a field ${\Bbb F}$). We recall this definition in a monoidal setting (see \cite{our1}, \cite{our2}).  Note that a  Hopf quasigroup is associative if an only if  is a Hopf algebra.

\begin{definition}
\label{Hqg}  A Hopf quasigroup $H$ in
${\mathcal C}$ is a non-associative bimonoid such that there exists a morphism $\lambda_{H}:H\rightarrow H$
in ${\mathcal C}$ (called the antipode of $H$) satisfying (\ref{leftHqg}) and (\ref{rightHqg}). If $H$ is a Hopf quasigroup in ${\mathcal C}$, the antipode
$\lambda_{H}$ is unique and antimultiplicative, i.e.,
\begin{equation} 
\label{lambda-anti}
\lambda_{H}\co\mu_{H}=\mu_{H}\circ
(\lambda_{H}\ot\lambda_{H})\circ
c_{H,H},
\end{equation}
(\cite{Majidesfera}, Proposition 4.2). A morphism between Hopf quasigroups $H$ and $B$ is a morphism $f:H\rightarrow B$ of unital magmas and comonoids. Then (see Lemma 1.4 of \cite{our1}) the equality
\begin{equation}
\label{antipode-morphism}
\lambda_B\co f=f\co \lambda_H,
\end{equation}
holds.

\end{definition}

\begin{remark}\label{antipodounico}
{\rm Note that if $H$ is both left and right Hopf quasigroup, the left and right antipodes are the same. Indeed, denote by $\lambda_H$ and $\varrho_H$ the left and right antipodes, respectively. Then, taking into account (\ref{primeradealpha}), the coassociativity of $\delta_{H}$ and condition (\ref{rightHqg}), 
$$\varrho_H= (\lambda_H\ast id_H)\ast \varrho_H=\mu_H\co (\mu_H\ot \varrho_H)\co (\lambda_H\ot \delta_H)\co \delta_H=\lambda_H.$$

As a consequence, $H$ is a Hopf quasigroup if and only if $H$ is a left and right Hopf quasigroup. 
}
\end{remark}

\begin{example}
\label{exloop}
{\rm  A loop $(L, \cdot,\diagup, \diagdown)$ is a set $L$
equipped with three binary operations of multiplication $\cdot$, right division $\diagup$ and
left division $\diagdown$, satisfying the identities:
\begin{equation}\label{IL}
 v\diagdown(v \cdot u) = u,
 \end{equation}
\begin{equation}\label{IR}
 u = (u \cdot v)\diagup v,
 \end{equation}
\begin{equation}\label{SL}
v\cdot  (v\diagdown u) = u,
 \end{equation}
\begin{equation}\label{SR}
u = (u\diagup v) \cdot  v,
 \end{equation}
and such that in addition it contains an  identity element $e_L$  satisfying that the equations $e_L\cdot x=x=x\cdot e_L$ hold for all $x$ in $L$. From now on, for brevity of notations, multiplications on $L$ will be expressed by juxtaposition.

Let $L$ be a loop and let $N$ be a non-empty subset of $L$.
Then we say that $N$ is a subloop of $L$ if it is closed under the three binary operations. Notice that in this case $e_L=e_N$.

A map
$h:L_{1}\rightarrow L_{2}$ is called a loop homomorphism if
$h(uv)=h(u)h(v)$, $h(u\diagup v)=h(u)\diagup h(v)$ and $h(u\diagdown v)=h(u)\diagdown h(v)$ for all $u,v\in L_{1}$.  A  bijective loop homomorphism  is called a
 loop isomorphism. It is a well known fact that, if $h:L_1\to L_2$ is a loop homomorphism we have $h(e_{L_{1}})=e_{L_{2}}$.

Let $R$ be a commutative ring and $L$ a loop. Then, the loop algebra 
$$RL=\bigoplus_{u\in L}Ru$$
is a cocommutative non-associative bimonoid with  product left and right division defined by
linear extensions of those defined in $L$ and
$$\delta_{RL}(u)=u\ot u, \;\; \varepsilon_{RL}(u)=1_{R}$$
 on the basis elements (see \cite{PI07}).

}
\end{example}

Now we give the relation between non-associative bimonoids with left division and left Hopf quasigroups.
\begin{proposition}
\label{igualdadesalpha}
The following assertions are equivalent:
\begin{itemize}
\item [(i)] $H$ is a non-associative bimonoid with left division $l_{H}$ such that
\begin{equation}
\label{etaalphamu}
l_{H}=\mu_H\co (\lambda_{H}\ot H), 
\end{equation}
where $\lambda_{H}=l_{H}\co (H\ot \eta_{H})$.
\item [(ii)] $H$ is a non-associative bimonoid with left division $l_{H}$ such that
\begin{equation}
\label{quasiasociativity}
\mu_H\co (\lambda_{H}\ot \mu_H)\co (\delta_H\ot H)=\varepsilon_H\ot H, 
\end{equation}
where $\lambda_{H}=l_{H}\co (H\ot \eta_{H})$.

\item [(iii)] $H$ is a left Hopf quasigroup.

\end{itemize}

\end{proposition}

\begin{proof} By (\ref{leftdivision}), (i) implies (ii). Moreover, composing in (\ref{quasiasociativity}) with  $(H\ot l_{H})\co (\delta_H\ot H)$ and using coassociativity, we get that (ii) implies (i). Now assume (i). Then, by (\ref{etaalphamu}) and (\ref{leftdivision}),
$$\mu_H\co (\lambda_H\ot \mu_H)\co (\delta_H\ot H)=l_{H}\co (H\ot \mu_H)\co (\delta_H\ot H)=\varepsilon_H\ot H,$$
and in a similar way $\mu_H\co (H\ot \mu_H)\co (H\ot \lambda_H\ot H)\co (\delta_H\ot H)=\varepsilon_H\ot H$. Finally, if $H$ is a left Hopf quasigroup, the morphism $l_{H}=\mu_H\co (\lambda_H\ot H)$ is a left division and satisfies (\ref{etaalphamu}).

\end{proof}

The relation between non-associative bimonoids with right division and right  Hopf quasigroups is the following (the proof is similar to the one used for left divisions):

\begin{proposition}
\label{igualdadesalpha-r}
The following assertions are equivalent:
\begin{itemize}
\item [(i)] $H$ is a non-associative bimonoid with right division $r_{H}$ such that
\begin{equation}
\label{etaalphamu-r}
r_{H}=\mu_H\co (H\ot \varrho_{H}), 
\end{equation}
where $\varrho_{H}=r_{H}\co (\eta_{H}\ot H)$.
\item [(ii)] $H$ is a non-associative bimonoid with right division $r_{H}$ such that
\begin{equation}
\label{quasiasociativity-r}
\mu_H\co (\mu_H\ot \varrho_{H})\co (H\ot \delta_H)=H\ot \varepsilon_H, 
\end{equation}
where $\varrho_{H}=r_{H}\co (\eta_{H}\ot H)$.

\item [(iii)] $H$ is a right Hopf quasigroup.

\end{itemize}

\end{proposition}

\begin{example}\label{IPloops}
{\rm

A loop $L$ is said to be a loop with the
inverse property  (for brevity an IP loop) if 
to every element $u\in L$, there corresponds an element $u^{-1}\in
L$ such that the equations
\begin{equation}
\label{IP-loop-1} u^{-1}(uv)=v=(vu)u^{-1},
\end{equation}
hold for  every $v\in L$.

If $L$ is an IP loop, it is easy to show that
for all $u\in L$ the element $u^{-1}$ is unique and
\begin{equation}
\label{IP-loop-2} u^{-1}u=e_{L}=uu^{-1}.
\end{equation}
Moreover, the mapping $u\rightarrow u^{-1}$ is an anti-automorphism
of the IP loop $L$:
\begin{equation}
\label{IP-loop-3} (uv)^{-1}=v^{-1}u^{-1}.
\end{equation}
Now let $R$ be a commutative ring and let $L$ be and IP loop. Then, by
Proposition 4.7 of \cite{Majidesfera}, the non-associative bimonoid 
$$RL=\bigoplus_{u\in L}Ru,$$
defined in Example \ref{exloop}, is a cocommutative Hopf quasigroup 
with $\lambda_{RL}(u)=u^{-1}$.

Relevant examples of IP-loops are provided by Moufang loops, closely related with groups with triality. With respect to its linearization, loop algebras of Moufang loops correspond to Moufang-Hopf algebras, a fact that can be interpreted as the correspondence between groups with triality and Hopf algebras with triality (see \cite{BMP-I12}).
}
\end{example}

\begin{example}\label{envelopeMalcev}
{\rm
Consider a commutative ring $R$ with $\frac{1}{2}$ and $\frac{1}{3}$ in $R$. A Malcev algebra  $(M,[,])$ over $R$ is a free module over $R$ with a bilinear anticommutative
operation [ , ] on $M$ satisfying that:
\[[J(a, b, c), a] = J(a, b, [a, c]),\]
where $J(a, b, c) = [[a, b], c] - [[a, c], b] - [a, [b, c]]$ is the Jacobian in $a$, $b$, $c$ (see \cite{PIS}). Notice that every Lie algebra is a Malcev algebra with $J=0$. As a particularization of the construction alluded in Example \ref{exsabinin}, the universal enveloping algebra $U(M)$ can be provided with a Hopf quasigroup structure.

}
\end{example}

\begin{remark}\label{naHxeraliza}
{\rm
Notice that a Hopf quasigroup is a particular instance of a non-associative bimonoid with left and right division. It suffices to take
\[
 l_{H}:=\mu_H\co(\lambda_H\ot H),
 \]
 \[
 r_{H}:=\mu_H\co(H\ot \lambda_H).
 \]

But the notion of a non-associative bimonoid is wider. For example, the loop algebra $RL$ of a  loop $L$ and the universal algebra $U(V)$ of a Sabinin algebra $V$
falls under its definition (see \cite{MPIS14}, \cite{PI07}).

}
\end{remark}

\section{Product alterations by two cocycles for non-associative bimonoids}

In this section we prove that  two-cocycles provide a deformation way of altering the product of a non-associative bimonoid to produce other non-associative bimonoid. These kind of cocycle deformations  were introduced in the Hopf algebra setting by Doi in \cite{Doi1}.

\begin{definition}
\label{defcociclo}

{\rm Let $H$ be a non-associative bimonoid, and let $\sigma:H\ot H\rightarrow K$ be a convolution invertible morphism. We say that $\sigma$ is a two-cocycle 
if the equality
\begin{equation}
\label{two-cocycle}
\partial^1(\sigma)\ast \partial^3(\sigma)=\partial^4(\sigma)\ast\partial^2(\sigma)
\end{equation}
holds, where $\partial^1(\sigma)=\varepsilon_H\ot \sigma$, $\partial^2(\sigma)=\sigma\co (\mu_H\ot H)$, $\partial^3(\sigma)=\sigma\co (H\ot\mu_H)$ and
 $\partial^4(\sigma)=\sigma\ot \varepsilon_H$. Equivalently, $\sigma$ is a two-cocycle  if 
 \begin{equation}
\label{two-cocycle-1}
\sigma\co (H\ot ((\sigma\ot \mu_{H})\co \delta_{H\ot H}))=\sigma\co (((\sigma\ot \mu_{H})\co \delta_{H\ot H})\ot H).
\end{equation}

Note that, if we compose in (\ref{two-cocycle-1}) with $\eta_{H}\ot 	\eta_{H}\ot H$ we obtain 
\begin{equation}
\label{2n-cocycle-1}
((\sigma\co (\eta_{H}\ot H))\ot (\sigma\co (\eta_{H}\ot H)))\co \delta_{H}=(\sigma\co (\eta_{H}\ot \eta_{H}))\ot (\sigma\co (\eta_{H}\ot H)),
\end{equation} 
and, if we compose with $H\ot \eta_{H}\ot 	\eta_{H}$, we get that, 
\begin{equation}
\label{2n-cocycle-2}
((\sigma\co (H\ot \eta_{H}))\ot (\sigma\co (H\ot \eta_{H})))\co \delta_{H}=(\sigma\co (H\ot \eta_{H}))\ot (\sigma\co (\eta_{H}\ot \eta_{H})).
\end{equation}

The a two-cocycle $\sigma$ is called normal if further
\begin{equation}
\label{normaltwo-cocycle}
\sigma\co (\eta_H\ot H)=\varepsilon_H=\sigma\co (H\ot \varepsilon_H),
\end{equation}
and it is easy to see that  if $\sigma$ is normal so is $\sigma^{-1}$ because
$$\sigma^{-1}\co (\eta_{H}\ot H)=\varepsilon_{H}\ast (\sigma^{-1}\co (\eta_{H}\ot H))=
 (\sigma\co (\eta_{H}\ot H))\ast  (\sigma^{-1}\co (\eta_{H}\ot H))=(\sigma\ast \sigma^{-1})\co (\eta_{H}\ot H)=\varepsilon_{H},$$
 and similarly $\sigma^{-1}\co (H\ot \eta_{H})=\varepsilon_{H}$. Analogously,  if $\sigma^{-1}$ is normal so is $\sigma$.
}
\end{definition}

\begin{remark}\label{formulas2cociclos}
{\rm It is not difficult to show that, if $\sigma$ is a two-cocycle,  $\tau =(\sigma^{-1}\co (\eta_H\ot \eta_H))\ot \sigma$ is a normal two-cocycle (see \cite{Take}). The inverse of $\tau$ is $\tau^{-1}=(\sigma\co (\eta_H\ot \eta_H))\ot \sigma^{-1}$ and the normal condition for $\tau$ follows from the identities $(\tau\co ( \eta_{H}\ot H))\ast (\tau\co ( \eta_{H}\ot H))=\tau\co ( \eta_{H}\ot H)$ and $(\tau\co (H\ot \eta_{H}))\ast (\tau\co (H\ot \eta_{H}))=\tau\co (H\ot \eta_{H})$ (this identities are consequence of (\ref{2n-cocycle-1}) and (\ref{2n-cocycle-2}) respectively). As a consequence, in the following we assume all two-cocycles are normal.  

On the other hand, the morphisms $\partial^i(\sigma)$, $i\in \{1,2,3,4\}$ are convolution invertible with inverses $\partial^i(\sigma^{-1})$, $i\in \{1,2,3,4\}$, respectively. Then the equalities
\begin{equation}
\label{3inversacon2}
\partial^3(\sigma)\ast \partial^2(\sigma^{-1})=\partial^1(\sigma^{-1})\ast\partial^4(\sigma),
\end{equation}
\begin{equation}
\label{inversa4con1}
\partial^4(\sigma^{-1})\ast \partial^1(\sigma)=\partial^2(\sigma)\ast\partial^3(\sigma^{-1}),
\end{equation}
\begin{equation}
\label{inversa3coninversa1}
\partial^3(\sigma^{-1})\ast \partial^1(\sigma^{-1})=\partial^2(\sigma^{-1})\ast\partial^4(\sigma^{-1}),
\end{equation}
hold. Moreover (\ref{3inversacon2}), (\ref{inversa4con1}) and (\ref{inversa3coninversa1}) are equivalent to 
\begin{equation}
\label{3inversacon2-new}
(\sigma\ot \sigma^{-1})\co (H\ot \mu_{H}\ot H)\co (H\ot c_{H,H}\ot c_{H,H}\ot H)\co (\delta_{H}\ot \delta_{H}\ot \delta_{H})=(\sigma\ot \sigma^{-1})\co (H\ot (c_{H,H}\co \delta_{H})\ot H),
\end{equation}
\begin{equation}
\label{inversa4con1-new}
((\sigma\co (\mu_{H}\ot H))\ot (\sigma^{-1}\co (H\ot \mu_{H})))\co  (H\ot c_{H,H}\ot c_{H,H}\ot H)\co (\delta_{H}\ot \delta_{H}\ot \delta_{H})=(\sigma^{-1}\ot \sigma)\co (H\ot \delta_{H}\ot H),
\end{equation}
and 
\begin{equation}
\label{inversa3coninversa1-new}
\sigma^{-1}\co (H\ot ((\mu_{H}\ot \sigma^{-1})\co \delta_{H\ot H}))=\sigma^{-1}\co (((\mu_{H}\ot \sigma^{-1})\co \delta_{H\ot H})\ot H),
\end{equation}
respectively.
}
\end{remark}

\begin{proposition}
\label{algebradeformada}
Let $H$ be a non-associative bimonoid. Let $\sigma$ be a two-cocycle. Define the product $\mu_{H^{\sigma}}$ as
$$\mu_{H^{\sigma}}=(\sigma\ot\mu_H\ot \sigma^{-1})\co (H\ot H\ot \delta_{H\ot H})\co \delta_{H\ot H}.$$
Then  $H^{\sigma}=(H, \eta_{H^{\sigma}}=\eta_H, \mu_{H^{\sigma}}, \varepsilon_{H^{\sigma}}=\varepsilon_H, \delta_{H^{\sigma}}=\delta_H)$ is a non-associative bimonoid. 

\end{proposition}

\begin{proof} The equalities (\ref{eta-eps}) and (\ref{delta-eta}) hold trivially. Using that  $H$ is a non-associative bimonoid and (\ref{normaltwo-cocycle}), we get that $\mu_{H^{\sigma}}\co (\eta_H\ot H)=id_H=\mu_{H^{\sigma}}\co (H\ot \eta_H)$. Moreover, by (\ref{mu-eps}), 
$$\varepsilon_H\co \mu_{H^{\sigma}}=\sigma\ast\sigma^{-1}=\varepsilon_H\ot \varepsilon_H.$$
Finally, by the naturality of $c$, the coassociativity of $\delta_{H}$ and the properties of the counit 
$$(\mu_{H^{\sigma}}\ot \mu_{H^{\sigma}})\co \delta_{H\ot H}$$
$$=(((\sigma\ot \mu_{H})\co \delta_{H\ot H})\ot (\sigma^{-1}\ast \sigma)\ot ((\mu_{H}\ot \sigma^{-1})\co \delta_{H\ot H}))\co (H\ot ((c_{H,H}\ot c_{H,H})\co \delta_{H\ot H})\ot H)\co (\delta_{H}\ot \delta_{H})$$
$$= \delta_H\co \mu_{H^{\sigma}}.$$

\end{proof}

\begin{proposition}
\label{morfismosauxiliares}
Let $H$ be a non-associative bimonoid with left division $l_{H}$, put $\lambda_{H}=l_{H}\co (H\ot \eta_{H})$, and let $\sigma$ be a two-cocycle. Define the morphism $f:H\rightarrow K$ as $f=\sigma\co ( H\ot \lambda_{H})\co \delta_H$. If the equality (\ref{primeradealpha}) holds, then $f$ is convolution invertible with inverse $f^{-1}=\sigma^{-1}\co (\lambda_{H}\ot H)\co \delta_H$. Moreover, the following identities hold:
\begin{equation}
\label{f-eta}
f\co \eta_{H}=f^{-1}\co \eta_{H}=id_{K}.
\end{equation}

If $r_{H}$ is a right division for $H$, put $\varrho_{H}=r_{H}\co (\eta_{H}\ot H)$. Let $\sigma$ be a two-cocycle. Define the morphism $g:H\rightarrow K$ as $g= \sigma^{-1}\co (\varrho_{H}\ot H)\co \delta_H$. If the equality (\ref{primeradealpha2}) holds, then $g$ is convolution invertible with inverse $g^{-1}=\sigma\co (H\ot \varrho_{H})\co \delta_H$. Moreover, 
\begin{equation}
\label{f-eta-r}
g\co \eta_{H}=g^{-1}\co \eta_{H}=id_{K}.
\end{equation}

\end{proposition}

\begin{proof} Indeed,

\begin{itemize}
\item[ ]$\hspace{0.38cm}f\ast f^{-1} $

\item[ ]$=(\sigma\ot \sigma^{-1})\co (H\ot (c_{H,H}\co (\lambda_{H}\ot \lambda_{H})\co  c_{H,H}\co \delta_H)\ot H)\co (\delta_H\ot H)\co \delta_H $
\item[ ]$\hspace{0.38cm}$ {\scriptsize ({\blue by (\ref{delta-eta}) and  naturality of $c$})}

\item[ ]$=(\sigma\ot \sigma^{-1})\co (H\ot (c_{H,H}\co \delta_{H}\co \lambda_{H})\ot H)\co (\delta_H\ot H)\co \delta_H $
{\scriptsize ({\blue by (\ref{anticomul})})}

\item[ ]$=(\partial^1(\sigma^{-1})\ast\partial^4(\sigma))\co  (((H\ot \lambda_{H})\co \delta_H)\ot H)\co \delta_H$
{\scriptsize ({\blue by  (\ref{alphadelta}), naturality of $c$, and  counit properties)}}

\item[ ]$=(\partial^3(\sigma)\ast \partial^2(\sigma^{-1}))\co (((H\ot \lambda_{H})\co \delta_H)\ot H)\co \delta_H$
{\scriptsize ({\blue by (\ref{3inversacon2}}))}

\item[ ]$=(\sigma\ot \sigma^{-1})\co (H\ot \mu_{H\ot H}\ot H)\co (H\ot c_{H,H}\ot c_{H,H}\ot H)\co (((\delta_{H}\ot ((\lambda_{H}\ot \lambda_{H})\co c_{H,H}\co \delta_{H}))\co \delta_{H})\ot\delta_{H})$
\item[ ]$\hspace{0.38cm}\co \delta_{H}$
{\scriptsize  ({\blue by (\ref{delta-eta}), (\ref{alphadelta}) and naturality of $c$})}

\item[ ]$=(\sigma\ot \sigma^{-1})\co (H\ot ((\mu_{H}\ot (id_{H}\ast \lambda_{H}))\co (\lambda_{H}\ot c_{H,H})\co (\delta_{H}\ot H)\co \delta_{H})\ot H)\co (\delta_{H}\ot H)\co \delta_{H}$ {\scriptsize  ({\blue by}}
\item[ ]$\hspace{0.38cm}$ {\scriptsize  {\blue  naturality of $c$, and  coassociativity})}

\item[ ]$=(((\sigma\co (H\ot  (\lambda_{H}\ast id_{H}))\co \delta_{H}))\ot (\sigma^{-1}\co (\eta_{H}\ot H)))\co \delta_{H}$  {\scriptsize  ({\blue by (\ref{new-h-2delta}),  and  counit  properties)})}

\item[ ]$= \sigma\co (H\ot \eta_{H})$ {\scriptsize  ({\blue  by (\ref{primeradealpha}), the normal condition for $\sigma^{-1}$ and counit properties})}

\item[ ]$=\varepsilon_H$
{\scriptsize ({\blue by (\ref{normaltwo-cocycle})}).}

\end{itemize}

On the other hand,

\begin{itemize}
\item[ ]$\hspace{0.38cm}f^{-1}\ast f $

\item[ ]$=(\partial^4(\sigma^{-1})\ast \partial^1(\sigma))\co (\lambda_{H}\ot H\ot \lambda_{H})\co (\delta_H\ot H)\co \delta_H$
{\scriptsize ({\blue by coassociativity, naturality of $c$, and counit }}\item[ ]$\hspace{0.38cm}$ {\scriptsize {\blue properties})}

\item[ ]$=(\partial^2(\sigma)\ast\partial^3(\sigma^{-1}))\co (\lambda_{H}\ot H\ot \lambda_{H})\co (\delta_H\ot H)\co \delta_H$
{\scriptsize ({\blue by (\ref{inversa4con1})})}

\item[ ]$=(\sigma\ot \sigma^{-1})\co (\mu_{H}\ot c_{H,H}\ot\mu_H)\co (H\ot c_{H,H}\ot c_{H,H}\ot H)$
\item[ ]$\hspace{0.38cm}  
\co (((((\lambda_{H}\ot \lambda_{H})\co  c_{H,H}\co \delta_H)\ot \delta_{H})\co \delta_{H})\ot (((\lambda_{H}\ot \lambda_{H})\co  c_{H,H}\co \delta_H))) \co \delta_{H}$
 {\scriptsize ({\blue by (\ref{delta-eta}) and (\ref{anticomul})})}

\item[ ]$=\sigma^{-1}\co (H\ot \sigma\ot (id_{H}\ast \lambda_{H}))\co(((\lambda_{H}\ot (\lambda_{H}\ast id_{H}))\co \delta_{H})\ot ((c_{H,H}\co (H\ot \lambda_{H})\co \delta_{H})))\co \delta_{H}$ {\scriptsize ({\blue by }}
\item[ ]$\hspace{0.38cm} ${\scriptsize {\blue coassociativity and naturality of $c$})}

\item[ ]$= \varepsilon_{H}$  {\scriptsize ({\blue by (\ref{primeradealpha}), (\ref{new-h-2delta}), the normal condition for $\sigma$ and $\sigma^{-1}$, naturality of $c$, counit properties, and (\ref{lambda-vareps})})}.

\end{itemize}

 Finally,  (\ref{f-eta}) follows from (\ref{delta-eta}), (\ref{new-lh-eta}), the normal condition for $\sigma$ and $\sigma^{-1}$, and from (\ref{eta-eps}). The proof for the right division is similar and we leave the details to the reader.

\end{proof}

\begin{remark}\label{cond-1}
{\rm Note that the equalities (\ref{primeradealpha}) and (\ref{primeradealpha2}) hold for every left Hopf quasigroup. Also, they hold for loop algebras associated to right or left Bol loops.  The so-called right Bol identity was introduced by  G. Bol  in \cite{Bol} and was also mentioned by Bruck in \cite{Bruck}. Let  $(L, \cdot,\diagup, \diagdown)$  be a loop. $L$ is called a right Bol loop if the right Bol identity 
\begin{equation}
\label{bol}
((x\cdot y)\cdot z)\cdot y=x\cdot ((y\cdot z)\cdot y)
\end{equation}
holds  for all $x,y,z\in L$. If the equality (left Bol identity) 
\begin{equation}
\label{left-bol}
y\cdot (z\cdot (y\cdot x))=(y\cdot (z\cdot y))\cdot x
\end{equation}
holds for all $x,y,z\in L$, we say that L is a left Bol loop. As was pointed in \cite{Rob}, Bol loops are more general than Moufang loops because $L$ is Moufang if and only if it satisfies (\ref{bol}) and (\ref{left-bol}). Also, Bol loops with the automorphic inverse property are Bruck loops. 

An interesting example of right Bol loops comes from matrix theory. The set of $n\times n$  positive definite symmetric matrices is a right Bol loop with the operation 
$$P\cdot Q=\sqrt{QP^{2}Q}.$$

Moreover, in the literature we can find other examples of right Bol loops obtained by modifying the operation in a direct product of groups. 

The cocommutative non-associative bimonoid $RL$ defined in Example \ref{exloop} satisfies the equality (\ref{primeradealpha}) if and only if the loop $L$ satisfies 
\begin{equation}
\label{Bol1-equ}
(a\diagdown e_{L})\cdot a=e_{L}.
\end{equation}

If $L$ is a right Bol loop the equality (\ref{Bol1-equ}) always holds. Indeed, first note that 
$$((a\cdot (a\diagdown e_{L}))\cdot a)\cdot (a\diagdown e_{L})=(e_{L}\cdot a)\cdot (a\diagdown e_{L})=a\cdot (a\diagdown e_{L})=e_{L}.$$

Then 
$$a\cdot (((a\diagdown e_{L})\cdot a)\cdot (a\diagdown e_{L}))=e_{L}=a\cdot (a\diagdown e_{L}).$$

As a consequence, 
$$((a\diagdown e_{L})\cdot a)\cdot (a\diagdown e_{L})=a\diagdown e_{L}=e_{L}\cdot (a\diagdown e_{L}).$$

Therefore, (\ref{Bol1-equ}) holds. In a similar way, it is easy to see (\ref{Bol1-equ}) for a left Bol loop.
}
\end{remark}

\begin{proposition}
\label{alphasigmacondelta}
Let $H$ be a non-associative bimonoid with left division $l_{H}$, put $\lambda_{H}=l_{H}\co (H\ot \eta_{H})$ and assume that (\ref{primeradealpha}) holds. Let $\sigma$ be a two-cocycle and let $f$, $f^{-1}$ be the morphisms introduced in Proposition (\ref{morfismosauxiliares}).  Define the morphism $l_{H^{\sigma}}:H\ot H\rightarrow H$ as 
$$l_{H^{\sigma}}=\mu_{H^{\sigma}}\co (f\ot \lambda_{H}\ot f^{-1}\ot H)\co (H\ot \delta_H\ot H)\co (\delta_H\ot H).$$

Then the equality
\begin{equation}
\label{deltaalphasigma}
\delta_H\co l_{H^{\sigma}}=(l_{H^{\sigma}}\ot l_{H^{\sigma}})\co (H\ot c_{H,H}\ot H)\co((c_{H,H}\co \delta_H)\ot \delta_H)
\end{equation}
holds. Moreover, 
\begin{equation}
\label{lh-sigma-eta-1}
l_{H^{\sigma}}\co (\eta_{H}\ot H)=id_{H}, 
\end{equation}
and 
\begin{equation}
\label{lh-sigma-eta-2}
l_{H^{\sigma}}\co (H\ot \eta_{H})=(f\ot \lambda_{H}\ot f^{-1})\co (\delta_{H}\ot H)\co \delta_{H}. 
\end{equation}

Therefore, we have 
\begin{equation}
\label{lh-sigma-eta-2-1}
l_{H^{\sigma}}=\mu_{H^{\sigma}}\co ((l_{H^{\sigma}}\co (H\ot \eta_{H}))\ot H),  
\end{equation}
\begin{equation}
\label{lh-sigma-eta-2-2}
(f^{-1}\ot l_{H^{\sigma}})\co (\delta_{H}\ot H)=\mu_{H^{\sigma}}\co (((\lambda_{H}\ot f^{-1})\co \delta_{H})\ot H). 
\end{equation}

Finally, if $H$ is cocommutative,  
\begin{equation}
\label{lh-sigma-eta-3}
l_{H^{\sigma}}\co (H\ot \eta_{H})=\lambda_{H}. 
\end{equation}
\end{proposition}

\begin{proof}
Indeed, the equality (\ref{deltaalphasigma}) holds because:
\begin{itemize}
\item[ ]$\hspace{0.38cm} (l_{H^{\sigma}}\ot l_{H^{\sigma}})\co (H\ot c_{H,H}\ot H)\co((c_{H,H}\co \delta_H)\ot \delta_H)$

\item[ ]$=\mu_{H^{\sigma}\ot H^{\sigma}}\co (((f\ot \lambda_{H}\ot \lambda_{H}\ot f^{-1})\co \delta_{H\ot H}\co (H\ot (f\ast f^{-1})\ot H)\co (\delta_{H}\ot H)\co\delta_{H})\ot \delta_{H})$ {\scriptsize ({\blue by }}
\item[ ]$\hspace{0.38cm} ${\scriptsize {\blue coassociativity  and naturality of $c$})}

\item[ ]$=\mu_{H^{\sigma}\ot H^{\sigma}}\co (((f\ot ((\lambda_{H}\ot \lambda_{H})\co 
 c_{H,H} \co  \delta_{H})\ot f^{-1}) \co (H\ot \delta_{H})\co \delta_{H})\ot \delta_{H})$ {\scriptsize ({\blue by the invertibility of $f$,}}
\item[ ]$\hspace{0.38cm}${\scriptsize {\blue coassociativity and counit properties})}

\item[ ]$=(\mu_{H^{\sigma}}\ot \mu_{H^{\sigma}})\co \delta_{H\ot H}\co (f\ot\lambda_{H}\ot f^{-1}\ot H)\co (H\ot \delta_H\ot H)\co (\delta_H\ot H)$
{\scriptsize ({\blue by (\ref{anticomul})})}

\item[ ]$=\delta_H\co l_{H^{\sigma}}$
{\scriptsize ({\blue by (\ref{delta-mu}) for $H^{\sigma}$}).}
\end{itemize}

The identity (\ref{lh-sigma-eta-1}) follows trivially because $\eta_{H}$ is the unit of $H^{\sigma}$ and by (\ref{delta-eta}), (\ref{f-eta}) and (\ref{new-lh-eta}). Also, using that $\eta_{H}$ is the unit of $H^{\sigma}$ we obtain (\ref{lh-sigma-eta-2}). The equality (\ref{lh-sigma-eta-2-1}) folllows directly from (\ref{lh-sigma-eta-2}), and (\ref{lh-sigma-eta-2-2}) is a consequence of the coassociativity of $\delta_{H}$, the invertibility of $f$ and the counit properties.

Finally, if $H$ is cocommutative,

\begin{itemize}
\item[ ]$\hspace{0.38cm}l_{H^{\sigma}}\co (H\ot \eta_{H}) $

\item[ ]$=(f\ot \lambda_{H}\ot f^{-1})\co (\delta_{H}\ot H)\co \delta_{H}$ 
{\scriptsize ({\blue by (\ref{lh-sigma-eta-2})})}

\item[ ]$=(f\ot f^{-1}\ot \lambda_{H})\co (H\ot (c_{H,H}\co \delta_{H}))\co \delta_{H}$
{\scriptsize ({\blue by coassociativity and naturality of $c$})}

\item[ ]$=((f\ast f^{-1})\ot \lambda_{H})\co \delta_{H}$
{\scriptsize ({\blue by coassociativity and cocommutativity of $H$})}

\item[ ]$=\lambda_{H}$
{\scriptsize ({\blue by the invertibility of $f$ and counit properties}).}
\end{itemize}

\end{proof}

The right division version of Proposition \ref{alphasigmacondelta} is the following:

\begin{proposition}
\label{alphasigmacondelta-r}
Let $H$ be a non-associative  bimonoid with right division $r_{H}$. Put $\varrho_{H}=r_{H}\co (\eta_{H}\ot H)$ and assume that (\ref{primeradealpha2}) holds. Let $\sigma$ be a two-cocycle and let $g$, $g^{-1}$ be the morphisms introduced in Proposition (\ref{morfismosauxiliares}).  Define the morphism $r_{H^{\sigma}}:H\ot H\rightarrow H$ as 
$$r_{H^{\sigma}}=\mu_{H^{\sigma}}\co (H\ot g^{-1}\ot \varrho_{H}\ot g)\co (H\ot \delta_H\ot H)\co (H\ot \delta_H).$$

Then the equality
\begin{equation}
\label{deltaalphasigma-r}
\delta_H\co r_{H^{\sigma}}=(r_{H^{\sigma}}\ot r_{H^{\sigma}})\co (H\ot c_{H,H}\ot H)\co(\delta_{H}\ot (c_{H,H}\co \delta_H))
\end{equation}
holds. Moreover, 
\begin{equation}
\label{lh-sigma-eta-1-r}
r_{H^{\sigma}}\co (H\ot \eta_{H})=id_{H}, 
\end{equation}
and 
\begin{equation}
\label{lh-sigma-eta-2-r}
r_{H^{\sigma}}\co (\eta_{H}\ot H)=(g^{-1}\ot \varrho_{H}\ot g)\co (\delta_{H}\ot H)\co \delta_{H}. 
\end{equation}

Therefore, we have 
\begin{equation}
\label{lh-sigma-eta-2-1-r}
r_{H^{\sigma}}=\mu_{H^{\sigma}}\co (H\ot (r_{H^{\sigma}}\co (\eta_{H}\ot H))),  
\end{equation}
\begin{equation}
\label{lh-sigma-eta-2-2-r}
(r_{H^{\sigma}}\ot g^{-1})\co (H\ot \delta_{H})=\mu_{H^{\sigma}}\co (H\ot 
((g^{-1}\ot \varrho_{H})\co \delta_{H})). 
\end{equation}

Finally, if $H$ is cocommutative,  
\begin{equation}
\label{lh-sigma-eta-3-r}
r_{H^{\sigma}}\co (\eta_{H}\ot H)=\varrho_{H}. 
\end{equation}
\end{proposition}

The following Lemmas give two equalities will be useful to get the main result of this section.

\begin{lemma}
\label{propiedadesauxiliares}
Let $H$ be a non-associative bimonoid with left division $l_{H}$, put $\lambda_{H}=l_{H}\co (H\ot \eta_{H})$ and assume that (\ref{primeradealpha}) holds. Let $\sigma$ be a two-cocycle and let $f$, $f^{-1}$ be the morphisms introduced in Proposition (\ref{morfismosauxiliares}). Then,  the equalities
\begin{equation}
\label{primeradealphasigma}
\sigma\co ((l_{H^{\sigma}}\co (H\ot \eta_{H}))\ot H)\co (H\ot \mu_{H^{\sigma}})\co (\delta_H\ot H)=(f\ot \sigma^{-1})\co (\delta_H\ot H),
\end{equation}
\begin{equation}
\label{segundadealphasigma}
\sigma^{-1}\co (H\ot (\mu_{H^{\sigma}}\co ((l_{H^{\sigma}}\co (H\ot \eta_{H}))\ot H)))\co (\delta_{H}\ot H)=\sigma\co (\lambda_{H}\ot f^{-1}\ot H)\co(\delta_H\ot H),
\end{equation}
hold.
\end{lemma}

\begin{proof}
We begin by showing (\ref{primeradealphasigma}):

\begin{itemize}
\item[ ]$\hspace{0.38cm} \sigma\co ((l_{H^{\sigma}}\co (H\ot \eta_H))\ot H)\co (H\ot \mu_{H^{\sigma}})\co (\delta_H\ot H)$

\item[ ]$= \sigma\co (((((f\ot \lambda_{H})\co \delta_{H})\ot f^{-1})\co \delta_{H})\ot \mu_{H^{\sigma}})\co (\delta_{H}\ot H) $ {\scriptsize ({\blue by (\ref{lh-sigma-eta-2})})}

\item[ ]$=\sigma\co (H\ot (\partial^4(\sigma^{-1})\ast \partial^1(\sigma))\ot ((\mu_H\ot \sigma^{-1})\co \delta_{H\ot H}))\co (((f\ot \lambda_{H})\co \delta_{H})\ot ((\lambda_{H}\ot H)\co \delta_{H})\ot H\ot H\ot H)$
\item[ ]$\hspace{0.38cm}\co (H\ot \delta_{H\ot H})\co (\delta_H\ot H)$ {\scriptsize ({\blue by  coassociativity})}

\item[ ]$=\sigma\co (H\ot (\partial^{2}(\sigma)\ast  \partial^{3}(\sigma^{-1}))\ot ((\mu_H\ot \sigma^{-1})\co \delta_{H\ot H}))\co (((f\ot \lambda_{H})\co \delta_{H})\ot ((\lambda_{H}\ot H)\co \delta_{H})\ot H\ot H\ot H)$
\item[ ]$\hspace{0.38cm}\co (H\ot \delta_{H\ot H})\co (\delta_H\ot H)$ {\scriptsize ({\blue by  (\ref{inversa4con1})})}

\item[ ]$=\sigma\co (H\ot \sigma\ot \sigma^{-1}\ot H)\co (H\ot ((\mu_{H}\ot c_{H,H})\co (H\ot c_{H,H}\ot H)\co ((\delta_{H}\co \lambda_{H})\ot H\ot H))\ot (((\delta_{H}$
\item[ ]$\hspace{0.38cm}\co \mu_{H})\ot \sigma^{-1})\co \delta_{H\ot H}))\co (((f\ot \lambda_{H})\co \delta_{H})\ot H\ot \delta_{H\ot H})\co (H\ot \delta_{H}\ot H)\co (\delta_{H}\ot H)$ {\scriptsize ({\blue by  coassociativity}}
\item[ ]$\hspace{0.38cm}${\scriptsize {\blue  and naturality of $c$})}

\item[ ]$=\sigma\co (H\ot \sigma\ot \sigma^{-1}\ot H)\co (H\ot ((\mu_{H}\ot c_{H,H})\co (H\ot c_{H,H}\ot H)\co ((c_{H,H}\co (\lambda_{H}\ot \lambda_{H})\co \delta_{H})\ot H\ot H))\ot (((\delta_{H}$
\item[ ]$\hspace{0.38cm}\co \mu_{H})\ot \sigma^{-1})\co \delta_{H\ot H}))\co (((f\ot \lambda_{H})\co \delta_{H})\ot H\ot \delta_{H\ot H})\co (H\ot \delta_{H}\ot H)\co (\delta_{H}\ot H)$ {\scriptsize ({\blue by  (\ref{anticomul})})}

\item[ ]$=\sigma\co (H\ot  \sigma^{-1}\ot H)\co (((f\ot \lambda_{H})\co \delta_{H})\ot ((\lambda_{H}\ot \sigma)\co (H\ot (\lambda_{H}\ast id_{H})\ot H)\co (\delta_{H}\ot H))\ot (((\delta_{H}$
\item[ ]$\hspace{0.38cm} \co \mu_{H})\ot \sigma^{-1})\co \delta_{H\ot H}))\co (H\ot \delta_{H\ot H})\co  (\delta_{H}\ot H)$ {\scriptsize ({\blue  by naturality of $c$})}

\item[ ]$=\sigma\co (H\ot  \sigma^{-1}\ot H)\co (((f\ot \lambda_{H})\co \delta_{H})\ot \lambda_{H}\ot (((\delta_{H} \co \mu_{H})\ot \sigma^{-1})\co \delta_{H\ot H}))\co (H\ot \delta_{H}\ot H) \co (\delta_{H}\ot H)$ 
\item[ ]$\hspace{0.38cm}${\scriptsize ({\blue  by (\ref{primeradealpha}), counit properties, naturality of $c$, and (\ref{normaltwo-cocycle})})}

\item[ ]$=(\sigma\ot \sigma^{-1})\co (H\ot c_{H,H}\ot H)\co ((c_{H,H}\co (\lambda_{H}\ot \lambda_{H})\co \delta_{H})\ot (((\delta_{H}\co \mu_{H})\ot \sigma^{-1})\co \delta_{H\ot H}))\co (((f\ot H) $
\item[ ]$\hspace{0.38cm}\co \delta_{H})\ot H\ot H)\co (\delta_{H}\ot H)$ {\scriptsize ({\blue  by naturality of $c$})}

\item[ ]$=(\sigma\ast \sigma^{-1})\co (((f\ot \lambda_{H})\co \delta_{H})\ot 
((\mu_{H}\ot \sigma^{-1}) \co \delta_{H\ot H}))\co (\delta_{H}\ot H)${\scriptsize ({\blue  by (\ref{anticomul})})} 

\item[ ]$=(f\ot \sigma^{-1})\co (\delta_H\ot H)$ {\scriptsize ({\blue by invertibility of $\sigma$, (\ref{lambda-vareps}), counit properties, (\ref{mu-eps}), and naturality of $c$}).}

\end{itemize}

To get (\ref{segundadealphasigma}), we firstly show the equality

\begin{equation}
\label{terceradealphasigma}
(f\ot\sigma^{-1})\co (H\ot \lambda_{H}\ot H)\co (\delta_H\ot H)=\sigma\co (H\ot\mu_H)\co (H\ot \lambda_{H}\ot H)\co(\delta_H\ot H).
\end{equation}

Indeed,

\begin{itemize}
\item[ ]$\hspace{0.38cm} (f\ot\sigma^{-1})\co (H\ot \lambda_{H}\ot H)\co (\delta_H\ot H)$

\item[ ]$=(\sigma\ot \sigma^{-1})\co (H\ot ((\lambda_{H}\ot \lambda_{H})\co \delta_{H})\ot H)
\co (\delta_H\ot H)$
{\scriptsize ({\blue by definition of $f$ and coassociativity})}

\item[ ]$=(\sigma\ot \sigma^{-1})\co (H\ot (c_{H,H}\co c_{H,H}\co (\lambda_{H}\ot \lambda_{H})\co \delta_{H})\ot H)
\co (\delta_H\ot H)$
{\scriptsize ({\blue by symmetry of $c$})}

\item[ ]$=(\sigma\ot \sigma^{-1})\co (H\ot (c_{H,H}\co \delta_{H}\co \lambda_{H})\ot H)
\co (\delta_H\ot H)$
{\scriptsize ({\blue by (\ref{anticomul})})}

\item[ ]$=(\partial^{3}(\sigma)\ast \partial^{2}(\sigma^{-1}))\co  (((H\ot \lambda_{H})\co \delta_H)\ot H)$
{\scriptsize ({\blue by (\ref{3inversacon2}) and (\ref{3inversacon2-new})})}

\item[ ]$=(\sigma\ot \sigma^{-1})\co (H\ot \mu_{H\ot H}\ot H)\co (H\ot c_{H,H}\ot c_{H,H}\ot H)\co 
(\delta_{H}\ot (c_{H,H}\co (\lambda_{H}\ot \lambda_{H})\co \delta_{H})\ot \delta_{H})\co (\delta_H\ot H)$
\item[ ]$\hspace{0.38cm} ${\scriptsize ({\blue by (\ref{anticomul})})}

\item[ ]$=(\sigma\ot \sigma^{-1})\co (H\ot c_{H,H}\ot H)\co (H\ot \mu_{H}\ot \mu_{H}\ot H)\co 
(\delta_{H}\ot ( (\lambda_{H}\ot \lambda_{H})\co \delta_{H})\ot \delta_{H})\co (\delta_H\ot H)$
\item[ ]$\hspace{0.38cm} ${\scriptsize ({\blue by naturality of $c$})}

\item[ ]$=(\sigma\ot \sigma^{-1})\co (H\ot c_{H,H}\ot H)\co (H\ot (id_{H}\ast \lambda_{H})\ot \mu_{H}\ot H)\co (\delta_{H}\ot \lambda_{H}\ot \delta_{H})\co (\delta_H\ot H)$ {\scriptsize ({\blue by coasso-}}
\item[ ]$\hspace{0.38cm} ${\scriptsize {\blue ciativity})}

\item[ ]$=\sigma\co (H\ot\mu_H)\co (H\ot \lambda_{H}\ot H)\co(\delta_H\ot H)$
{\scriptsize ({\blue by (\ref{new-h-2delta}), counit properties, naturality of $c$ and normality for $\sigma^{-1}$}).}

\end{itemize}

As a consequence,

\begin{itemize}
\item[ ]$\hspace{0.38cm}\sigma^{-1}\co (H\ot (\mu_{H^{\sigma}}\co ((l_{H^{\sigma}}\co (H\ot \eta_{H}))\ot H)))\co (\delta_{H}\ot H)$

\item[ ]$= \sigma^{-1}\co (H\ot \mu_{H^{\sigma}})\co (H\ot ((f\ot \lambda_{H}\ot f^{-1})\co (\delta_{H}\ot H)\co \delta_{H})\ot H)\co (\delta_{H}\ot H) $
{\scriptsize ({\blue by (\ref{lh-sigma-eta-2})})}

\item[ ]$=  \sigma^{-1}\co (H\ot ((\mu_{H}\ot \sigma^{-1})\co (H\ot c_{H,H}\ot H)\co ((c_{H,H}\co (\lambda_{H}\ot \lambda_{H})\co \delta_{H})\ot \delta_{H})))$
\item[ ]$\hspace{0.38cm}\co (H\ot ((f\ot ((\sigma\ot H)\co (H\ot c_{H,H})\co  (((\lambda_{H}\ot H)\co c_{H,H}\co \delta_{H})\ot f^{-1}\ot H)\co (\delta_{H}\ot H)))\co (\delta_{H}\ot H))\ot H)$
\item[ ]$\hspace{0.38cm}\co (\delta_{H}\ot \delta_{H})$
{\scriptsize ({\blue by naturality of $c$ and (\ref{anticomul})})}

\item[ ]$= (\sigma^{-1}\ot ((f\ot\sigma^{-1})\co (H\ot \lambda_{H}\ot H)\co (\delta_H\ot H)))  $
\item[ ]$\hspace{0.38cm}\co (H\ot (c_{H,H}\co (H\ot (\mu_{H}\co (H\ot \sigma\ot H)\co (((\lambda_{H}\ot \lambda_{H})\co \delta_{H})\ot f^{-1}\ot \delta_{H})\co (\delta_{H}\ot H)))\co (\delta_{H}\ot H))\ot H)))$
\item[ ]$\hspace{0.38cm}\co (\delta_{H}\ot \delta_{H})$ {\scriptsize ({\blue by naturality of $c$})}

\item[ ]$= (\sigma^{-1}\ot (\sigma\co (H\ot\mu_H)\co (H\ot \lambda_{H}\ot H)\co(\delta_H\ot H))) $
\item[ ]$\hspace{0.38cm}\co (H\ot (c_{H,H}\co (H\ot (\mu_{H}\co (H\ot \sigma\ot H)\co (((\lambda_{H}\ot \lambda_{H})\co \delta_{H})\ot f^{-1}\ot \delta_{H})\co (\delta_{H}\ot H)))\co (\delta_{H}\ot H))\ot H)))$
\item[ ]$\hspace{0.38cm}\co (\delta_{H}\ot \delta_{H})$ {\scriptsize ({\blue by (\ref{terceradealphasigma})})}

\item[ ]$=((\sigma^{-1}\co (H\ot \mu_{H}))\ot (\sigma \co (H\ot \mu_{H})))\co (H\ot H\ot c_{H,H}\ot H\ot H)\co (H\ot c_{H,H}\ot c_{H,H}\ot H) $
\item[ ]$\hspace{0.38cm}\co (\delta_{H}\ot (c_{H,H}\co (\lambda_{H}\ot \lambda_{H})\co \delta_{H})\ot (\sigma\co (((\lambda_{H}\ot f^{-1})\co \delta_{H})\ot H))\ot \delta_{H})\co (\delta_{H}\ot H\ot \delta_{H})\co  (\delta_{H}\ot H)$
\item[ ]$\hspace{0.38cm}$ {\scriptsize ({\blue  by naturality of $c$ and coassociativity})}

\item[ ]$=(\partial_{3}(\sigma^{-1})\ast \partial_{3}(\sigma))\co (H\ot \lambda_{H}\ot H)\co (\delta_{H}\ot (\sigma\co (\lambda_{H}\ot f^{-1}\ot H)\co (\delta_{H}\ot H))\ot H)\co (\delta_{H}\ot \delta_{H})$ 
\item[ ]$\hspace{0.38cm}${\scriptsize ({\blue by (\ref{anticomul})})}

\item[ ]$=\sigma\co (\lambda_{H}\ot f^{-1}\ot H)\co(\delta_H\ot H)$
{\scriptsize({\blue by invertibility of $\partial_{3}(\sigma)$ (see Remark \ref{formulas2cociclos}), counit properties and (\ref{lambda-vareps})}),}

\end{itemize}

and the proof is complete.

\end{proof}

\begin{lemma}
\label{propiedadesauxiliares-r}
Let $H$ be a non-associative bimonoid with right division $r_{H}$, put $\varrho_{H}=r_{H}\co (\eta_{H}\ot H)$ and assume that (\ref{primeradealpha2}) holds. Let $\sigma$ be a two-cocycle and let $g$, $g^{-1}$ be the morphisms introduced in Proposition (\ref{morfismosauxiliares}). Then,  the equalities
\begin{equation}
\label{primeradealphasigma-r}
\sigma^{-1}\co (H\ot (r_{H^{\sigma}}\co (\eta_{H}\ot H)))\co ( \mu_{H^{\sigma}}\ot H)\co (H\ot \delta_H)=(\sigma\ot g)\co (H\ot \delta_H),
\end{equation}
\begin{equation}
\label{segundadealphasigma-r}
\sigma\co ((\mu_{H^{\sigma}}\co (H\ot (r_{H^{\sigma}}\co (\eta_{H}\ot H))))\ot H)\co (H\ot \delta_{H})=\sigma^{-1}\co (H\ot g^{-1}\ot \varrho_{H})\co(H\ot \delta_H),
\end{equation}
hold.
\end{lemma}

\begin{proof} The proof is similar to the one performed in the previous lemma but using 
\begin{equation}
\label{terceradealphasigma-r}
(\sigma\ot g)\co (H\ot \varrho_{H}\ot H)\co (H\ot \delta_{H})=\sigma^{-1}\co (\mu_{H}\ot H)\co (H\ot \varrho_{H}\ot H)\co (H\ot \delta_{H})
\end{equation}
instead of (\ref{terceradealphasigma}).
\end{proof}

The following theorem is the main result of this section. We will show that,  under suitable conditions, $H^{\sigma}$ is a non-associative bimonoid with (right) left division ($r_{H^{\sigma}}$) $l_{H^{\sigma}}$.

\begin{theorem}
\label{construccion}
The following asssertions hold:
\begin{itemize}
\item[(i)] Let $H$ be a left Hopf quasigroup with left antipode $\lambda_{H}$. Let $\sigma$ be a two-cocycle.  Then the non-associative bimonoid $H^{\sigma}$ defined in Proposition \ref{algebradeformada} is a left Hopf quasigroup with left antipode $\lambda_{H^{\sigma}}=l_{H^{\sigma}}\co (H\ot \eta_{H})$, where  $l_{H^{\sigma}}$ is the morphism introduced in Proposition \ref{alphasigmacondelta}.
\item[(ii)] Let $H$ be a right Hopf quasigroup with right antipode $\varrho_{H}$. Let $\sigma$ be a two-cocycle.  Then the non-associative bimonoid $H^{\sigma}$ defined in Proposition \ref{algebradeformada} is a right Hopf quasigroup with right antipode $\varrho_{H^{\sigma}}=r_{H^{\sigma}}\co (\eta_{H}\ot H)$, where  $r_{H^{\sigma}}$ is the morphism introduced in Proposition \ref{alphasigmacondelta-r}.
\item[(iii)] Let $H$ be a  Hopf quasigroup with  antipode $\lambda_{H}$. Let $\sigma$ be a two-cocycle.  Then the non-associative bimonoid $H^{\sigma}$ defined in Proposition \ref{algebradeformada} is a  Hopf quasigroup with  antipode 
$\lambda_{H^{\sigma}}$.
\end{itemize}

\end{theorem}

\begin{proof}  We prove (i). The proof for (ii) is similar using Lemma \ref{propiedadesauxiliares-r} instead of Lemma \ref{propiedadesauxiliares}. The assertion (iii) follows from Remark \ref{antipodounico}. Note that if $H$ is a Hopf quasigroup $\lambda_{H}$ is a left and right antipode and then $$\lambda_{H^{\sigma}}=l_{H^{\sigma}}\co (H\ot \eta_{H})=r_{H^{\sigma}}\co (\eta_{H}\ot H)=\varrho_{H^{\sigma}}$$ because 
$f=g^{-1}$ and $f^{-1}=g$.

First note that  by Proposition \ref{algebradeformada}, $H^{\sigma}$ is a non-associative bimonoid. Therefore, to complete the proof, we only need to show  (\ref{leftdivision}) for $l_{H^{\sigma}}$ and $\mu_{H^{\sigma}}$, because (\ref{lh-sigma-eta-2-1}) holds (see Proposition  \ref{alphasigmacondelta}) and then, by Proposition \ref{igualdadesalpha}, we obtain that $H^{\sigma}$ is a left Hopf quasigroup where $\lambda_{H^{\sigma}}=l_{H^{\sigma}}\co (H\ot \eta_{H})$ is the left antipode. Indeed, in one hand we have 

\begin{itemize}
\item[ ]$\hspace{0.38cm} l_{H^{\sigma}}\co (H\ot \mu_{H^{\sigma}})\co (\delta_H\ot H)$

\item[ ]$=(\sigma \ot ((\mu_{H}\ot \sigma^{-1})\co \delta_{H\ot H}))\co (H\ot c_{H,H}\ot H)\co 
(f\ot (((\delta_{H}\co \lambda_{H})\ot f^{-1})\co \delta_{H})\ot (\delta_{H}\co \mu_{H^{\sigma}})) $
\item[ ]$\hspace{0.38cm} \co (\delta_H\ot H\ot H) \co (\delta_H\ot H)$  {\scriptsize ({\blue by definition of $\mu_{H^{\sigma}}$})}

\item[ ]$=(\sigma \ot ((\mu_{H}\ot \sigma^{-1})\co \delta_{H\ot H}))\co (H\ot c_{H,H}\ot H)\co 
(f\ot (((c_{H,H}\co (\lambda_{H}\ot \lambda_{H})\co \delta_{H})\ot f^{-1})\co \delta_{H})
$
\item[ ]$\hspace{0.38cm} \ot ((\mu_{H^{\sigma}}\ot \mu_{H^{\sigma}})\co \delta_{H\ot H}))\co (H\ot \delta_{H}\ot H)\co (\delta_H\ot H)  $  {\scriptsize ({\blue by (\ref{anticomul}) and (\ref{delta-mu}) for $H^{\sigma}$})}

\item[ ]$= ((\sigma\co (\lambda_{H}\ot f^{-1}\ot \mu_{H^{\sigma}})\co (((H\ot \delta_{H})\co \delta_{H})\ot H))\ot ((\mu_{H}\ot \sigma^{-1})\co \delta_{H\ot H})) \co (H\ot c_{H,H}\ot \mu_{H^{\sigma}})$
\item[ ]$\hspace{0.38cm}\co (c_{H,H}\ot c_{H,H}\ot H)\co (((f\ot \lambda_{H})\co \delta_{H})\ot \delta_{H}\ot \delta_{H})\co (\delta_H\ot H) $  {\scriptsize({\blue by naturality of $c$ and coassociativity})}

\item[ ]$= (((\sigma\co (\lambda_{H}\ot f^{-1}\ot \mu_{H^{\sigma}})\co (((H\ot \delta_{H})\co \delta_{H})\ot H))\co ((((f^{-1}\ast f)\ot H)\co \delta_{H})\ot H))\ot ((\mu_{H}\ot \sigma^{-1})\co \delta_{H\ot H})) $
\item[ ]$\hspace{0.38cm}\co (H\ot c_{H,H}\ot \mu_{H^{\sigma}}) \co (c_{H,H}\ot c_{H,H}\ot H)\co (((f\ot \lambda_{H})\co \delta_{H})\ot \delta_{H}\ot \delta_{H})\co (\delta_H\ot H) $  {\scriptsize ({\blue by invertibility }}
\item[ ]$\hspace{0.38cm}$ {\scriptsize {\blue of $f$ and counit properties})}

\item[ ]$= (((\sigma\co ((l_{H^{\sigma}}\co (H\ot \eta_{H}))\ot \mu_{H^{\sigma}})\co (\delta_{H}\ot H)))\ot ((\mu_{H}\ot \sigma^{-1})\co \delta_{H\ot H}))\co (H\ot c_{H,H}\ot \mu_{H^{\sigma}})  $
\item[ ]$\hspace{0.38cm}\co (c_{H,H}\ot c_{H,H}\ot H)\co (((f\ot ((\lambda_{H}\ot f^{-1})\co \delta_{H}))\co \delta_{H})\ot \delta_{H}\ot \delta_{H})\co (\delta_H\ot H) $ {\scriptsize ({\blue by naturality of $c$}}, \item[ ]$\hspace{0.38cm}$ {\scriptsize {\blue coassociativity of $\delta_{H}$ and (\ref{lh-sigma-eta-2})})}

\item[ ]$= (((f\ot \sigma^{-1})\co (\delta_{H}\ot H))\ot ((\mu_{H}\ot \sigma^{-1})\co \delta_{H\ot H}))\co (H\ot c_{H,H}\ot \mu_{H^{\sigma}})  \co (c_{H,H}\ot c_{H,H}\ot H)$
\item[ ]$\hspace{0.38cm}\co (((f\ot ((\lambda_{H}\ot f^{-1})\co \delta_{H})) \co \delta_{H})\ot \delta_{H}\ot \delta_{H})\co (\delta_H\ot H) ${\scriptsize ({\blue by  (\ref{primeradealphasigma})})}

\item[ ]$= ( \sigma^{-1}\ot ((\mu_{H}\ot \sigma^{-1})\co \delta_{H\ot H}))\co (H\ot c_{H,H}\ot \mu_{H^{\sigma}})  \co (c_{H,H}\ot c_{H,H}\ot H)\co (((f\ot ((\lambda_{H}\ot (f^{-1}\ast f))\co \delta_{H})) $
\item[ ]$\hspace{0.38cm} \co \delta_{H})\ot \delta_{H}\ot \delta_{H})\co (\delta_H\ot H) $ {\scriptsize ({\blue by  naturality of $c$ and coassociativity})}

\item[ ]$= ( \sigma^{-1}\ot ((\mu_{H}\ot \sigma^{-1})\co \delta_{H\ot H}))\co (H\ot c_{H,H}\ot \mu_{H^{\sigma}})  \co (c_{H,H}\ot c_{H,H}\ot H)\co (((f\ot \lambda_{H})\co \delta_{H})\ot \delta_{H}\ot \delta_{H})$
\item[ ]$\hspace{0.38cm} \co (\delta_H\ot H) $ {\scriptsize ({\blue by invertibility of $f$ and counit properties})}

\item[ ]$= (\mu_{H}\ot \sigma^{-1})\co (H\ot c_{H,H}\ot H)\co (\delta_{H}\ot (\delta_{H}\co \mu_{H})\ot \sigma^{-1})\co (((f\ot \lambda_{H})\co \delta_{H})\ot (\sigma^{-1}\ast \sigma)\ot \delta_{H\ot H})$
\item[ ]$\hspace{0.38cm}\co (H\ot \delta_{H\ot H})\co (\delta_{H}\ot H) $ 
 {\scriptsize ({\blue by  naturality of $c$, coassociativity of $\delta_{H}$ and definition of $\mu_{H^{\sigma}}$})}

\item[ ]$= (\mu_{H}\ot \sigma^{-1})\co (H\ot c_{H,H}\ot H)\co (\delta_{H}\ot ((\mu_{H}\ot \mu_{H})\co \delta_{H\ot H})\ot \sigma^{-1})\co (((f\ot \lambda_{H})\co \delta_{H})\ot (\sigma^{-1}\ast \sigma)\ot \delta_{H\ot H})$
\item[ ]$\hspace{0.38cm}\co (H\ot \delta_{H\ot H})\co (\delta_{H}\ot H) $ 
 {\scriptsize ({\blue by   (\ref{delta-mu}) and counit properties})}

\item[ ]$= (\mu_{H}\ot (\sigma^{-1}\co (H\ot ((\mu_{H}\ot \sigma^{-1})\co \delta_{H\ot H}))))\co  (H\ot c_{H,H}\ot H\ot H)\co (((f\ot (\delta_{H}\co \lambda_{H}))\co \delta_{H})\ot \mu_{H}\ot H\ot H)$
\item[ ]$\hspace{0.38cm} \co (H\ot \delta_{H\ot H})\co (\delta_{H}\ot H)$ {\scriptsize ({\blue by  naturality of $c$ and  coassociativity})}

\item[ ]$= (\mu_{H}\ot (\partial^{3}(\sigma^{-1})\ast \partial^{1}(\sigma^{-1})))\co  (H\ot c_{H,H}\ot H\ot H)\co (((f\ot (\delta_{H}\co \lambda_{H}))\co \delta_{H})\ot \mu_{H}\ot H\ot H)\co (H\ot \delta_{H\ot H})$
\item[ ]$\hspace{0.38cm} \co (\delta_{H}\ot H)$ {\scriptsize ({\blue by  (\ref{inversa3coninversa1-new})})}

\item[ ]$= (\mu_{H}\ot (\partial^{2}(\sigma^{-1})\ast \partial^{4}(\sigma^{-1})))\co  (H\ot c_{H,H}\ot H\ot H)\co (((f\ot (\delta_{H}\co \lambda_{H}))\co \delta_{H})\ot \mu_{H}\ot H\ot H)\co (H\ot \delta_{H\ot H})$
\item[ ]$\hspace{0.38cm} \co (\delta_{H}\ot H)$ {\scriptsize ({\blue by  (\ref{inversa3coninversa1})})}

\item[ ]$= (\mu_{H}\ot (\sigma^{-1}\co (((\mu_{H}\ot \sigma^{-1})\co \delta_{H\ot H})\ot H)))\co  (H\ot c_{H,H}\ot H\ot H)\co (((f\ot (\delta_{H}\co \lambda_{H}))\co \delta_{H})\ot \mu_{H}\ot H\ot H)$
\item[ ]$\hspace{0.38cm} \co (H\ot \delta_{H\ot H}) \co (\delta_{H}\ot H)$ {\scriptsize ({\blue by  (\ref{inversa3coninversa1-new})})}

 \item[ ]$=(\mu_{H}\ot (\sigma^{-1}\co (((\mu_{H}\ot \sigma^{-1})\co \delta_{H\ot H})\ot H)))\co  (H\ot c_{H,H}\ot H\ot H)\co (((f\ot (c_{H,H}\co (\lambda_{H}\ot \lambda_{H})$
\item[ ]$\hspace{0.38cm} \co \delta_{H}))\co \delta_{H})\ot \mu_{H}\ot H\ot H) \co (H\ot \delta_{H\ot H}) \co (\delta_{H}\ot H)$
{\scriptsize ({\blue by (\ref{anticomul})})}

\item[ ]$=(H\ot (\sigma^{-1}\co ((\mu_{H}\ot \sigma^{-1})\co \delta_{H\ot H}) \ot H))\co (c_{H,H}\ot H\ot H)\co (f\ot \lambda_{H}\ot (\mu_{H}\co (\lambda_{H}\ot \mu_{H})\co (\delta_{H}\ot H))$
\item[ ]$\hspace{0.38cm}\ot H\ot H) \co (H\ot H\ot \delta_{H\ot H})\co (H\ot \delta_{H}\ot H)\co (\delta_{H}\ot H) $ 
{\scriptsize ({\blue by  naturality of $c$ and  coassociativity})}

\item[ ]$=(H\ot \sigma^{-1})\co (c_{H,H}\ot H) \co (f\ot ((\mu_{H}\ot \sigma^{-1})\co (H\ot c_{H,H}\ot H)\co ((\delta_{H}\co \lambda_{H})\ot \delta_{H})\co \delta_{H})\ot \delta_{H})\co (\delta_{H}\ot H)$
\item[ ]$\hspace{0.38cm} $ {\scriptsize ({\blue by (\ref{leftdivision}), counit properties and naturality of $c$})}

\item[ ]$=(H\ot \sigma^{-1})\co (\sigma^{-1}\ot c_{H,H}\ot H)\co (f\ot ((H\ot c_{H,H})\co (H\ot \mu_{H}\ot H)\co (((\lambda_{H}\ot \lambda_{H})\co \delta_{H}) \ot \delta_{H})\co \delta_{H})\ot \delta_{H})$
\item[ ]$\hspace{0.38cm} \co (\delta_{H}\ot H)$ 
{\scriptsize ({\blue by naturality of $c$ and (\ref{anticomul})})}

\item[ ]$=(H\ot \sigma^{-1})\co (\sigma^{-1}\ot c_{H,H}\ot H)\co 
(f\ot ((\lambda_{H}\ot (c_{H,H}\co ((\lambda_{H}\ast id_{H})\ot  H)\co \delta_{H}))\co \delta_{H})\ot \delta_{H})\co (\delta_{H}\ot H)$
\item[ ]$\hspace{0.38cm} $ {\scriptsize ({\blue by coassociativity of $\delta_{H}$})}

\item[ ]$=(f\ast f^{-1})\ot H$
{\scriptsize ({\blue by (\ref{primeradealpha}), naturality of $c$, normality for $\sigma^{-1}$ and counit properties})}

\item[ ]$=\varepsilon_H\ot H$
{\scriptsize ({\blue by invertibility of $f$}).}

\end{itemize}

Finally, on the other hand,
\begin{itemize}
\item[ ]$\hspace{0.38cm} \mu_{H^{\sigma}}\co (H\ot l_{H^{\sigma}})\co (\delta_H\ot H)$

\item[ ]$=(\mu_{H}\ot \sigma^{-1})\co (H\ot c_{H,H}\ot H)\co (\sigma\ot \delta_{H}\ot ((l_{H^{\sigma}}\ot l_{H^{\sigma}})\co (H\ot c_{H,H}\ot H)\co ((c_{H,H}\co \delta_{H})\ot \delta_{H})))$
\item[ ]$\hspace{0.38cm} \co (H\ot c_{H,H}\ot H\ot H)\co (\delta_{H}\ot ((l_{H^{\sigma}}\ot H)\co (H\ot c_{H,H})\co ((c_{H,H}\co \delta_{H})\ot H))\ot \delta_{H})\co (\delta_{H}\ot H)$ 
\item[ ]$\hspace{0.38cm}${\scriptsize ({\blue by (\ref{deltaalphasigma})})}

\item[ ]$=(\mu_{H}\ot (\sigma^{-1}\co (H\ot l_{H^{\sigma}})\co (\delta_{H}\ot H)))\co (\sigma\ot ((H\ot c_{H,H})\co (\delta_{H}\ot l_{H^{\sigma}})\co (\delta_{H}\ot H))\ot H)\co (H\ot c_{H,H}\ot \delta_{H}) $
\item[ ]$\hspace{0.38cm}\co (\delta_{H}\ot l_{H^{\sigma}}\ot H)\co (\delta_{H}\ot \delta_{H}) $ {\scriptsize ({\blue by naturality of $c$ and coassociativity})}

\item[ ]$=(\mu_{H}\ot (\sigma^{-1}\co (H\ot (\mu_{H^{\sigma}}\co ((l_{H^{\sigma}}\co (H\ot \eta_{H}))\ot H)))\co (\delta_{H}\ot H)))\co (\sigma\ot ((H\ot c_{H,H})\co (\delta_{H}\ot l_{H^{\sigma}})\co $
\item[ ]$\hspace{0.38cm} (\delta_{H}\ot H))\ot H) \co (H\ot c_{H,H}\ot \delta_{H})\co (\delta_{H}\ot l_{H^{\sigma}}\ot H)\co (\delta_{H}\ot \delta_{H}) $ {\scriptsize ({\blue by (\ref{lh-sigma-eta-2-1})})}

\item[ ]$=(\mu_{H}\ot (\sigma\co (\lambda_{H}\ot f^{-1}\ot H)\co(\delta_H\ot H)))\co  (\sigma\ot ((H\ot c_{H,H})\co (\delta_{H}\ot l_{H^{\sigma}})\co (\delta_{H}\ot H))\ot H)  $
\item[ ]$\hspace{0.38cm} \co (H\ot c_{H,H}\ot \delta_{H})\co (\delta_{H}\ot l_{H^{\sigma}}\ot H)\co (\delta_{H}\ot \delta_{H}) $ {\scriptsize ({\blue by (\ref{segundadealphasigma})})}

\item[ ]$=(\mu_{H}\ot \sigma)\co (H\ot c_{H,H}\ot H)\co (\sigma \ot ((((H\ot \lambda_{H})\co \delta_{H})\ot ((f^{-1}\ot l_{H^{\sigma}})\co (\delta_{H}\ot H))\ot H)\co (\delta_{H}\ot \delta_{H})))$
\item[ ]$\hspace{0.38cm} \co (H\ot c_{H,H}\ot H)\co (\delta_{H}\ot l_{H^{\sigma}}\ot H)\co (\delta_{H}\ot \delta_{H}) $ {\scriptsize ({\blue by naturality of $c$ and coassociativity})}

\item[ ]$=(\mu_{H}\ot \sigma)\co (H\ot c_{H,H}\ot H)\co (\sigma \ot ((((H\ot \lambda_{H})\co \delta_{H})\ot (\mu_{H^{\sigma}}\co (((\lambda_{H}\ot f^{-1})\co \delta_{H})\ot H))\ot H)\co (\delta_{H}\ot \delta_{H})))$
\item[ ]$\hspace{0.38cm} \co (H\ot c_{H,H}\ot H)\co (\delta_{H}\ot l_{H^{\sigma}}\ot H)\co (\delta_{H}\ot \delta_{H}) $ {\scriptsize ({\blue by (\ref{lh-sigma-eta-2-2})})}

\item[ ]$=\mu_{H}\co (H\ot \mu_{H^{\sigma}}\ot \sigma)\co (H\ot H\ot c_{H,H}\ot H)\co (H\ot (c_{H,H}\co (\lambda_{H}\ot \lambda_{H})\co \delta_{H})\ot H\ot H)\co (\sigma\ot \delta_{H}\ot \delta_{H})$
\item[ ]$\hspace{0.38cm} \co (H\ot c_{H,H}\ot H)\co (\delta_{H}\ot ((f^{-1}\ot l_{H^{\sigma}})\co (\delta_{H}\ot H))\ot H)\co (\delta_{H}\ot \delta_{H})$ {\scriptsize({\blue by naturality of $c$ and coassociativity})}

\item[ ]$=\mu_{H}\co (H\ot \mu_{H^{\sigma}}\ot \sigma)\co (H\ot H\ot c_{H,H}\ot H)\co (H\ot (\delta_{H}\co \lambda_{H})\ot H\ot H)\co (\sigma\ot \delta_{H}\ot \delta_{H})\co (H\ot c_{H,H}\ot H)$
\item[ ]$\hspace{0.38cm} \co (\delta_{H}\ot (\mu_{H^{\sigma}}\co (((\lambda_{H}\ot f^{-1})\co \delta_{H})\ot H))\ot H)\co (\delta_{H}\ot \delta_{H})$ {\scriptsize ({\blue by (\ref{lh-sigma-eta-2-2}) and (\ref{anticomul})})}

\item[ ]$=\mu_{H}\co  (H\ot ((((\sigma \ot \mu_{H})\co \delta_{H\ot H})\ot (\sigma^{-1}\ast \sigma))\co \delta_{H\ot H}\co (\lambda_{H}\ot H)))\co (\sigma\ot \delta_{H}\ot H)\co (H\ot c_{H,H}\ot H)$
\item[ ]$\hspace{0.38cm} \co (\delta_{H}\ot (\mu_{H^{\sigma}}\co (((\lambda_{H}\ot f^{-1})\co \delta_{H})\ot H))\ot H)\co (\delta_{H}\ot \delta_{H})$ {\scriptsize ({\blue by naturality of $c$ and coassociativity})}

\item[ ]$=\mu_{H}\co  (H\ot ((\sigma \ot \mu_{H})\co (H\ot c_{H,H}\ot H)\co ((\delta_{H}\co \lambda_{H})\ot \delta_{H})))\co (\sigma\ot \delta_{H}\ot H)\co (H\ot c_{H,H}\ot H)$
\item[ ]$\hspace{0.38cm} \co (\delta_{H}\ot (\mu_{H^{\sigma}}\co (((\lambda_{H}\ot f^{-1})\co \delta_{H})\ot H))\ot H)\co (\delta_{H}\ot \delta_{H})$ {\scriptsize ({\blue by invertibility of $\sigma$, naturality of $c$ and  }}
\item[ ]$\hspace{0.38cm} ${\scriptsize {\blue counit properties})}

\item[ ]$=\mu_{H}\co  (H\ot ((\sigma \ot \mu_{H})\co (H\ot c_{H,H}\ot H)\co ((c_{H,H}\co (\lambda_{H}\ot \lambda_{H})\co \delta_{H})\ot \delta_{H})))\co (\sigma\ot \delta_{H}\ot H)\co (H\ot c_{H,H}\ot H)$
\item[ ]$\hspace{0.38cm} \co (\delta_{H}\ot (\mu_{H^{\sigma}}\co (((\lambda_{H}\ot f^{-1})\co \delta_{H})\ot H))\ot H)\co (\delta_{H}\ot \delta_{H})$ {\scriptsize ({\blue by (\ref{anticomul})})}

\item[ ]$=\mu_{H}\co  (H\ot (\mu_{H}\co (H\ot \sigma\ot H)\co (((\lambda_{H}\ot \lambda_{H})\co \delta_{H})\ot \delta_{H})))\co (\sigma\ot \delta_{H}\ot H)\co (H\ot c_{H,H}\ot H)$
\item[ ]$\hspace{0.38cm} \co (\delta_{H}\ot (\mu_{H^{\sigma}}\co (((\lambda_{H}\ot f^{-1})\co \delta_{H})\ot H))\ot H)\co (\delta_{H}\ot \delta_{H})$ {\scriptsize ({\blue by naturality of $c$})}

\item[ ]$=\mu_{H}\co (H\ot (\mu_{H}\co (\lambda_{H}\ot H))\co (\delta_{H}\ot (\sigma\co (\lambda_{H}\ot H))\ot H)\co (\sigma\ot \delta_{H}\ot \delta_{H})\co (H\ot c_{H,H}\ot H)$
\item[ ]$\hspace{0.38cm} \co (\delta_{H}\ot (\mu_{H^{\sigma}}\co (((\lambda_{H}\ot f^{-1})\co \delta_{H})\ot H))\ot H)\co (\delta_{H}\ot \delta_{H})$ {\scriptsize ({\blue by coassociativity})}

\item[ ]$=((\sigma\co (\lambda_{H}\ot H))\ot H)\co (\sigma\ot H\ot \delta_{H})\co (H\ot c_{H,H}\ot H)
\co (\delta_{H}\ot (\mu_{H^{\sigma}}\co (((\lambda_{H}\ot f^{-1})\co \delta_{H})\ot H))\ot H)$ 
\item[ ]$\hspace{0.38cm} \co (\delta_{H}\ot \delta_{H})$ {\scriptsize ({\blue by (\ref{leftdivision}) and counit properties})}

\item[ ]$=(\sigma\ot H)\co (H\ot ((\mu_{H^{\sigma}}\ot \sigma)\co (H\ot c_{H,H}\ot H)\co ((c_{H,H}\co (\lambda_{H}\ot \lambda_{H})\co \delta_{H})\ot f^{-1}\ot \delta_{H}))\ot H)\co (\delta_{H}\ot H\ot \delta_{H})$
\item[ ]$\hspace{0.38cm} \co (\delta_{H}\ot H)$ {\scriptsize ({\blue by naturality of $c$ and coassociativity})}

\item[ ]$=(\sigma\ot H)\co (H\ot ((\mu_{H^{\sigma}}\ot \sigma)\co (H\ot c_{H,H}\ot H)\co ((\delta_{H}\co \lambda_{H})\ot f^{-1}\ot \delta_{H}))\ot H)\co (\delta_{H}\ot H\ot \delta_{H}) \co (\delta_{H}\ot H)$
\item[ ]$\hspace{0.38cm}${\scriptsize ({\blue by (\ref{anticomul})})}

\item[ ]$=((\sigma\co (H\ot ((\sigma\ot \mu_{H})\co\delta_{H\ot H})))\ot (\sigma^{-1}\ast \sigma)\ot H)\co (H\ot H\ot c_{H,H}\ot H\ot H)\co (H\ot (\delta_{H}\co \lambda_{H})\ot f^{-1}$
\item[ ]$\hspace{0.38cm}\ot \delta_{H}\ot H) \co (\delta_{H}\ot H\ot \delta_{H}) \co (\delta_{H}\ot H) $ {\scriptsize ({\blue by naturality of $c$ and coassociativity})}

\item[ ]$=((\partial^1(\sigma)\ast \partial^3(\sigma))\ot H)\co (H\ot \lambda_{H}\ot f^{-1}\ot \delta_{H}) \co (\delta_{H}\ot H\ot H) \co (\delta_{H}\ot H) $ {\scriptsize ({\blue by (\ref{two-cocycle-1})})}

\item[ ]$=((\partial^4(\sigma)\ast \partial^2(\sigma))\ot H)\co (H\ot \lambda_{H}\ot f^{-1}\ot \delta_{H}) \co (\delta_{H}\ot H\ot H) \co (\delta_{H}\ot H) $ {\scriptsize ({\blue by (\ref{two-cocycle})})}

\item[ ]$=(\sigma\ot H)\co (((\sigma\ot \mu_{H})\co \delta_{H\ot H})\ot H\ot H)\co (H\ot \lambda_{H} \ot f^{-1}\ot \delta_{H}) \co  (\delta_{H}\ot H\ot H) \co (\delta_{H}\ot H)$ {\scriptsize ({\blue by }}
\item[ ]$\hspace{0.38cm}${\scriptsize {\blue  (\ref{two-cocycle-1})})}

\item[ ]$=(\sigma\ot H)\co (((\sigma\ot \mu_{H})\co (H\ot c_{H,H}\ot H)\co (\delta_{H}\ot (c_{H,H}\co (\lambda_{H}\ot \lambda_{H})\co \delta_{H})))\ot f^{-1}\ot \delta_{H}) \co  (\delta_{H}\ot H\ot H)  $ 
\item[ ]$\hspace{0.38cm}\co (\delta_{H}\ot H)$ {\scriptsize ({\blue by (\ref{anticomul})})}

\item[ ]$=(\sigma\ot (\sigma\co ((id_{H}\ast \lambda_{H})\ot H)))\co (H\ot c_{H,H}\ot \delta_{H})\co (\delta_{H}\ot \lambda_{H}\ot f^{-1}\ot H)\co (\delta_{H}\ot H\ot H) \co  (\delta_{H}\ot H)$ 
\item[ ]$\hspace{0.38cm}${\scriptsize ({\blue by naturality of $c$ and coassociativity})}

\item[ ]$=(f\ast f^{-1})\ot H$
{\scriptsize ({\blue by (\ref{new-h-2delta}), naturality of $c$, normality for $\sigma$ and counit properties} )}

\item[ ]$=\varepsilon_H\ot H$
{\scriptsize ({\blue  by invertibility of $f$}),}

\end{itemize}

and the proof is complete.
\end{proof}

\section{Two-cocycles and  skew pairings}

In this section we will see that, in a similar way that for the Hopf algebra case, a class of two-cocycles is provided by invertible skew pairings for non-associative bimonoids.  The following definition is inspired in the corresponding one for Hopf algebras introduced by Doi and Takeuchi in \cite{D-T} (see also \cite{P-M} for the monoidal setting). 

\begin{definition}
\label{skewpairing}
{\rm Let $A$ and $H$ be non-associative bimonoids  in ${\mathcal C}$. A pairing between $A$ and $H$ over $K$ is a morphism $\tau:A\ot H\rightarrow K$ such that the equalities
\begin{itemize}
\item[(a1)]$\tau\co (\mu_A\ot H)=(\tau\ot \tau)\co (A\ot c_{A,H}\ot H)\co (A\ot A\ot\delta_H),$
\item[(a2)]$\tau\co (A\ot \mu_H)=(\tau\ot \tau)\co (A\ot c_{A,H}\ot H)\co (\delta_{A}\ot H\ot H),$
\item[(a3)]$\tau\co (A\ot \eta_{H})=\varepsilon_{A}, $
\item[(a4)]$\tau\co (\eta_{A}\ot H)=\varepsilon_{H}, $
\end{itemize}
hold.

A skew pairing between $A$ and $H$ is a pairing between $A^{cop}$ and $H$, i.e., a morphism $\tau:A\ot H\rightarrow K$ satisfying (a1), (a3), (a4) and
\begin{itemize}
\item[(a2')]$\tau\co (A\ot \mu_H)=(\tau\ot \tau)\co (A\ot c_{A,H}\ot H)\co ((c_{A,A}\co\delta_{A})\ot H\ot H).$
\end{itemize}

It is easy to see that, by naturality of $c$, equality (a2') is equivalent to 
\begin{equation}
\label{c2'-new}
\tau\co (A\ot \mu_H)=(\tau\ot \tau)\co (A\ot c_{A,H}\ot H)\co (\delta_{A}\ot c_{H,H}).
\end{equation}

}
\end{definition}

\begin{remark}\label{Hopfpairing}
{\rm Note that, if $A$ and $H$ are Hopf quasigroups, a pairing between $A$ and $H^{cop}$ corresponds with the definition of Hopf pairing introduced in \cite{FT}.
}
\end{remark}

\begin{proposition}
\label{propiedadesskew}
Let $A$, $H$ be non-associative bimonoids with left division $l_A$ and $l_H$, respectively. Let $\tau:A\ot H\rightarrow K$ be a skew pairing. Then $\tau$ is convolution invertible. Moreover, if $\tau^{-1}$ is the inverse of $\tau$, the equalities
\begin{equation}
\label{skewpairing-1}
\tau^{-1}\co (\eta_A\ot H)=\varepsilon_H,
\end{equation}
\begin{equation}
\label{skewpairing-2}
\tau^{-1}\co (A\ot \eta_H)=\varepsilon_A,
\end{equation}
and
\begin{equation}
\label{skewpairing-3}
\tau^{-1}\co (A\ot \mu_{H})=(\tau^{-1}\ot \tau^{-1})\co (A\ot c_{A,H}\ot H)\co (\delta_{A}\ot H\ot H),
\end{equation}
hold.

\end{proposition}

\begin{proof}
Define $\tau^{-1}=\tau\co (\lambda_{A}\ot H),$ where $\lambda_{A}=l_A\co (A\ot \eta_A).$ Then, 
\begin{itemize}
\item[ ]$\hspace{0.38cm} \tau\ast \tau^{-1}$

\item[ ]$=(\tau\ot \tau)\co (A\ot c_{A,H}\ot H)\co (((A\ot \lambda_{A})\co \delta_{A})\ot \delta_{H})$ {\scriptsize ({\blue by naturality of $c$})}

\item[ ]$=\tau \co ((id_{A}\ast \lambda_{A})\ot H)$
{\scriptsize ({\blue by (a1) of Definition \ref{skewpairing}})}

\item[ ]$=\varepsilon_{A}\ot (\tau\co (\eta_{A}\ot H))$ {\scriptsize ({\blue by (\ref{new-h-2delta}) for $A$)}}

\item[ ]$=\varepsilon_A\ot \varepsilon_H$
{\scriptsize ({\blue by (a4) of Definition \ref{skewpairing}})}.

\end{itemize}

 Moreover, if $\lambda_{H}=l_H\co (H\ot \eta_H)$,
 the morphism $\overline{\tau}=\tau\co (\lambda_{A}\ot \lambda_{H})$ satisfies $\tau^{-1}\ast \overline{\tau}=\varepsilon_A\ot \varepsilon_H$. Indeed,

\begin{itemize}
\item[ ]$\hspace{0.38cm} \tau^{-1}\ast \overline{\tau}$

\item[ ]$=(\tau \ot \tau)\co (A\ot c_{A,H}\ot H)\co ((c_{A,A}\co (\lambda_{A}\ot \lambda_{A})\co c_{A,A}\co\delta_{A})\ot ((H\ot \lambda_{H})\co\delta_{H}))$ {\scriptsize ({\blue by naturality of $c$})}

\item[ ]$=(\tau \ot \tau)\co (A\ot c_{A,H}\ot H)\co (( c_{A,A}\co\delta_{A}\co \lambda_{A})\ot ((H\ot \lambda_{H})\co\delta_{H}))$ {\scriptsize ({\blue by (\ref{anticomul}) for $A$})}

\item[ ]$=\tau\co (\lambda_{A}\ot (id_{H}\ast \lambda_{H}))$
{\scriptsize ({\blue by (a2') of Definition \ref{skewpairing}})}

\item[ ]$=\tau\co (\lambda_{A}\ot \varepsilon_{H}\ot \eta_{H})$
{\scriptsize ({\blue by (\ref{new-h-2delta}) for $H$})}

\item[ ]$=(\varepsilon_{A}\co \lambda_{A})\ot\varepsilon_{H}$
{\scriptsize ({\blue by (a3) of Definition \ref{skewpairing}})}

\item[ ]$=\varepsilon_A\ot \varepsilon_H$
{\scriptsize ({\blue by (\ref{lambda-vareps}) for $A$})}.

\end{itemize}

As a consequence, $\tau^{-1}=\tau\co (\lambda_{A}\ot H)$ is the convolution inverse of $\tau$ because 
$$\tau=\tau\ast (\tau^{-1}\ast \overline{\tau})=(\tau\ast \tau^{-1})\ast \overline{\tau}=\overline{\tau}.$$

Thus 
\begin{equation}
\label{igualdadtau}
\tau=\tau^{-1}\co (A\ot \lambda_{H}).
\end{equation}

It is not difficult to obtain the equalities (\ref{skewpairing-1}) and (\ref{skewpairing-2}) because 

\begin{itemize}
\item[ ]$\hspace{0.38cm} \tau^{-1}\co (\eta_{A}\ot H)$

\item[ ]$=\tau\co ( (\lambda_{A}\co \eta_{A})\ot H))$ {\scriptsize ({\blue by definition of $\tau^{-1}$})}

\item[ ]$=\tau\co ( \eta_{A}\ot H))$
{\scriptsize ({\blue by (\ref{lambda-eta}) for $A$})}

\item[ ]$=\varepsilon_H$
{\scriptsize ({\blue by (a4) of Definition \ref{skewpairing}})}, 

\end{itemize}

and 

\begin{itemize}
\item[ ]$\hspace{0.38cm} \tau^{-1}\co (A\ot  \eta_{H})$

\item[ ]$=\tau\co ( \lambda_{A}\ot \eta_{H}))$ {\scriptsize ({\blue by definition of $\tau^{-1}$})}

\item[ ]$=\varepsilon_{A}\co \lambda_{A}$
{\scriptsize ({\blue by (a4) of Definition \ref{skewpairing}})}

\item[ ]$=\varepsilon_A$
{\scriptsize ({\blue by (\ref{lambda-vareps}) for $A$})}.

\end{itemize}

Finally, the proof for (\ref{skewpairing-3}) is the following:

\begin{itemize}
\item[ ]$\hspace{0.38cm} \tau^{-1}\co (A\ot \mu_{H})$

\item[ ]$=(\tau \ot \tau)\co (A\ot c_{A,H}\ot H)\co (( c_{A,A}\co\delta_{A}\co \lambda_{A})\ot H\ot H)$ {\scriptsize ({\blue by (a2') of Definition \ref{skewpairing})}}

\item[ ]$=(\tau \ot \tau)\co (A\ot c_{A,H}\ot H)\co (( c_{A,A}\co (\lambda_{A}\ot \lambda_{A})\co c_{A,A}\co\delta_{A})\ot H\ot H)$
{\scriptsize ({\blue by  (\ref{anticomul}) for $A$})}

\item[ ]$=(\tau^{-1}\ot \tau^{-1})\co (A\ot c_{A,H}\ot H)\co (\delta_{A}\ot H\ot H)$
{\scriptsize ({\blue by naturality of $c$ and $c^{2}=id$})}.

\end{itemize}

\end{proof}

\begin{proposition}
\label{propiedadesskew-r}
Let $A$, $H$ be non-associative bimonoids with right division $r_A$ and $r_H$, respectively. Let $\tau:A\ot H\rightarrow K$ be a skew pairing. Then $\tau$ is convolution invertible. Moreover, if $\tau^{-1}$ is the inverse of $\tau$, the equalities
\begin{equation}
\label{skewpairing-1-r}
\tau^{-1}\co (\eta_A\ot H)=\varepsilon_H,
\end{equation}
\begin{equation}
\label{skewpairing-2-r}
\tau^{-1}\co (A\ot \eta_H)=\varepsilon_A,
\end{equation}
and
\begin{equation}
\label{skewpairing-3-r}
\tau^{-1}\co (A\ot \mu_{H})=(\tau^{-1}\ot \tau^{-1})\co (A\ot c_{A,H}\ot H)\co (\delta_{A}\ot H\ot H),
\end{equation}
hold.

\end{proposition}

\begin{proof} The proof is similar to the one given in the previous proposition but defining $\tau^{-1}$ as $\tau^{-1}=\tau\co (\varrho_{A}\ot H)$ where $\varrho_{A}=r_A\co (\eta_A\ot A).$
\end{proof}

\begin{remark}
Note that, in the conditions of Proposition \ref{propiedadesskew},  we obtain that 
\begin{equation}
\label{tau1}
\tau=\tau\co (\lambda_{A}\ot \lambda_{H}) 
\end{equation}
and 
\begin{equation}
\label{tau2}
\tau^{-1}=\tau^{-1}\co (\lambda_{A}\ot \lambda_{H}). 
\end{equation}

Similarly, in the conditions of Proposition \ref{propiedadesskew-r}, for right divisions we have: 
\begin{equation}
\label{tau1-r}
\tau=\tau\co (\varrho_{A}\ot \varrho_{H}) 
\end{equation}
and 
\begin{equation}
\label{tau2-r}
\tau^{-1}=\tau^{-1}\co (\varrho_{A}\ot \varrho_{H}). 
\end{equation}
\end{remark}

\begin{proposition}
\label{propiedadskewnova}
Let $A$, $H$ be non-associative bimonoids with left division $l_A$ and $l_H$, respectively. Let $\tau:A\ot H\rightarrow K$ be a skew pairing.  If $\lambda_{H}=l_H\co (H\ot \eta_H)$ is an isomorphism the equality
\begin{equation}
\label{skewpairing-4}
\tau^{-1}\co (\mu_A\ot H)=(\tau^{-1}\ot \tau^{-1})\co (A\ot c_{A,H}\ot H)\co (A\ot A\ot (c_{H,H}\co \delta_H))
\end{equation}
holds, where $\tau^{-1}$ is the morphism defined in Proposition \ref{propiedadesskew}. Moreover, if $A$ is a Hopf quasigroup (\ref{skewpairing-4}) holds for any non-associative bimonoid $H$ with left division.

\end{proposition}

\begin{proof}
By composing with the isomorphism $A\ot A\ot \lambda_{H}$ in the left side of (\ref{skewpairing-4}), we have 

\begin{itemize}
\item[ ]$\hspace{0.38cm} \tau^{-1}\co (\mu_A\ot \lambda_{H})$

\item[ ]$=\tau\co (\mu_A\ot H)$
{\scriptsize ({\blue by (\ref{igualdadtau})}}

\item[ ]$=(\tau\ot \tau)\co (A\ot c_{A,H}\ot H)\co (A\ot A\ot\delta_H)$
{\scriptsize ({\blue by (a1) of Definition \ref{skewpairing}})}

\item[ ]$=((\tau^{-1}\co (A\ot \lambda_{H}))\ot (\tau^{-1}\co (A\ot \lambda_{H})))\co (A\ot c_{A,H}\ot H)\co (A\ot A\ot \delta_{H})$
{\scriptsize ({\blue by (\ref{igualdadtau})}}

\item[ ]$=(\tau^{-1}\ot \tau^{-1})\co (A\ot c_{A,H}\ot H)\co (A\ot A\ot (c_{H,H}\co (\lambda_{H}\ot \lambda_{H})\co c_{H,H}\co \delta_H))$
{\scriptsize ({\blue by naturality of $c$ and}}
\item[ ]$\hspace{0.38cm}${\scriptsize {\blue  coassociativity})}

\item[ ]$=(\tau^{-1}\ot \tau^{-1})\co (A\ot c_{A,H}\ot H)\co (A\ot A\ot (c_{H,H}\co \delta_H\co \lambda_{H}))$ {\scriptsize ({\blue by (\ref{anticomul}) and and $c^{2}=id$})}.
\end{itemize}

Therefore, (\ref{skewpairing-4}) holds.

Finally, if we assume that $A$ is a Hopf quasigroup, the antipode $\lambda_A$ is antimultiplicative. Then condition (\ref{skewpairing-4}) is true without the assumption of that $\lambda_{H}$ be an isomorphism because 
 
\begin{itemize}
\item[ ]$\hspace{0.38cm} \tau^{-1}\co (\mu_A\ot H)$

\item[ ]$=\tau\co ((\mu_{A}\co (\lambda_{A}\ot \lambda_{A})\co c_{A,A})\ot H)$ {\scriptsize ({\blue by  (\ref{lambda-anti})})}

\item[ ]$=(\tau^{-1} \ot \tau^{-1})\co (A\ot c_{A,H}\ot H)\co (c_{A,A}\ot \delta_{H})$
{\scriptsize ({\blue by  (a1) of Definition \ref{skewpairing}})}

\item[ ]$=(\tau^{-1}\ot \tau^{-1})\co (A\ot c_{A,H}\ot H)\co (A\ot A\ot (c_{H,H}\co \delta_H))$
{\scriptsize ({\blue by naturality of $c$ and $c^2=id$})}.

\end{itemize}

\end{proof}

Similarly we can prove the result for non-associative bimonoids with right division.

\begin{proposition}
\label{propiedadskewnova-r}
Let $A$, $H$ be non-associative bimonoids with right division $r_A$ and $r_H$, respectively. Let $\tau:A\ot H\rightarrow K$ be a skew pairing.  If $\varrho_{H}=r_H\co (\eta_H\ot H)$ is an isomorphism the equality
\begin{equation}
\label{skewpairing-4-r}
\tau^{-1}\co (\mu_A\ot H)=(\tau^{-1}\ot \tau^{-1})\co (A\ot c_{A,H}\ot H)\co (A\ot A\ot (c_{H,H}\co \delta_H))
\end{equation}
holds, where $\tau^{-1}$ is the morphism defined in  Proposition \ref{propiedadesskew-r}. 

Moreover, if $A$ is a Hopf quasigroup (\ref{skewpairing-4-r}) holds for any non-associative bimonoid $H$ with right division.

\end{proposition}

The following Propositions give the connection between skew pairings and two-cocycles.

\begin{proposition}
\label{Hopfpairingcocycle}
Let $A$, $H$ be non-associative bimonoids with left division  $l_A$ and $l_H$, respectively. Then  $A\ot H=(A\ot H, \eta_{A\ot H}, \mu_{A\ot H}, \varepsilon_{A\ot H}, \delta_{A\ot H})$ is a non-associative bimonoid with left division $l_{A\ot H}=(l_{A}\ot l_{H})\co (A\ot c_{H,A}\ot H)$. If $A$, $H$ are left Hopf quasigroups with left antipodes $\lambda_{A}$, $\lambda_{H}$ respectively, $A\ot H$ is a left Hopf quasigroup with left antipode $\lambda_{A\ot H}=\lambda_{A}\ot \lambda_{H}$.

Moreover, let $\tau:A\ot H\rightarrow K$ be a skew 
pairing. The morphism 
$\omega=\varepsilon_A\ot (\tau\co c_{H,A})\ot \varepsilon_H$ is a normal two-cocycle with convolution inverse $\omega^{-1}=\varepsilon_A\ot (\tau^{-1}\co c_{H,A})\ot \varepsilon_H$, where $\tau^{-1}$ is defined as in Proposition \ref{propiedadesskew}.

\end{proposition}

\begin{proof}
Trivially, $A\ot H$  is a non-associative bimonoid. The morphism $l_{A\ot H}=(l_{A}\ot l_{H})\co (A\ot c_{H,A}\ot H)$ is a left division for $A\ot H$ because

\begin{itemize}
\item[ ]$\hspace{0.38cm} l_{A\ot H} \co (A\ot H\ot \mu_{A\ot H})\co (\delta_{A\ot H}\ot A\ot H)$

\item[ ]$=((l_{A}\co (A\ot \mu_{A})\co (\delta_{A}\ot A))\ot (l_{H}\co (H\ot \mu_{H})\co (\delta_{H}\ot H)))\co (A\ot c_{H,A}\ot H)$  {\scriptsize ({\blue by naturality of $c$ and }}
\item[ ]$\hspace{0.38cm}${\scriptsize {\blue  $c^2=id$})}

\item[ ]$=(\varepsilon_{A}\ot A\ot \varepsilon_{H}\ot H)\co (A\ot c_{H,A}\ot H)$
{\scriptsize ({\blue by (\ref{leftdivision}) for $A$ and $H$})}

\item[ ]$=\varepsilon_{A\ot H}\ot A\ot H$
{\scriptsize ({\blue by naturality of $c$})},

\end{itemize}

and 

\begin{itemize}
\item[ ]$\hspace{0.38cm} \mu_{A\ot H} \co (A\ot H\ot l_{A\ot H})\co (\delta_{A\ot H}\ot A\ot H)$

\item[ ]$=((\mu_{A}\co (A\ot l_{A})\co (\delta_{A}\ot A))\ot (\mu_{H}\co (H\ot l_{H})\co (\delta_{H}\ot H)))\co (A\ot c_{H,A}\ot H)$ {\scriptsize ({\blue by naturality of $c$ and }}
\item[ ]$\hspace{0.38cm}${\scriptsize {\blue  $c^2=id$})}

\item[ ]$=(\varepsilon_{A}\ot A\ot \varepsilon_{H}\ot H)\co (A\ot c_{H,A}\ot H)$
{\scriptsize ({\blue by (\ref{leftdivision}) for $A$ and $H$})}

\item[ ]$=\varepsilon_{A\ot H}\ot A\ot H$
{\scriptsize ({\blue by naturality of $c$})}.

\end{itemize}

If $A$, $H$ are left Hopf quasigroups with left antipodes $\lambda_{A}$, $\lambda_{H}$, by Proposition \ref{igualdadesalpha}, we have that 
$$\lambda_{A\ot H}=l_{A\ot H}\co (A\ot H\ot \eta_{A\ot H})=\lambda_{A}\otimes \lambda_{H}.$$ 

Then, $A\ot H$ is a left Hopf quasigroup because 
\begin{itemize}
\item[ ]$\hspace{0.38cm} \mu_{A\ot H} \co ((\lambda_{A}\ot \lambda_{H})\ot \mu_{A\ot H})\co (\delta_{A\ot H}\ot A\ot H)$

\item[ ]$=((\mu_{A}\co (\lambda_{A}\ot \mu_{A})\co (\delta_{A}\ot A))\ot (\mu_{H}\co (\lambda_{H}\ot \mu_{H})\co (\delta_{H}\ot H)))\co (A\ot c_{H,A}\ot H)$ {\scriptsize ({\blue by naturality of $c$}}
\item[ ]$\hspace{0.38cm}${\scriptsize {\blue  and $c^2=id$})}

\item[ ]$=(\varepsilon_{A}\ot A\ot \varepsilon_{H}\ot H)\co (A\ot c_{H,A}\ot H)$
{\scriptsize ({\blue by (\ref{quasiasociativity}) for $A$ and $H$})}

\item[ ]$=\varepsilon_{A\ot H}\ot A\ot H$
{\scriptsize ({\blue by naturality of $c$})}.

\end{itemize}

Let $\tau:A\ot H\rightarrow K$ be a skew pairing. Then $\omega=\varepsilon_A\ot (\tau\co c_{H,A})\ot \varepsilon_H$ is a two-cocycle. Indeed, in one hand we have 
\begin{itemize}
\item[ ]$\hspace{0.38cm}\partial^1(\omega)\ast \partial^3(\omega) $

\item[ ]$= \varepsilon_{A}\ot (\tau\co c_{H,A}\co (H\ot (\mu_{A}\co (A\ot (\tau\co c_{H,A})\ot A)\co (A\ot H\ot \delta_{A}))\ot \varepsilon_{H}
$ {\scriptsize ({\blue by naturality of $c$, counit }}
\item[ ]$\hspace{0.38cm}${\scriptsize {\blue properties  and (\ref{mu-eps})})}

\item[ ]$= \varepsilon_{A}\ot (\tau\co (\mu_{A}\ot H)\co (A\ot c_{H,A})\co (c_{H,A}\ot ((\tau\co c_{H,A})\ot A)\co (H\ot \delta_{A})))\ot \varepsilon_{H}
$ {\scriptsize ({\blue by naturality of $c$)})}

\item[ ]$=\varepsilon_{A}\ot ((\tau\ot \tau)\co (A\ot c_{A,H}\ot H)\co (A\ot A\ot \delta_{H})\co (A\ot c_{H,A})\co (c_{H,A}\ot ((\tau\co c_{H,A})\ot A)\co (A\ot \delta_{A})))\ot \varepsilon_{H}$ 
\item[ ]$\hspace{0.38cm}$ {\scriptsize ({\blue by (a1) of Definition \ref{skewpairing}})}

\item[ ]$=\varepsilon_{A}\ot ((\tau\ot \tau)\co (A\ot (c_{A,H}\co c_{H,A})\ot H)\co (A\ot H\ot c_{H,A})\co (A\ot \delta_{H}\ot A)\co (c_{H,A}\ot ((\tau\co c_{H,A})\ot A)$ 
\item[ ]$\hspace{0.38cm}\co (A\ot \delta_{A})))\ot \varepsilon_{H}$ {\scriptsize ({\blue by naturality of $c$})}

\item[ ]$=\varepsilon_{A}\ot ((\tau\ot \tau)\co (A\ot H\ot c_{H,A})\co (A\ot \delta_{H}\ot A)\co (c_{H,A}\ot ((\tau\co c_{H,A})\ot A)\co (H\ot \delta_{A})))\ot \varepsilon_{H}$ {\scriptsize ({\blue by  }}
\item[ ]$\hspace{0.38cm}$ {\scriptsize {\blue naturality of $c$ and $c^2=id$}),}
\end{itemize}
and, in the other hand, 
\begin{itemize}
\item[ ]$\hspace{0.38cm} \partial^4(\omega)\ast\partial^2(\omega)$

\item[ ]$=\varepsilon_{A}\ot (((\tau\co c_{H,A})\ot (\tau\co c_{H,A}\co (\mu_{H}\ot A)))\co (H\ot c_{H,A}\ot H\ot A)\co (\delta_{H}\ot A\ot H\ot A))\ot \varepsilon_{H} $ {\scriptsize ({\blue by }}
\item[ ]$\hspace{0.38cm}$ {\scriptsize {\blue  naturality of $c$, counit properties, and (\ref{mu-eps})})}

\item[ ]$=\varepsilon_{A}\ot ((\tau\ot (\tau\co (A\ot \mu_{H})\co (c_{H,A}\ot H)))\co (A\ot \delta_{H}\ot A\ot H)\co (c_{H,A}\ot c_{H,A}))\ot \varepsilon_{H} $ {\scriptsize ({\blue by naturality}}
\item[ ]$\hspace{0.38cm}$ {\scriptsize {\blue  of $c$})}

\item[ ]$=\varepsilon_{A}\ot ((\tau\ot ((\tau \ot \tau)\co (A\ot c_{A,H}\ot H)\co ((c_{A,A}\co \delta_{A})\ot H\ot H)\co (c_{H,A}\ot H)))\co (A\ot \delta_{H}\ot A\ot H)$
\item[ ]$\hspace{0.38cm}\co (c_{H,A}\ot c_{H,A}))\ot \varepsilon_{H} $ {\scriptsize ({\blue by (a2') of Definition \ref{skewpairing}})}

\item[ ]$= \varepsilon_{A}\ot ((\tau\ot \tau)\co (A\ot H\ot c_{H,A})\co (A\ot \delta_{H}\ot A)\co (c_{H,A}\ot ((\tau\co c_{H,A})\ot A)\co (H\ot \delta_{A})))\ot \varepsilon_{H}$ {\scriptsize ({\blue by  }}
\item[ ]$\hspace{0.38cm}$ {\scriptsize {\blue naturality of $c$ and  $c^2=id$})}.

\end{itemize}

Finally, $\omega$ is convolution invertible because 

\begin{itemize}
\item[ ]$\hspace{0.38cm} \omega\ast \omega^{-1}$

\item[ ]$=\varepsilon_{A}\ot ((\tau\ast \tau^{-1})\co c_{H,A})\ot \varepsilon_{H}$ {\scriptsize ({\blue by naturality of $c$ and counit properties})}

\item[ ]$=\varepsilon_{A\ot H}\ot \varepsilon_{A\ot H}$
{\scriptsize ({\blue by invertibility of $\tau$})}, 

\end{itemize}
 and similarly 
\begin{itemize}
\item[ ]$\hspace{0.38cm} \omega^{-1}\ast \omega$

\item[ ]$=\varepsilon_{A}\ot ((\tau^{-1}\ast \tau)\co c_{H,A})\ot \varepsilon_{H}$ {\scriptsize ({\blue by naturality of $c$ and counit properties})}

\item[ ]$=\varepsilon_{A\ot H}\ot \varepsilon_{A\ot H}$
{\scriptsize ({\blue by invertibility of $\tau$})}.

\end{itemize}

\end{proof}

\begin{proposition}
\label{Hopfpairingcocycle-r}
Let $A$, $H$ be non-associative bimonoids with right division  $r_A$ and $r_H$, respectively. Then  $A\ot H=(A\ot H, \eta_{A\ot H}, \mu_{A\ot H}, \varepsilon_{A\ot H}, \delta_{A\ot H})$ is a non-associative bimonoid with left division $r_{A\ot H}=(r_{A}\ot r_{H})\co (A\ot c_{H,A}\ot H)$. If $A$, $H$ are right Hopf quasigroups with right antipodes $\varrho_{A}$, $\varrho_{H}$ respectively, $A\ot H$ is a right Hopf quasigroup with right antipode $\varrho_{A\ot H}=\varrho_{A}\ot \varrho_{H}$.

Moreover, let $\tau:A\ot H\rightarrow K$ be a skew 
pairing. The morphism 
$\omega=\varepsilon_A\ot (\tau\co c_{H,A})\ot \varepsilon_H$ is a two-cocycle with convolution inverse $\omega^{-1}=\varepsilon_A\ot (\tau^{-1}\co c_{H,A})\ot \varepsilon_H$, where $\tau^{-1}$ is defined as in Proposition \ref{propiedadesskew-r}.

\end{proposition}

\begin{proof}
Similar to the one of Proposition \ref{Hopfpairingcocycle}.
\end{proof}

As a corollary of Propositions \ref{Hopfpairingcocycle} and \ref{Hopfpairingcocycle-r} we have 

\begin{corollary}
Let $A$, $H$ be  Hopf quasigroups with  antipodes $\lambda_{A}$, $\lambda_{H}$ respectively. Then  $A\ot H=(A\ot H, \eta_{A\ot H}, \mu_{A\ot H}, \varepsilon_{A\ot H}, \delta_{A\ot H})$ is a Hopf quasigroup with  antipode $\lambda_{A\ot H}=\lambda_{A}\ot \lambda_{H}$.

Moreover, let $\tau:A\ot H\rightarrow K$ be a skew 
pairing. The morphism 
$\omega=\varepsilon_A\ot (\tau\co c_{H,A})\ot \varepsilon_H$ is a two-cocycle with convolution inverse $\omega^{-1}=\varepsilon_A\ot (\tau^{-1}\co c_{H,A})\ot \varepsilon_H$, where $\tau^{-1}$ is defined as in Proposition \ref{propiedadesskew} (or as in Proposition \ref{propiedadesskew-r}).
\end{corollary}

Also, we get the following result which is a generalization of the one given in \cite{FT}, Proposition 2.2. (There is a slightly difference because the definition of Hopf pairing in \cite{FT} corresponds with our notion of pairing betwen $A$ and $H^{cop}$).
\begin{corollary}
\label{FangTorrecillas}
Let $A$, $H$ be left Hopf quasigroups  with left antipodes $\lambda_{A}$, $\lambda_{H}$ respectively. Let $\tau:A\ot H\rightarrow K$ be a skew pairing. Then $A\bowtie_{\tau} H=(A\ot H, \eta_{A\bowtie_{\tau} H}, \mu_{A\bowtie_{\tau} H}, \varepsilon_{A\bowtie_{\tau} H}, \delta_{A\bowtie_{\tau} H}, \lambda_{A\bowtie_{\tau} H})$ has a structure of left Hopf quasigroup, where
$$\eta_{A\bowtie_{\tau} H}=\eta_{A\otimes H},\;\; \varepsilon_{A\bowtie_{\tau} H}=\varepsilon_{A\otimes H},\;\; \delta_{A\bowtie_{\tau} H}=\delta_{A\otimes H},$$
\begin{equation}
\label{mutimes}
\mu_{A\bowtie_{\tau} H}=(\mu_A\ot\mu_H)\co (A\ot \tau\ot A\ot H\ot \tau^{-1}\ot H)\co (A\ot \delta_{A\ot H}\ot A\ot H\ot H)\co(A\ot \delta_{A\ot H}\ot H)\co (A\ot c_{H,A}\ot H)
\end{equation}
and 
\begin{equation}
\label{lambdatimes}
\lambda_{A\bowtie_{\tau} H}=(\tau^{-1}\ot \lambda_A\ot \lambda_H\ot \tau)\co (A\ot H\ot \delta_{A\ot H})\co \delta_{A\ot H}.
\end{equation}

\end{corollary}

\begin{proof}
The result follows  by application of Theorem \ref{construccion} to the left Hopf quasigroup $A\otimes H$ and the two-cocycle $\omega=\varepsilon_A\ot (\tau\co c_{H,A})\ot \varepsilon_H$. Using the naturality of $c$, the counit properties, the coassociativity of the coproducts,  and (\ref{lambda-vareps}), it is easy to check that 
$$\mu_{A\bowtie_{\tau} H}=\mu_{{(A\otimes H)}^{\omega}},\;\;\; \lambda_{A\bowtie_{\tau} H}=\lambda_{{(A\otimes H)}^{\omega}}.$$ 

\end{proof}

Similarly, for right Hopf quasigroups we have 

\begin{corollary}
\label{FangTorrecillas-r}
Let $A$, $H$ be right Hopf quasigroups  with right antipodes $\varrho_{A}$, $\varrho_{H}$ respectively. Let $\tau:A\ot H\rightarrow K$ be a skew pairing. Then $A\bowtie_{\tau} H=(A\ot H, \eta_{A\bowtie_{\tau} H}, \mu_{A\bowtie_{\tau} H}, \varepsilon_{A\bowtie_{\tau} H}, \delta_{A\bowtie_{\tau} H}, \varrho_{A\bowtie_{\tau} H})$ has a structure of right Hopf quasigroup, where
$$\eta_{A\bowtie_{\tau} H}=\eta_{A\otimes H},\;\; \varepsilon_{A\bowtie_{\tau} H}=\varepsilon_{A\otimes H},\;\; \delta_{A\bowtie_{\tau} H}=\delta_{A\otimes H},$$
\begin{equation}
\label{mutimes-r}
\mu_{A\bowtie_{\tau} H}=(\mu_A\ot\mu_H)\co (A\ot \tau\ot A\ot H\ot \tau^{-1}\ot H)\co (A\ot \delta_{A\ot H}\ot A\ot H\ot H)\co(A\ot \delta_{A\ot H}\ot H)\co (A\ot c_{H,A}\ot H)
\end{equation}
and 
\begin{equation}
\label{lambdatimes-r}
\varrho_{A\bowtie_{\tau} H}=(\tau^{-1}\ot \varrho_A\ot \varrho_H\ot \tau)\co (A\ot H\ot \delta_{A\ot H})\co \delta_{A\ot H}.
\end{equation}

\end{corollary}

Therefore, as a consequence of Corollaries \ref{FangTorrecillas} and \ref{FangTorrecillas-r} we obtain the following result:

\begin{corollary}
\label{FangTorrecillas-q}
Let $A$, $H$ be  Hopf quasigroups  with  antipodes $\lambda_{A}$, $\lambda_{H}$ respectively. Let $\tau:A\ot H\rightarrow K$ be a skew pairing. Then $A\bowtie_{\tau} H=(A\ot H, \eta_{A\bowtie_{\tau} H}, \mu_{A\bowtie_{\tau} H}, \varepsilon_{A\bowtie_{\tau} H}, \delta_{A\bowtie_{\tau} H}, \lambda_{A\bowtie_{\tau} H})$ has a structure of  Hopf quasigroup, where
$$\eta_{A\bowtie_{\tau} H}=\eta_{A\otimes H},\;\; \varepsilon_{A\bowtie_{\tau} H}=\varepsilon_{A\otimes H},\;\; \delta_{A\bowtie_{\tau} H}=\delta_{A\otimes H},$$
\begin{equation}
\label{mutimes-q}
\mu_{A\bowtie_{\tau} H}=(\mu_A\ot\mu_H)\co (A\ot \tau\ot A\ot H\ot \tau^{-1}\ot H)\co (A\ot \delta_{A\ot H}\ot A\ot H\ot H)\co(A\ot \delta_{A\ot H}\ot H)\co (A\ot c_{H,A}\ot H)
\end{equation}
and 
\begin{equation}
\label{lambdatimes-q}
\lambda_{A\bowtie_{\tau} H}=(\tau^{-1}\ot \lambda_A\ot \lambda_H\ot \tau)\co (A\ot H\ot \delta_{A\ot H})\co \delta_{A\ot H}.
\end{equation}

\end{corollary}

\begin{remark}
\label{DobleDrinfeld}
{\rm  When particularizing to the Hopf algebra setting, it is a well-known fact that the Drinfeld double of a Hopf algebra $H$ (roughly speaking, a product involving $H$ and the opposite comonoid of its dual Hopf algebra $H^*$) is an example of a deformation of a Hopf algebra by the two-cocycle associated to a skew pairing. We want to point out that in our context we can not describe the Drinfeld double in this way because the dual of a Hopf quasigroup $H$ is not a Hopf quasigroup but a Hopf coquasigroup.   }
\end{remark}

\begin{example}
\label{prin-ej}
Let ${\Bbb F}$ be a field such that Char(${\Bbb F}$)$\neq 2$ and denote the tensor product over ${\Bbb F}$ as $\ot$. Consider the nonabelian group $S_{3}=\{\sigma_{0}, \sigma_{1}, \sigma_{2}, \sigma_{3}, \sigma_{4}, \sigma_{5}\}$ where $\sigma_{0}$ is  the identity, $o(\sigma_{1})=o(\sigma_{2})=o(\sigma_{3})=2$ and $o(\sigma_{4})=o(\sigma_{5})=3$. Let $u$ be an additional element such that $u^2=1$. By Theorem 1 of \cite{Chein} the set 
$$L=M(S_{3},2)=\{\sigma_{i}u^{\alpha}\;; \; \alpha=0,1\}$$
is a Moufang loop where the product is defined by 
$$\sigma_{i}u^{\alpha}.\;\sigma_{j}u^{\beta}=(\sigma_{i}^{\nu}\sigma_{j}^{\mu})^{\nu}u^{\alpha +\beta}.$$
$$\nu=(-1)^{\beta}, \; \mu=(-1)^{\alpha +\beta}.$$

Then, $L$ is an IP loop and by example \ref{IPloops}, $A={\Bbb F}L$ is a cocommutative Hopf quasigroup. 

On the other hand, let $H_{4}$ be the $4$-dimensional Taft Hopf algebra. This Hopf algebra is the smallest non commutative, non cocommutative Hopf algebra. The basis of $H_{4}$ is $\{1,x,y,w=xy\}$ and the multiplication table is defined by 
\begin{center}
\begin{tabular}{|c|c|c|c|c|}
\hline  $\;$ & $x$ & $y$ & $w$   \\
\hline  $ x$ &  $1$ & $w$ & $y$ \\
\hline  $ y$ &  $-w$ & $0$ & $0$ \\
\hline  $ w$ &  $-y$ & $0$ & $0$ \\
\hline
\end{tabular}
\end{center}

The costructure of $H_{4}$ is given by 
$$ \delta_{H_{4}}(x)=x\ot x,\; \delta_{H_{4}}(y)=y\ot x +1\ot y,\; \delta_{H_{4}}(w)=w\ot 1 +x\ot w,$$
$$\varepsilon_{H_{4}}(x)=1_{\Bbb F},\; \varepsilon_{H_{4}}(y)=\varepsilon_{H_{4}}(w)=0, $$
and the antipode $\lambda_{H_{4}}$ is described by 
$$\l \lambda_{H_{4}}(x)=x,\; \lambda_{H_{4}}(y)=w,\; \lambda_{H_{4}}(w)=-y.$$

By Proposition \ref{Hopfpairingcocycle}, $A\ot H_{4}$ is a non commutative, non cocommutative Hopf quasigroup and the morphism $\tau:A\ot H_{4}\rightarrow {\Bbb F}$ defined by 
$$\tau (\sigma_{i}u^{\alpha}\ot z)=\left\{ \begin{array}{ccc} 1 & {\rm if} &
z=1 \\
 (-1)^{\alpha} & {\rm if} &
z=x \\ 
 0 & {\rm if} &
z=y, w 
 \end{array}\right.$$
 is a skew pairing such that $\tau=\tau^{-1}$. Then, by Proposition \ref{Hopfpairingcocycle},  
 $$\omega= \varepsilon_A\ot (\tau\co c_{H_{4},A})\ot \varepsilon_{H_{4}}$$ is a two-cocycle with convolution inverse $\omega^{-1}=\omega$. Finally, $A\bowtie_{\tau}H_{4}$ is Hopf quasigroup isomorphic to 
 $(A\ot H_{4})^{\omega}$.

\end{example}

\section{Double crossproducts and skew pairings}

In this section we will show that the construction of $A\bowtie_{\tau} H$ introduced in the previous section is also a special case of the double cross product defined in \cite{M}. First of all, we need to recall some definitions, following  \cite{Brz}, \cite {BrzezJiao1} and \cite{jeniffer}, to  state a characterization of double cross products in the quasigroup setting.

\begin{definition}
\label{quasimodulo-comodulo}
 {\rm
Let  $H$ be a left Hopf quasigroup. We say that $(M,\varphi_{M})$ is a
left $H$-quasimodule if $M$ is an object in ${\mathcal C}$ and
$\varphi_{M}:H\ot M\rightarrow M$ is a morphism in ${\mathcal C}$ (called the action) satisfying
\begin{equation}
\label{unidadquasimodulo}
\varphi_{M}\circ(\eta_{H}\ot M)=id_{M},
\end{equation}
\begin{equation}
\label{productoquasimodulo}
\varphi_{M}\co (H\ot\varphi_{M})\co  (((H\ot \lambda_{H})\co \delta_H)\ot M)=\varepsilon_H\ot M=\varphi_{M}\co (\lambda_H\ot \varphi_M)\co (\delta_H\ot M).
\end{equation}

Given two left ${H}$-quasimodules $(M,\varphi_{M})$
and $(N,\varphi_{N})$, $f:M\rightarrow N$ is a morphism of left
$H$-quasimodules if 
\begin{equation}
\label{morphism-mod}
\varphi_{N}\co(H\ot f)=f\co\varphi_{M}.
\end{equation}

We denote the category of left $H$-quasimodules by $_{H}{\mathcal {QC}}$.

If $(M,\varphi_{M})$ and $(N,\varphi_{N})$ are left $H$-quasimodules, the tensor product 
 $M\ot N$ is a left $H$-quasimodule with the diagonal action 
 $$\varphi_{M\otimes
N}=(\varphi_{M}\ot\varphi_{N})\circ
(H\ot c_{H,M}\ot N)\co(\delta_{H}\ot M\ot N).$$

This makes the category of left $H$-quasimodules into a strict monoidal category $(_{H}{\mathcal {QC}}, \ot , K)$ (see Remark 3.3 of \cite{BrzezJiao1}).

We will say that a unital magma $A$ is a left $H$-quasimodule magma  if it is a left $H$-quasimodule with action $\varphi_{A}: H\ot A\rightarrow A$ and the following equalities 
\begin{equation}
\label{etaquasimodulo}
\varphi_{A}\circ(H\ot \eta_A)=\varepsilon_H\ot \eta_A,
\end{equation}
\begin{equation}
\label{muquasimodulo}
\mu_A\co \varphi_{A\ot A}=\varphi_A\co (H\ot \mu_A),
\end{equation}
hold, i.e., $\varphi_{A}$ is a morphism of unital magmas.

A comonoid $A$ is a left $H$-quasimodule comonoid if it is a left $H$-quasimodule with action $\varphi_{A}$ and 
\begin{equation}
\label{epsilonquasimodulo}
\varepsilon_A\co \varphi_{A}=\varepsilon_H\ot \varepsilon_A,
\end{equation}
\begin{equation}
\label{deltaquasimodulo}
\delta_A\co \varphi_{A}=\varphi_{A\ot A}\co (\delta_{H}\ot\delta_A),
\end{equation}
hold, i.e., $\varphi_{A}$ is a comonoid morphism.

Replacing (\ref{productoquasimodulo}) by the equality 
\begin{equation}
\label{modulo}
\varphi_{M}\co(H\ot\varphi_{M})=\varphi_M\co (\mu_H\ot M),
\end{equation}
we obtain the definition of left $H$-module  and the ones of left $H$-module magma and comonoid (because (\ref{productoquasimodulo}) follows trivially from (\ref{modulo})). Note that $(H, \mu_H)$ is not an $H$-module but an $H$-quasimodule. The morphism between left $H$-modules is defined as for $H$-quasimodules and we denote the category of left $H$-modules by $\;_{H}{\mathcal C}$.  Note that for the right side we have similar definitions.

}
\end{definition}

\begin{proposition}
\label{doublecrossskewpairing}
Let $A$, $H$ be  (right) left  Hopf quasigroups and let $\tau:A\ot H\rightarrow K$ be a skew pairing. Define $\varphi_A:H\ot A\rightarrow A$ and $\phi_H:H\ot A\rightarrow H$ as $$\varphi_A=(\tau\ot A\ot \tau^{-1})\co (A\ot H\ot \delta_A\ot H)\co \delta_{A\ot H}\co c_{H,A}$$ and $$\phi_H=(\tau\ot H\ot \tau^{-1})\co (A\ot H\ot c_{A,H}\ot H)\co (A\ot H\ot A\ot \delta_H)\co \delta_{A\ot H}\co c_{H,A}.$$ 

Then
\begin{itemize}
\item[(i)] The pair $(A, \varphi_A)$ is a left $H$-module comonoid.
\item[(ii)] If the (right) left antipode of $H$ is an isomorphism, the pair $(H, \phi_H)$ is a right $A$-module 
comonoid.
\item[(iii)] If  $H$ is a Hopf quasigroup, the pair $(H, \phi_H)$ is a right $A$-module 
comonoid.

\end{itemize} 
\end{proposition}

\begin{proof}
We prove the result for left Hopf quasigroups. The proof for right Hopf quasigroups is similar and is left to the reader. Trivially, by (\ref{delta-eta}) for $H$, (a3) of Definition \ref{skewpairing}, (\ref{skewpairing-2}), and the counit properties we obtain $\varphi_A\co (\eta_H\ot A)= id_A$. The equality $\varepsilon_A\co \varphi_A=\varepsilon_H\ot \varepsilon_A$ follows by 
the counit properties, the invertibility of $\tau$ and the naturality of $c$.  Moreover, 
\begin{itemize}
\item[ ]$\hspace{0.38cm} \varphi_A\co (\mu_H\ot A)$

\item[ ]$=((\tau\co (A\ot \mu_H))\ot A\ot (\tau^{-1}\co (A\ot \mu_H)))\co (A\ot H\ot H\ot \delta_A\ot H\ot H)\co \delta_{A\ot H\ot H}\co (c_{H,A}\ot H)$
\item[ ]$\hspace{0.38cm}\co (H\ot c_{H,A})$ {\scriptsize ({\blue by (a2') of Definition \ref{skewpairing} and (\ref{skewpairing-3})})}

\item[ ]$=(((\tau\ot \tau)\co (A\ot c_{A,H}\ot H)\co ((c_{A,A}\co \delta_{A})\ot H\ot H))\ot A\ot ((\tau^{-1}\ot \tau^{-1})\co (A\ot c_{A,H}\ot H)\co (\delta_{A}\ot H\ot H)))$
\item[ ]$\hspace{0.38cm}\co (A\ot H\ot H\ot \delta_{A}\ot H\ot H)\co \delta_{A\ot H\ot H}\co  (c_{H,A}\ot H)\co (H\ot c_{H,A})$ {\scriptsize({\blue by naturality of $c$ and coassociativity})}

\item[ ]$=\varphi_A\co (H\ot \varphi_A)$
 {\scriptsize ({\blue by naturality of $c$ and (\ref{delta-mu})})}.

\end{itemize}

Finally,
\begin{itemize}
\item[ ]$\hspace{0.38cm} (\varphi_A\ot \varphi_A)\co \delta_{H\ot A}$

\item[ ]$=(A\ot (\tau^{-1}\ast\tau)\ot A)\co (A\ot A\ot c_{A,H})\co (A\ot \delta_A\ot H)\co (\delta_A\ot H\ot \tau^{-1})\co (\tau\ot \delta_{A\ot H})\co \delta_{A\ot H}\co c_{H,A}$
\item[ ]$\hspace{0.38cm}$ {\scriptsize ({\blue by naturality of $c$ and coassociativity})}
                                 
\item[ ]$=\delta_A\co\varphi_A$ {\scriptsize ({\blue by naturality of $c$, invertibility of $\tau$, and counit properties}).}

\end{itemize}

The proof for $(ii)$ follows a similar pattern but using (a1) of Definition \ref{skewpairing} and (\ref{skewpairing-4}) instead of (a2')  and (\ref{skewpairing-3}). By Proposition \ref{propiedadskewnova}, we obtain (iii) because  in the quasigroup setting condition (\ref{skewpairing-4}) is true without the assumption of that $\lambda_{H}$ be an isomorphism.

\end{proof}

 The following result is a version of \cite{M}, Theorem 7.2.2 for left Hopf quasigroups (see also \cite{Christian}, Theorem IX.2.3).

\begin{theorem}
\label{productoMajid}
Let $A$, $H$ be left Hopf quasigroups with left antipodes $\lambda_{A}$, $\lambda_{H}$ respectively. Assume that $(A, \varphi_A)$ is a left $H$-module comonoid and $(H, \phi_H)$ is a right $A$-module comonoid. Then the following assertions are equivalent:

\begin{itemize}
\item [(i)] The double crossproduct $A\bowtie H$ built on the object $A\ot H$ with product
$$\mu_{A\bowtie H}=(\mu_A\ot \mu_H)\co (A\ot \varphi_A\ot \phi_H\ot H)\co (A\ot \delta_{H\ot A}\ot H)$$
and tensor product unit, counit and coproduct, is a left Hopf quasigroup with left antipode
$$\lambda_{A\bowtie H}=(\varphi_A\ot \phi_H)\co \delta_{H\ot A}\co (\lambda_H\ot \lambda_A)\co c_{A,H}.$$

\item [(ii)] The  equalities

 \begin{equation}\label{EtaAfi}
\varphi_A\co (H\ot \eta_A)=\varepsilon_H\ot \eta_A,
 \end{equation}
\begin{equation}\label{EtaHpsi}
\phi_H\co (\eta_H\ot A)=\eta_H\ot \varepsilon_A,
 \end{equation}
\begin{equation}\label{fipsicompatibles}
(\phi_H\ot \varphi_A)\co \delta_{H\ot A}=c_{A,H}\co(\varphi_A\ot \phi_H)\co \delta_{H\ot A},
 \end{equation}
\begin{equation}\label{muAfi}
\varphi_A\co (H\ot \mu_A)\co (\lambda_H\ot \lambda_A\ot A)=\mu_A\co(A\ot \varphi_A)\co ((\lambda_{A\bowtie H}\co c_{H,A})\ot A),
 \end{equation}
\begin{equation}\label{primeraqg}
\mu_H\co(\phi_H\ot \mu_H)\co (\lambda_H\ot ((\varphi_A\ot \phi_H)\co \delta_{H\ot A})\ot H)\co  (\delta_H\ot A\ot H)=\varepsilon_H\ot \varepsilon_A\ot H,
 \end{equation}
\begin{equation}\label{segundaqg}
\mu_H\co(\phi_H\ot \mu_H)\co (H\ot ((\varphi_A\ot \phi_H)\co \delta_{H\ot A})\ot H)\co  (((H\ot \lambda_H)\co\delta_H)\ot A\ot H)=\varepsilon_H\ot \varepsilon_A\ot H,
 \end{equation}
hold.

\end{itemize}

\end{theorem}

\begin{proof} 
$(i)\Rightarrow (ii)$  First of all, we have 
\begin{equation}
\label{prin1}
id_{A\ot H}=((\mu_{A}\co (A\ot (\varphi_{A}\co (H\ot \eta_{A}))))\ot H)\co (A\ot \delta_{H})
\end{equation}
because 
\begin{itemize}
\item[ ]$\hspace{0.38cm} id_{A\ot H}$

\item[ ]$=\mu_{A\bowtie H}\co (A\ot H\ot \eta_{A\ot H})$ {\scriptsize ({\blue by unit properties})}

\item[ ]$=((\mu_{A}\co (A\ot (\varphi_{A}\co (H\ot \eta_{A}))))\ot H)\co (A\ot \delta_{H})$ {\scriptsize ({\blue by (\ref{delta-eta}) for $A$, (\ref{unidadquasimodulo}) for $\phi_{H}$, and the properties of $\eta_{A}$})}.

\end{itemize}

Therefore, composing with $A\ot \varepsilon_{H}$ on the left side and with $\eta_{A}\ot H$ on the right sides of  the equality (\ref{prin1}) we get (\ref{EtaAfi}). In a similar way the identity $\mu_{A\bowtie H}\co (\eta_{A\ot H}\ot A\ot H)=id_{A\ot H}$ leads to (\ref{EtaHpsi}). As far as (\ref{fipsicompatibles}), it can be obtained by composing with $\eta_A\ot H\ot A\ot \eta_H$ on the right and with $\varepsilon_A\ot H\ot A\ot \varepsilon_H$ on the left in the two terms of  the equality
$$\delta_{A\ot H}\co\mu_{A\bowtie H}=(\mu_{A\bowtie H}\ot \mu_{A\bowtie H})\co (A\ot H\ot c_{A\ot H,A\ot H}\ot A\ot H)\co (\delta_{A\ot H}\ot\delta_{A\ot H}).$$

Indeed: 

\begin{itemize}
\item[ ]$\hspace{0.38cm} (\phi_H\ot \varphi_A)\co \delta_{H\ot A}$

\item[ ]$= ((((\varepsilon_{A}\co \varphi_{A})\ot \phi_{H})\co \delta_{H\ot A})\ot ((\varphi_{A}\ot (\varepsilon_{H}\co\phi_{H}))\co \delta_{H\ot A}))\co \delta_{A\ot H} $ {\scriptsize ({\blue by (\ref{epsilonquasimodulo}) for $\varphi_{A}$ and $\phi_{H}$, naturality }}
\item[ ]$\hspace{0.38cm}${\scriptsize {\blue of $c$, and counit properties})}

\item[ ]$=(\varepsilon_A\ot H\ot A\ot \varepsilon_H)\co (\mu_{A\bowtie H}\ot \mu_{A\bowtie H})\co (A\ot H\ot c_{A\ot H,A\ot H}\ot A\ot H)\co (\delta_{A\ot H}\ot\delta_{A\ot H})$
\item[ ]$\hspace{0.38cm}\co (\eta_A\ot H\ot A\ot \eta_H)$ {\scriptsize ({\blue by (\ref{delta-eta}) for $A$ and $H$, naturality of $c$, and the properties of $\eta_{A}$ and $\eta_{H}$})}

\item[ ]$=(\varepsilon_A\ot H\ot A\ot \varepsilon_H)\co\delta_{A\ot H}\co \mu_{A\bowtie H}\co (\eta_A\ot H\ot A\ot \eta_H)$ {\scriptsize ({\blue by (\ref{delta-mu}) for $A\bowtie H$})}

\item[ ]$=c_{A,H}\co(\varphi_A\co \phi_H)\co \delta_{H\ot A}$ {\scriptsize ({\blue by naturality of $c$, and the properties of $\eta_{A}$, $\eta_{H}$, $\varepsilon_{A}$ and $\varepsilon_{H}$})}.

\end{itemize}

On the other hand, if $A\bowtie H$ is a left Hopf quasigroup with left antipode $\lambda_{A\bowtie H}$, (\ref{leftHqg}) holds. Then, we have 
\begin{equation}
\label{prin2}
\mu_{A\bowtie H}\circ (\lambda_{A\bowtie H}\ot \mu_{A\bowtie  H})\circ (\delta_{A\otimes H}\ot A\ot H)=
\varepsilon_{A\otimes H}\ot A\ot H.
\end{equation}

Composing with  $A\ot \varepsilon_{H}$ on the left and with $A\ot H\ot A\ot \eta_{H}$ on the right in the two terms of  the equality (\ref{prin2}) we get

\begin{equation}
\label{prin3}
\varepsilon_A\ot\varepsilon_H\ot A=
\end{equation}
$$\mu_A\co (A\ot \varphi_A)\co (((\varphi_{A}\ot \phi_{H})\co \delta_{H\ot A}\co (\lambda_{H}\ot 
\lambda_{A}))\ot A)$$
$$\co (H\ot ((A\ot \mu_A)\co (\delta_A\ot A)))\co (c_{A,H}\ot A)\co (A\ot ((H\ot \varphi_A)\co (\delta_H\ot A))).$$

Indeed:

\begin{itemize}
\item[ ]$\hspace{0.38cm} \varepsilon_A\ot\varepsilon_H\ot A$ 

\item[ ]$ = \varepsilon_A\ot\varepsilon_H\ot A\ot (\varepsilon_{H}\co \eta_{H})$ {\scriptsize ({\blue by (\ref{eta-eps}) for $H$})}

\item[ ]$=(A\ot \varepsilon_H)\co \mu_{A\bowtie H}\co (\lambda_{A\bowtie H}\ot \mu_{A\bowtie H})\co (\delta_{A\ot H}\ot A\ot \eta_H)$ {\scriptsize ({\blue by (\ref{prin2}) for $A\bowtie H$})}

\item[ ]$=\mu_A\co (A\ot \varphi_A)\co (\lambda_{A\bowtie H}\ot \mu_A)\co (A\ot c_{A,H}\ot A)\co (\delta_A\ot ((H\ot \varphi_A)\co (\delta_H\ot A)))$ {\scriptsize ({\blue by (\ref{mu-eps}) for $H$, (\ref{epsilonquasimodulo})}}
\item[ ]$\hspace{0.38cm}${\scriptsize {\blue for $\phi_{H}$, the naturality of $c$ and counit properties})}

\item[ ]$=\mu_A\co (A\ot \varphi_A)\co (((\varphi_{A}\ot \phi_{H})\co \delta_{H\ot A}\co (\lambda_{H}\ot 
\lambda_{A}))\ot A)\co (H\ot ((A\ot \mu_A)\co (\delta_A\ot A)))\co (c_{A,H}\ot A)$
\item[ ]$\hspace{0.38cm} \co (A\ot ((H\ot \varphi_A)\co (\delta_H\ot A)))$ {\scriptsize ({\blue by the naturality of $c$})}.

\end{itemize}

Having into account that $(H\ot \varphi_A)\co (\delta_H\ot A)$ and $(A\ot \mu_A)\co (\delta_A\ot A)$ are isomorphisms with inverses 
$(H\ot \varphi_A)\co (H\ot \lambda_H\ot A)\co(\delta_H\ot A)$ and $(A\ot \mu_A)\co (A\ot \lambda_A\ot A)\co (\delta_A\ot A)$, respectively, we have

\begin{itemize}

\item[ ]$\hspace{0.38cm} \mu_A\co(A\ot \varphi_A)\co ((\lambda_{A\bowtie H}\co c_{H,A})\ot A)$ 
 
\item[ ]$\hspace{0.38cm} \mu_A\co (A\ot \varphi_A)\co (((\varphi_{A}\ot \phi_{H})\co \delta_{H\ot A}\co (\lambda_{H}\ot \lambda_{A}))\ot A)$ {\scriptsize ({\blue by $c^2=id$})}

\item[ ]$=\mu_A\co (A\ot \varphi_A)\co (((\varphi_{A}\ot \phi_{H})\co \delta_{H\ot A}\co (\lambda_{H}\ot 
\lambda_{A}))\ot A)\co (H\ot ((A\ot \mu_A)\co (\delta_A\ot A)))\co (c_{A,H}\ot A)$
\item[ ]$\hspace{0.38cm} \co (A\ot ((H\ot \varphi_A)\co (\delta_H\ot A)))\co (A\ot ((H\ot \varphi_A)\co (H\ot \lambda_H\ot A)\co (\delta_H\ot A)))\co (c_{H,A}\ot A)\co (H\ot ((A\ot \mu_A)$
\item[ ]$\hspace{0.38cm}\co (A\ot \lambda_A\ot A)\co (\delta_A\ot A))$ {\scriptsize ({\blue by composition with the inverses})}

\item[ ]$ =(\varepsilon_{A}\ot  \varepsilon_{H}\ot A)\co (A\ot (((H\ot \varphi_A)\co (H\ot \lambda_{H}\ot A)\co (\delta_H\ot A))))\co (c_{H,A}\ot A)\co (H\ot ((A\ot \mu_A)$
\item[ ]$\hspace{0.38cm}\co  (A\ot \lambda_A\ot A)\co (\delta_A\ot A)))$ {\scriptsize ({\blue by (\ref{prin3})})}

\item[ ]$=\varphi_A\co (H\ot \mu_A)\co (\lambda_H\ot \lambda_A\ot A)$ {\scriptsize ({\blue by naturality of $c$ and counit properties}).}

\end{itemize}

Therefore,  (\ref{muAfi}) holds. Now we show  (\ref{primeraqg}): Composing with  $ \varepsilon_{A}\ot H$ on the left and with $\eta_{A}\ot H\ot A\ot H$ on the right in the two terms of  the equality (\ref{prin2}) we get

\begin{itemize}
\item[ ]$\hspace{0.38cm} \varepsilon_H\ot\varepsilon_A\ot H$

\item[ ]$=\mu_H\co(\phi_H\ot \mu_H)\co (\lambda_H\ot ((\varphi_A\ot \phi_H)\co \delta_{H\ot A})\ot H)\co  (\delta_H\ot A\ot H)$ {\scriptsize ({\blue by (\ref{eta-eps}), (\ref{mu-eps}), (\ref{delta-eta}), (\ref{lambda-eta}) for $A$, (\ref{etaquasimodulo})  }}
\item[ ]$\hspace{0.38cm}${\scriptsize {\blue  (\ref{epsilonquasimodulo}) for $\varphi_{A}$, (\ref{unidadquasimodulo}) for $\phi_{H}$,  naturality of $c$ and counit properties}).}

\end{itemize}

Finally, by (\ref{leftHqg}) for $A\bowtie H$, we have 
\begin{equation}
\label{prin4}
\mu_{A\bowtie H}\circ (A\ot H\ot \mu_{A\bowtie H})\co (A\ot H\ot \lambda_{A\bowtie H}\ot A\ot H)\co  (\delta_{A\otimes H}\ot A\ot H)=\varepsilon_{A\otimes H}\ot A\ot H.
\end{equation}

Then, composing with  $ \varepsilon_{A}\ot H$ on the left and with $\eta_{A}\ot H\ot A\ot H$ on the right in the two terms of  the equality (\ref{prin4}) 
 
\begin{itemize}
\item[ ]$\hspace{0.38cm} \varepsilon_H\ot\varepsilon_A\ot H$

\item[ ]$=\mu_H\co(\phi_H\ot \mu_H)\co (H\ot \varphi_A\co \phi_H\ot H)\co (H\ot \delta_{H\ot A}\ot H)\co (((H\ot \lambda_H)\co\delta_H)\ot A\ot H)$ {\scriptsize ({\blue by (\ref{eta-eps}),}}
\item[ ]$\hspace{0.38cm}${\scriptsize {\blue (\ref{mu-eps}), (\ref{delta-eta}), (\ref{lambda-eta}) for $A$, (\ref{etaquasimodulo}),  (\ref{epsilonquasimodulo}), (\ref{EtaAfi}) for $\varphi_{A}$, (\ref{unidadquasimodulo}) for $\phi_{H}$,  naturality of $c$ and counit properties}).}

\end{itemize}

Therefore, (\ref{segundaqg}) holds.

$(ii)\Rightarrow (i)$ We only prove the equalities involving the left antipode. The proof for the other conditions are analogous to the ones given in \cite{M}, Theorem 7.2.2.

\begin{itemize}
\item[ ]$\hspace{0.38cm} \mu_{A\bowtie H}\co (\lambda_{A\bowtie H}\ot \mu_{A\bowtie H})\co (\delta_{A\ot H}\ot A\ot H)$

\item[ ]$=((\mu_A\co (A\ot \varphi_A))\ot \mu_H)\co (((((\varphi_A\ot \phi_H)\co \delta_{H\ot A})\ot A\ot (\phi_H\co (\phi_H\ot A)))\co \delta_{H\ot A\ot A})\ot H)$
\item[ ]$\hspace{0.38cm} \co (((\lambda_H\ot \lambda_A)\co c_{A,H})\ot \mu_{A\bowtie H})\co (\delta_{A\ot H}\ot A\ot H)$ {\scriptsize ({\blue by the  comonoid morphism condition for $\phi_H$, coassociativity,}}
\item[ ]$\hspace{0.38cm}$ {\scriptsize {\blue and naturality of $c$})}

\item[ ]$=((\mu_A\co(A\ot \varphi_A)\co ((\lambda_{A\bowtie H}\co c_{H,A})\ot A))\ot (\mu_H\co (\phi_H\co (H\ot \mu_A)\co (\lambda_H\ot \lambda_A\ot A))\ot H))$
\item[ ]$\hspace{0.38cm} \co (H\ot ((A\ot c_{H,A}\ot A)\co (c_{H,A}\ot c_{A,A}))\ot A\ot H)\co ((c_{H,H}\co \delta_H)\ot (c_{A,A}\co \delta_A)\ot \delta_A\ot H)\co (c_{A,H}\ot \mu_{A\bowtie H})$
\item[ ]$\hspace{0.38cm} \co (\delta_{A\ot H}\ot A\ot H)$ 
{\scriptsize ({\blue by (\ref{anticomul}) for $\lambda_{H}$ and  $\lambda_{A}$, coassociativity, naturality of 
$c$, and condition of $A$-module for $H$})}

\item[ ]$=((\varphi_A\co (H\ot \mu_A)\co (\lambda_H\ot \lambda_A\ot A))\ot (\mu_H\co (\phi_H\co (H\ot \mu_A)\co (\lambda_H\ot \lambda_A\ot A))\ot H))$
\item[ ]$\hspace{0.38cm} \co (H\ot ((A\ot c_{H,A}\ot A)\co (c_{H,A}\ot c_{A,A}))\ot A\ot H)\co ((c_{H,H}\co \delta_H)\ot (c_{A,A}\co \delta_A)\ot \delta_A\ot H)\co (c_{A,H}\ot \mu_{A\bowtie H})$
\item[ ]$\hspace{0.38cm} \co (\delta_{A\ot H}\ot A\ot H)$ 
{\scriptsize ({\blue by (\ref{muAfi}))})}

\item[ ]$=(A\ot \mu_H)\co (((\varphi_A\ot \phi_H)\co \delta_{H\ot A})\ot H)\co (\lambda_H\ot (\mu_A\co(\lambda_A\ot \mu_A)\co (\delta_A\ot A))\ot \mu_H)$
\item[ ]$\hspace{0.38cm} \co (c_{A,H}\ot ((\varphi_A\ot \phi_H)\co \delta_{H\ot A})\ot H)\co (A\ot \delta_H\ot A\ot H)$ 
{\scriptsize ({\blue by (\ref{anticomul}) for $\lambda_{H}$ and  $\lambda_{A}$, (\ref{delta-mu}) for $A$, and naturality}}
\item[ ]$\hspace{0.38cm}${\scriptsize {\blue  of $c$})}

\item[ ]$=(A\ot \mu_H)\co (((\varphi_A\ot \phi_H)\co \delta_{H\ot A})\ot \mu_H)\co (\lambda_H\ot ((\varphi_A\ot\phi_H)\co \delta_{H\ot A})\ot H)\co
(\varepsilon_A\ot \delta_H\ot A\ot H)$
\item[ ]$\hspace{0.38cm}$ {\scriptsize  ({\blue by (\ref{leftHqg}) for $A$ and naturality of $c$})}

\item[ ]$=(A\ot \mu_H)\co ((\varphi_A\co (H\ot \varphi_A))\ot (\phi_H\co (H\ot \varphi_A))\ot \mu_H)\co (\delta_{H\ot H\ot A}\ot \phi_H\ot H)\co (\lambda_H\ot \delta_{H\ot A}\ot H)$
\item[ ]$\hspace{0.38cm} \co (\varepsilon_A\ot \delta_H\ot A\ot H)$
{\scriptsize  ({\blue by the condition of comonoid morphism for $\varphi_A$ and naturality of $c$})}

\item[ ]$=(A\ot \mu_H)\co (A\ot \phi_H\ot \mu_H)\co (c_{H,A}\ot \varphi_A\ot \phi_{H}\ot H)\co ((\lambda_H\ot (\varphi_A\co ((\lambda_H\ast id_{H})\ot A))$
\item[ ]$\hspace{0.38cm} \co (\delta_{H}\ot A))\ot H\ot A\ot H\ot A\ot H)\co (\delta_{H\ot A}\ot \ot A\ot H)\co (\varepsilon_A\ot \delta_{H\ot A}\ot H)$
{\scriptsize  ({\blue by (\ref{anticomul}) for $\lambda_{H}$, the}}
\item[ ]$\hspace{0.38cm}${\scriptsize  {\blue  condition of $H$-module for $A$, coassociativity of $\delta_{H}$, and naturality of $c$})}

\item[ ]$=(A\ot (\mu_H\co(\phi_H\ot \mu_H)\co (\lambda_H\ot \varphi_A\ot \phi_H\ot H)\co (H\ot \delta_{H\ot A}\ot H)\co (\delta_H\ot A\ot H)))\co (c_{H,A}\ot A\ot H) $

\item[ ]$\hspace{0.38cm} \co (\varepsilon_A\ot H\ot \delta_A\ot H)$
{\scriptsize  ({\blue by (\ref{primeradealpha}), the condition of $H$-module for $A$ and counit properties})}

\item[ ]$=(A\ot \varepsilon_H\ot \varepsilon_A\ot H)\co (c_{H,A}\ot A\ot H)\co (\varepsilon_A\ot H\ot \delta_A\ot H)$ {\scriptsize  ({\blue by (\ref{primeraqg})})}

\item[ ]$=\varepsilon_{A\ot H}\ot A\ot H$ {\scriptsize  ({\blue by counit properties and naturality of $c$}).}

\end{itemize}

On the other hand,

\begin{itemize}
\item[ ]$\hspace{0.38cm} \mu_{A\bowtie H}\co (A\ot H\ot \mu_{A\bowtie H})\co (A\ot H\ot \lambda_{A\bowtie H}\ot A\ot H)\co (\delta_{A\ot H}\ot A\ot H)$

\item[ ]$=\mu_{A\bowtie H}\co (A\ot H\ot (\mu_A\co (A\ot\varphi_A)\co (((\varphi_A\ot \phi_H)\co \delta_{H\ot A})\ot A))$
\item[ ]$\hspace{0.38cm} \ot (\mu_H\co (\phi_H\ot H)\co (H\ot \mu_A\ot H)))
\co (A\ot H\ot \delta_{H\ot A\ot A}\ot H)\co (A\ot H\ot ((\lambda_H\ot \lambda_A)\co c_{A,H})\ot A\ot H)$
\item[ ]$\hspace{0.38cm} \co (\delta_{A\ot H}\ot A\ot H)$
{\scriptsize  ({\blue by the condition of  comonoid morphism for $\phi_H$, coassociativity, the  condition of $A$-module for}}
\item[ ]$\hspace{0.38cm}${\scriptsize  {\blue  $H$, and naturality of $c$})}

\item[ ]$=\mu_{A\bowtie H}\co (A\ot H\ot (\mu_A\co(A\ot \varphi_A)\co ((\lambda_{A\bowtie H}\co c_{H,A})\ot A))\ot (\mu_H\co (\phi_H\ot H)\co (H\ot \mu_A\ot H)$
\item[ ]$\hspace{0.38cm} \co (\lambda_H\ot \lambda_A\ot A\ot H)))\co (A\ot H\ot H\ot ((A\ot c_{H,A}\ot A)\co (c_{H,A}\ot c_{A,A}))\ot A\ot H)$
\item[ ]$\hspace{0.38cm} \co (A\ot H\ot (c_{H,H}\co \delta_H)\ot (c_{A,A}\co \delta_A)\ot \delta_A\ot H)\co (A\ot H\ot c_{A,H}\ot A\ot H)\co (\delta_{A\ot H}\ot A\ot H)$
\item[ ]$\hspace{0.38cm} ${\scriptsize ({\blue by (\ref{anticomul}) for the antipodes $\lambda_{A}$, $\lambda_{H}$, and naturality of $c$})}

\item[ ]$=\mu_{A\bowtie H}\co (A\ot H\ot ((\varphi_A\co (H\ot \mu_A)\co (\lambda_H\ot \lambda_A\ot A))\ot ((\mu_H\co (\phi_H\ot H)\co (H\ot \mu_A\ot H)$
\item[ ]$\hspace{0.38cm} \co (\lambda_H\ot \lambda_A\ot A\ot H))))\co (A\ot H\ot H\ot ((A\ot c_{H,A}\ot A)\co (c_{H,A}\ot c_{A,A}))\ot A\ot H)$
\item[ ]$\hspace{0.38cm} \co (A\ot H\ot (c_{H,H}\co \delta_H)\ot (c_{A,A}\co \delta_A)\ot \delta_A\ot H)\co (A\ot H\ot c_{A,H}\ot A\ot H)\co (\delta_{A\ot H}\ot A\ot H)$
\item[ ]$\hspace{0.38cm} ${\scriptsize ({\blue by  (\ref{muAfi})})}

\item[ ]$=\mu_{A\bowtie H}\co (A\ot H\ot A\ot \mu_{H})\co (A\ot H\ot ((\varphi_{A}\ot \phi_{H})\co (H\ot c_{H,A}\ot A)\co ((\delta_{H}\co \lambda_{H})\ot ((\mu_{A}\ot \mu_{A})\co \delta_{A\ot A}$
\item[ ]$\hspace{0.38cm} \co (\lambda_{A}\ot A))))\ot H)\co (A\ot H\ot c_{A,H}\ot A\ot H)\co (\delta_{A\ot H}\ot A\ot H)$ {\scriptsize ({\blue by  (\ref{anticomul}) for $\lambda_{A}$ and $\lambda_{H}$, and naturality }}
\item[ ]$\hspace{0.38cm}${\scriptsize {\blue  of $c$})}

\item[ ]$=(\mu_{A}\ot \mu_{H})\co (A\ot ((\varphi_{A}\ot \phi_{H})\co \delta_{H\ot A}\co (H\ot \varphi_{A}))\ot (\mu_{H}\co (\phi_{H}\ot H)))\co (A\ot H\ot (\delta_{H\ot A}\co (\lambda_{H}\ot (\mu_{A}$
\item[ ]$\hspace{0.38cm}\co (\lambda_{A}\ot A)))\co (c_{A,H}\ot A))\ot H)\co  (\delta_{A\ot H}\ot A\ot H)$ {\scriptsize ({\blue by  the definition of $\mu_{A\bowtie H}$, and (\ref{delta-mu}) for $H$})}

\item[ ]$=(\mu_A\ot \mu_H)\co (A\ot (\varphi_A\co (H\ot \varphi_A))\ot (\phi_H\co (H\ot \varphi_A))\ot \mu_H)\co (A\ot \delta_{H\ot H\ot A}\ot \phi_H\ot H)$
\item[ ]$\hspace{0.38cm} 
\co (A\ot H\ot \delta_{H\ot A}\ot H)\co (A\ot H\ot H\ot \mu_A\ot H)\co (A\ot H\ot ((\lambda_H\ot \lambda_A)\co c_{A,H})\ot A\ot H)$
\item[ ]$\hspace{0.38cm} \co (\delta_{A\ot H}\ot A\ot H)$
 {\scriptsize ({\blue by the condition of  comonoid morphism for $\varphi_A$})}

\item[ ]$=(\mu_A\ot ((\mu_H\co(\phi_H\ot \mu_H)\co (H\ot \varphi_A\ot \phi_H\ot H)\co (H\ot \delta_{H\ot A}\ot H)\co (((H\ot \lambda_H)\co\delta_H)\ot A\ot H))$
\item[ ]$\hspace{0.38cm} \co (A\ot (\varphi_A\co (\mu_H\ot A))\ot H\ot A\ot H)\co (A\ot H\ot \lambda_H\ot c_{H,A}\ot A\ot H)$
\item[ ]$\hspace{0.38cm}\co (A\ot H\ot (c_{H,H}\co \delta_H)\ot (\delta_A\co \mu_A)\ot H)\co
(A\ot H\ot H\ot\lambda_A\ot A\ot H)\co (A\ot H\ot c_{A,H}\ot A\ot H)$
\item[ ]$\hspace{0.38cm} \co (\delta_{A\ot H}\ot A\ot H)$
{\scriptsize ({\blue by (\ref{anticomul}) for $\lambda_{H}$, the condition of left $H$-module for $A$, coassociativity of the coproducts, and}}
\item[ ]$\hspace{0.38cm}${\scriptsize {\blue  naturality of $c$})}

\item[ ]$=((\mu_A\co (A\ot \varphi_A)\co (A\ot \mu_{H}\ot A))\ot \varepsilon_H\ot \varepsilon_A\ot H)\co  (A\ot H\ot H\ot c_{H,A}\ot A\ot H)\co (A\ot H\ot (c_{H,H}$
\item[ ]$\hspace{0.38cm}\co (H\ot \lambda_{H})\co \delta_{H})\ot (\delta_{A}\co \mu_{A}\co (\lambda_{A}\ot A))\ot H)\co (A\ot H\ot c_{A,H}\ot A\ot H)  \co (\delta_{A\ot H}\ot A\ot H)$ {\scriptsize ({\blue  by}}
\item[ ]$\hspace{0.38cm}${\scriptsize {\blue  (\ref{segundaqg})})}

\item[ ]$=((\mu_{A}\co (A\ot (\varphi_{A}\co ((\mu_{H}\co (H\ot \lambda_{H}))\ot (\mu_{A}\co (\lambda_{A}\ot A)))\co (H\ot c_{A,H}\ot A)))\co (\delta_{A\ot H}\ot A))\ot H)$ 
\item[ ]$\hspace{0.38cm}${\scriptsize ({\blue  by naturality of $c$ and counit properties})}

\item[ ]$=((\mu_A\co (A\ot \varphi_A)\co (A\ot (id_{H}\ast\lambda_{H})\ot \mu_A)\co (A\ot H\ot \lambda_A\ot A)\co (A\ot c_{A,H}\ot A)\co (\delta_A\ot H\ot A))\ot H)$
\item[ ]$\hspace{0.38cm}$ {\scriptsize ({\blue by naturality of $c$})}

\item[ ]$=((\mu_A\co (A\ot\mu_A)\co(A\ot \lambda_A\ot A)\co (\delta_A\ot A))\ot H)\co (A\ot \varepsilon_H\ot A\ot H)$
 {\scriptsize ({\blue by (\ref{new-h-2delta}) and (\ref{unidadquasimodulo}) for $\varphi_{A}$})}

\item[ ]$=\varepsilon_{A\ot H}\ot A\ot H$
{\scriptsize ({\blue by (\ref{leftHqg}) for $A$}).}

\end{itemize}

\end{proof}

Similarly,  the following result can be proved for right Hopf quasigroups.

\begin{theorem}
\label{productoMajid-r}
Let $A$, $H$ be right Hopf quasigroups with right antipodes $\varrho_{A}$, $\varrho_{H}$ respectively. Assume that $(A, \varphi_A)$ is a left $H$-module comonoid and $(H, \phi_H)$ is a right $A$-module comonoid. Then the following assertions are equivalent:

\begin{itemize}
\item [(i)] The double crossproduct $A\bowtie H$ built on the object $A\ot H$ with product
$$\mu_{A\bowtie H}=(\mu_A\ot \mu_H)\co (A\ot \varphi_A\ot \phi_H\ot H)\co (A\ot \delta_{H\ot A}\ot H)$$
and tensor product unit, counit and coproduct, is a right Hopf quasigroup with right antipode
$$\varrho_{A\bowtie H}=(\varphi_A\ot \phi_H)\co \delta_{H\ot A}\co (\varrho_{H}\ot \varrho_{A})\co c_{A,H}.$$

\item [(ii)] The  equalities (\ref{EtaAfi}), (\ref{EtaHpsi}), (\ref{fipsicompatibles}), 
\begin{equation}\label{muAfi-r}
\phi_H\co (\mu_{H}\ot A)\co (H\ot\varrho_H\ot \varrho_A)=\mu_H\co(\phi_{H}\ot H)\co (H\ot (\varrho_{A\bowtie H}\co c_{H,A})),
 \end{equation}
\begin{equation}\label{primeraqg-r}
\mu_A\co(\mu_{A}\ot \varphi_A)\co (A\ot ((\varphi_{A}\ot \phi_{H})\co \delta_{H\ot A})\ot \varrho_{A})\co (A\ot H\ot \delta_{A})=A\ot \varepsilon_H\ot \varepsilon_A,
 \end{equation}
\begin{equation}\label{segundaqg-r}
\mu_{A}\co (\mu_{A}\ot \varphi_{A})\co (A\ot ((\varphi_{A}\ot \phi_{H})\co \delta_{H\ot A})\ot A)\co (A\ot H\ot ((\varrho_{A}\ot A)\co \delta_{A}))=A\ot \varepsilon_H\ot \varepsilon_A,
 \end{equation}
hold.

\end{itemize}

\end{theorem}

Then as a consequence of  Theorems (\ref{productoMajid}) and (\ref{productoMajid-r}) we have:

\begin{corollary}
\label{productoMajid-q}
Let $A$, $H$ be Hopf quasigroups with  antipodes $\lambda_{A}$, $\lambda_{H}$ respectively. Assume that $(A, \varphi_A)$ is a left $H$-module comonoid and $(H, \phi_H)$ is a right $A$-module comonoid. Then the following assertions are equivalent:

\begin{itemize}
\item [(i)] The double crossproduct $A\bowtie H$ built on the object $A\ot H$ with product
$$\mu_{A\bowtie H}=(\mu_A\ot \mu_H)\co (A\ot \varphi_A\ot \phi_H\ot H)\co (A\ot \delta_{H\ot A}\ot H)$$
and tensor product unit, counit and coproduct, is a  Hopf quasigroup with  antipode
$$\lambda_{A\bowtie H}=(\varphi_A\ot \phi_H)\co \delta_{H\ot A}\co (\lambda_{H}\ot \lambda_{A})\co c_{A,H}.$$

\item [(ii)] The  equalities (\ref{EtaAfi}), (\ref{EtaHpsi}), (\ref{fipsicompatibles}), (\ref{muAfi}), (\ref{primeraqg}), (\ref{segundaqg}), (\ref{muAfi-r}), (\ref{primeraqg-r}), and (\ref{segundaqg-r}) hold for $\lambda_{H}$ and $\lambda_{A}$.

\end{itemize}

\end{corollary}

\begin{proof} The proof follows using  Theorems \ref{productoMajid} and \ref{productoMajid-r}, because in this case $\lambda_{H}=\varrho_{H}$ and $\lambda_{A}=\varrho_{A}$.

\end{proof}

Now we show that the construction  of $A\bowtie_{\tau} H$ introduced in the previous section is an example of a double cross product.

\begin{proposition}
\label{crossproductexample}
Let $A$, $H$ be left Hopf quasigroups with left antipodes $\lambda_{A}$, $\lambda_{H}$  such that $\lambda_H$ is an isomorphism, and let $\tau:A\ot H\rightarrow K$ be a skew pairing. Then the left Hopf quasigroup $A\bowtie_{\tau} H$ introduced in Corollary \ref{FangTorrecillas} is the double crossproduct induced by the actions $\varphi_{A}$, $\phi_{H}$ defined in Proposition \ref{doublecrossskewpairing}.

\end{proposition}

\begin{proof}
First note that 
\begin{equation}
\label{prin5}
(\varphi_A\ot \phi_H)\co \delta_{H\ot A}=(\tau\ot A\ot H\ot \tau^{-1})\co (A\ot H\ot\delta_{A\ot H})\co \delta_{A\ot H}\co c_{H,A}.
\end{equation}

Indeed:
\begin{itemize}
\item[ ]$\hspace{0.38cm} (\varphi_A\ot \phi_H)\co \delta_{H\ot A}$

\item[ ]$=(A\ot (\tau^{-1}\ast\tau)\ot H)\co (\tau\ot\delta_A\ot \delta_H\ot \tau^{-1})\co (A\ot H\ot \delta_{A\ot H})\co \delta_{A\ot H}\co c_{H,A}$ {\scriptsize ({\blue by naturality of $c$, and}}
\item[ ]$\hspace{0.38cm}${\scriptsize {\blue  coassociativity})}
                                 
\item[ ]$=(\tau\ot A\ot H\ot \tau^{-1})\co (A\ot H\ot\delta_{A\ot H})\co \delta_{A\ot H}\co c_{H,A}$ {\scriptsize ({\blue by  invertibility of $\tau$, naturality of $c$, and counit}}
\item[ ]$\hspace{0.38cm}${\scriptsize {\blue   properties}).}
\end{itemize}

Then, as a consequence,  it is not difficult to see that $\mu_{A\bowtie_{\tau} H}=\mu_{A\bowtie H}$. On the other hand $\lambda_{A\bowtie_{\tau} H}=\lambda_{A\bowtie H}$ because 

\begin{itemize}
\item[ ]$\hspace{0.38cm} \lambda_{A\bowtie H}$

\item[ ]$=(\tau\ot \lambda_{A}\ot \lambda_{H}\ot \tau^{-1})\co (A\ot H\ot ((A\ot c_{A,H}\ot H)\co ((c_{A,A}\co \delta_{A})\ot (c_{H,H}\co \delta_{H}))))\co (A\ot c_{A,H}\ot H)$
\item[ ]$\hspace{0.38cm}\co ((c_{A,A}\co \delta_{A})\ot (c_{H,H}\co \delta_{H}))$ {\scriptsize ({\blue by (\ref{anticomul}) for $\lambda_{A}$ and $\lambda_{H}$, naturality of $c$, $c^{2}=id$, (\ref{tau1}) and (\ref{tau2})})}

\item[ ]$=(\tau^{-1}\ot \lambda_{A}\ot \lambda_{H}\ot \tau)\co (A\ot H\ot c_{A\ot H,A\ot H})\co (c_{A\ot H,A\ot H}\ot A\ot H)\co (A\ot H\ot \delta_{A\ot H})\co c_{A\ot H,A\ot H}\co \delta_{A\ot H}$
\item[ ]$\hspace{0.38cm}$ {\scriptsize ({\blue by naturality of $c$, and $c^{2}=id$})}

\item[ ]$=(\tau^{-1}\ot \lambda_{A}\ot \lambda_{H}\ot \tau)\co (\delta_{A\ot H}\ot A\ot H) \co \delta_{A\ot H}$ {\scriptsize ({\blue by naturality of $c$, and $c^{2}=id$})}

\item[ ]$=\lambda_{A\bowtie_{\tau} H}$ {\scriptsize ({\blue by coasociativity}).}

\end{itemize}
\end{proof}

For the right side we have a similar result, i.e.:

\begin{proposition}
\label{crossproductexample-r}
Let $A$, $H$ be right Hopf quasigroups with right antipodes $\varrho_{A}$, $\varrho_{H}$ such that $\varrho_H$ is an isomorphism, and let $\tau:A\ot H\rightarrow K$ be a skew pairing. Then the right Hopf quasigroup $A\bowtie_{\tau} H$ introduced in Corollary \ref{FangTorrecillas-r} is the doublecross product induced by the actions $\varphi_{A}$, $\phi_{H}$ defined in Proposition \ref{doublecrossskewpairing}.

\end{proposition}

Finally, by Propositions \ref{crossproductexample}, \ref{crossproductexample-r}, and (iii) of Proposition \ref{doublecrossskewpairing},  we have the corresponding result for quasigroups without conditions over the antipode of $H$.

\begin{corollary}
\label{crossproductexample-q}
Let $A$, $H$ be  Hopf quasigroups  and let $\tau:A\ot H\rightarrow K$ be a skew pairing. Then the  Hopf quasigroup $A\bowtie_{\tau} H$ introduced in Corollary \ref{FangTorrecillas-q} is the double crossproduct induced by the actions $\varphi_{A}$, $\phi_{H}$ defined in Proposition \ref{doublecrossskewpairing}.
\end{corollary}

\section{Quasitriangular Hopf quasigroups, skew pairings, biproducts and projections}

In this section we will explore the connections between Yetter-Drinfeld modules for Hopf quasigroups, projections of Hopf quasigroups, skew pairings, and quasitriangular structures, obtaining the non-associative version of the main results proved in \cite{P-M}.  

In the Hopf quasigroup setting the notion of left $H$-comodule is exactly the same as for ordinary Hopf algebras since it only depends on the comonoid structure of $H$. Then,  we will denote a left $H$-comodule by $(M,\rho_{M})$ where $M$ is an object in ${\mathcal C}$ and $\rho_{M}:M\rightarrow H\ot M$ is a morphism in ${\mathcal C}$ (called the coaction) satisfying the comodule conditions:
\begin{equation}
\label{ep-coact}
(\varepsilon_{H}\ot M)\co\rho_{M}=id_{M}, 
\end{equation}
\begin{equation}
\label{delta-coact}
(H\otimes
\rho_{M})\co\rho_{M}=(\delta_{H}\ot M)\circ
\rho_{M}.
\end{equation} 

Given two left $H$-comodules $(M,\rho_{M})$ and
$(N,\rho_{N})$, $f:M\rightarrow N$ is a morphism of left
$H$-comodules if $\rho_{N}\co f=(H\ot f)\circ
\rho_{M}$. We denote the category of left $H$-comodules by
$\;^{H}{\mathcal C}$.

For  two left $H$-comodules $(M,\rho_{M})$ and
$(N,\rho_{N})$, the tensor product $M\ot N$ is a left $H$-comodule with the codiagonal coaction 
$$\rho_{M\ot N}=(\mu_{H}\ot M\ot N)\circ
(H\ot c_{M,H}\ot N)\co(\rho_{M}\otimes
\rho_{N}).$$

This makes the category of left $H$-comodules into a strict monoidal category $(^{H}{\mathcal C}, \ot , K)$.

Moreover, we will say that a unital magma $A$ is a left $H$-comodule magma if it is a left $H$-comodule with coaction $\rho_{A}$ and the following equalities hold:
\begin{equation}
\label{etacomodulo}
\rho_{A}\co\eta_A=\eta_H\ot \eta_A,
\end{equation}
\begin{equation}
\label{mucomodulo}
\rho_A\co\mu_A=(H\ot \mu_A)\co \rho_{A\ot A}.
\end{equation}

Finally, a comonoid $A$ is a left $H$-comodule comonoid if it is a a left $H$-comodule with coaction $\rho_A$ and 
\begin{equation}
\label{epsiloncomodulo}
(H\ot \varepsilon_A)\co \rho_A=\eta_H\ot \varepsilon_A,
\end{equation}
\begin{equation}
\label{deltacomodulo}
(H\ot \delta_A)\co \rho_A=\rho_{A\ot A}\co \delta_A,
\end{equation}
hold.

Now, following \cite{our1}, we recall the notion of Yetter-Drinfeld quasimodule for a Hopf quasigroup $H$.

\begin{definition}
\label{YD}
Let $H$ be a Hopf quasigroup. We say that $M=(M, \varphi_{M},
\rho_{M})$ is a left-left Yetter-Drinfeld quasimodule over $H$ if
 $(M,\varphi_{M})$ is a left $H$-quasimodule and $(M,\rho_{M})$ is a
left $H$-comodule which satisfies the following equalities:
\begin{itemize}
\item[(b1)]$(\mu_{H}\ot M)\co(H\ot c_{M,H})\co((\rho_{M}\co\varphi_{M})\ot H)\circ
(H\ot c_{H,M})\co(\delta_{H}\ot M)$
\item[]$=(\mu_{H}\ot\varphi_{M})\co(H\ot c_{H,H}\otimes
M)\co(\delta_{H}\ot\rho_{M}).$

\item[(b2)]$(\mu_{H}\ot M)\circ
(H\ot c_{M,H})\co(\rho_M\ot \mu_H)$
\item[ ]$=(\mu_H\ot M)\co(\mu_H\ot c_{M,H})\co(H\ot c_{M,H}\ot H)\co(\rho_M\ot H\ot H).$

\item[(b3)]$(\mu_{H}\ot M)\co(H\ot \mu_H\ot M)\co
(H\ot H\ot c_{M,H})\co(H\ot \rho_M\ot H)$
\item[ ]$=(\mu_H\ot M)\co(\mu_H\ot c_{M,H})\co(H\ot \rho_M\ot H).$
\end{itemize}

Let $M$ and $N$ be two left-left  Yetter-Drinfeld quasimodules over $H$.
We say that  $f:M\rightarrow N$ is a morphism of left-left Yetter-Drinfeld
quasimodules  if $f$ is a morphism of $H$-quasimodules and $H$-comodules.

\end{definition}

We shall denote by
$^{H}_{H}{\mathcal Q}{\mathcal Y}{\mathcal D}$ the category of left-left
Yetter-Drinfeld quasimodules over $H$. Moreover, if we assume that $M$ is a left $H$-module, we say that $M$ is a left-left Yetter-Drinfeld module over $H$. Obviously, left-left Yetter-Drinfeld modules with the obvious morphisms is a subcategory of $^{H}_{H}{\mathcal Q}{\mathcal Y}{\mathcal D}$. We will denote it by  $^{H}_{H}{\mathcal Y}{\mathcal D}$. Note that if $H$ is a Hopf algebra, conditions (b2) and (b3) trivialize. In this case,  $^{H}_{H}{\mathcal Y}{\mathcal D}$ is the classical category of left-left Yetter-Drinfeld modules over $H$.

Let $(M,\varphi_{M},\rho_{M})$ and $(N,\varphi_{N},\rho_{N})$ be two objects in $^{H}_{H}{\mathcal Q} {\mathcal Y} {\mathcal D}$. Then $M\ot N$, with the diagonal structure $\varphi_{M\ot N}$ and the codiagonal costructure $\rho_{M\ot N}$, is an object in 
$^{H}_{H}{\mathcal Q} {\mathcal Y} {\mathcal D}$. Then
$(^{H}_{H}{\mathcal Q} {\mathcal Y} {\mathcal D}, \ot, K)$ is a strict  monoidal category. If moreover $\lambda_H$ is an isomorphism,  $(^{H}_{H}{\mathcal Y} {\mathcal D}, \ot, K)$ is a strict braided monoidal category where the braiding $t$ and its inverse are defined by 
\begin{equation}
\label{braidyd}
t_{M,N}=(\varphi_N\ot M)\co(H\ot c_{M,N})\co(\rho_M\ot N)
\end{equation}
and 
$$t^{-1}_{M,N}=c_{N,M}\co ((\varphi_N\co c_{N,H})\ot M)\co (N\ot \lambda_H^{-1}\ot M)\co (N\ot \rho_M),$$ 
respectively (see Proposition 1.8 of \cite{our1}). As a consequence  we can consider Hopf quasigroups in  $^{H}_{H}{\mathcal {YD}}$. The definition is the following:

\begin{definition}
Let $H$ be a Hopf quasigroup such that its antipode is an isomorphism. 
 Let $(D,u_{D}, m_{D})$ be a unital magma in ${\mathcal C}$ such that $(D,e_{D},\Delta_{D})$ is a comonoid in ${\mathcal C}$, and let $s_{D}:D\rightarrow D$ be a morphism in ${\mathcal C}$. We say that the triple $(D, \varphi_D, \varrho_D)$ is a Hopf quasigroup in $^{H}_{H}{\mathcal {YD}}$ if:
\begin{itemize}
\item[(c1)] The triple $(D, \varphi_D, \rho_D)$  is a left-left Yetter-Drinfeld $H$-module.

\item[(c2)] The triple  $(D,u_{D}, m_{D})$ is a unital magma  in $^{H}_{H}{\mathcal {YD}}$, i.e.,  $(D,u_{D}, m_{D})$ is a unital magma in ${\mathcal C}$, 
$(D,\varphi_{D})$ is a left $H$-module magma and $(D,\rho_{D})$ is a left $H$-comodule magma.

\item[(c3)] The triple  $(D,e_{D},\Delta_{D})$ is a   comonoid in 
$^{H}_{H}{\mathcal {YD}}$, i.e., $(D,e_{D},\Delta_{D})$ is a  comonoid in  ${\mathcal C}$, $(D,\varphi_{D})$ is a left $H$-module comonoid and 
$(D,\rho_{D})$ is a left $H$-comodule comonoid.

\item[(c4)] The following identities hold:
\begin{itemize}
\item[(c4-1)] $e_{D}\co u_{D}=id_{K},$
\item[(c4-2)] $e_{D}\co m_{D}=e_{D}\ot e_{D},$
\item[(c4-3)] $\Delta_{D}\co e_{D}=e_{D}\ot e_{D},$
\item[(c4-4)] $\Delta_{D}\co m_{D}=(m_{D}\ot m_{D})\co (D\ot t_{D,D}\ot D)\co (\Delta_{D}\ot \Delta_{D})$,
\end{itemize}
where $t_{D,D}$ is the braiding of $^{H}_{H}{\mathcal {YD}}$ for $M=N=D$.

\item[(c5)] The following identities hold:
\begin{itemize}
\item[(c5-1)] $m_D\co(s_D\ot m_D)\co(\Delta_D\ot D)=e_D\ot D=m_D\co(D\ot m_D)\co(D\ot s_D\ot D)\co(\Delta_D\ot D).$
\item[(c5-2)] $m_D\co(m_D\ot D)\co(D\ot s_D\ot D)\co(D\ot \Delta_D)=D\ot e_D=\mu_D\circ(m_D\ot s_D)\co(D\ot \Delta_D).$
\end{itemize}
\end{itemize}

Note that under these conditions, $s_{D}$ is a morphism in $^{H}_{H}{\mathcal {YD}}$ (see Lemmas 1.11, 1.12 of \cite{our1}).

By Theorem 1.14 of \cite{our1} we know that if $(D, \varphi_D, \varrho_D)$ is a Hopf quasigroup in $^{H}_{H}{\mathcal {YD}}$,  then 
$$D\rtimes H=(D\ot H, \eta_{D\rtimes H}, \mu_{D\rtimes H}, \varepsilon_{D\rtimes H}, \delta_{D\rtimes H}, \lambda_{D\rtimes H})$$ is a Hopf quasigroup in ${\mathcal C}$, with the biproduct structure induced by the smash product coproduct, i.e., 
$$\eta_{D\rtimes H}=\eta_D \ot \eta_H, \;\;\;\mu_{D\rtimes H}=(\mu_D\ot \mu_H)\co(D\ot \Psi_{D}^{H}\ot H),$$
$$\varepsilon_{D\rtimes H}=\varepsilon_D\ot \varepsilon_H,\;\;\;\delta_{D\rtimes H}=(D\ot \Gamma_{D}^{H}\ot H)\co(\delta_D\ot \delta_H),$$
$$\lambda_{D\rtimes H}=\Psi_{D}^{H}\co (\lambda_H\ot \lambda_D)\co \Gamma_{D}^{H},$$
where the morphisms $\Gamma_{D}^{H}:D\ot H\rightarrow H\ot D,$ $\Psi_{D}^{H}:H\ot D\rightarrow D\ot H$  are defined by 
 $$\Gamma_{D}^{H}=(\mu_H\ot D)\co (H\ot c_{D,H})\co (\varrho_D\ot H),\;\;\;\;
 \Psi_{D}^{H}=(\varphi_D\ot H)\co (H\ot c_{H,D})\co (\delta_H\ot D). $$
\end{definition}

Let $H$ and $B$ be Hopf quasigroups and let $f:H\rightarrow B$ and $g:B\rightarrow H$ be morphisms of  Hopf quasigroups such that $g\co f=id_{H}$. By Proposition 2.1 of \cite{our1} we know that  $q_{H}^{B}=id_{B}\ast (f\co\lambda_{H}\co g):B\rightarrow B$
is an idempotent morphism. Moreover, if $B_H$ is the image of $q_{H}^{B}$ and $p_{H}^{B}:B\rightarrow B_H$, $i_{H}^{B}:B_H\rightarrow B$ a factorization of $q_{H}^{B}$, 
$$
\setlength{\unitlength}{3mm}
\begin{picture}(30,4)
\put(3,2){\vector(1,0){4}} \put(11,2.5){\vector(1,0){10}}
\put(11,1.5){\vector(1,0){10}} \put(1,2){\makebox(0,0){$B_{H}$}}
\put(9,2){\makebox(0,0){$B$}} \put(24,2){\makebox(0,0){$B\otimes
H$}} \put(5.5,3){\makebox(0,0){$i_{H}^{B}$}}
\put(16,3.5){\makebox(0,0){$(B\ot g)\co\delta_{B}$}}
\put(16,0.5){\makebox(0,0){$B\ot\eta_H$}}
\end{picture}
$$
is an equalizer diagram. As a consequence, the triple $(B_{H}, u_{B_{H}}, m_{B_{H}})$ is a unital magma where $u_{B_{H}}$ and $m_{B_{H}}$ are the factorizations, through the equalizer  $i_{H}^{B}$, of the morphisms $\eta_{B}$ and $\mu_{B}\co (i_{H}^{B}\ot i_{H}^{B})$, respectively. Therefore, the equalities 
\begin{equation}
\label{product-bh-2}
u_{B_{H}}=p_{H}^{B}\co \eta_{B}, \;\;\; m_{B_{H}}=p_{H}^{B}\co \mu_{B}\co  (i_{H}^{B}\ot i_{H}^{B}),
\end{equation}
hold.

\begin{definition}
Let $H$ be a  Hopf quasigroup. A  Hopf quasigroup
projection over $H$ is a triple $(B,f,g)$ where $B$ is a  Hopf
quasigroup, $f:H\rightarrow B$ and $g:B\rightarrow H$ are morphisms
of  Hopf quasigroups such that $g\co f=id_{H}$, and the equality
\begin{equation}
\label{condicionfuerte-1}
q_{H}^{B}\co \mu_B\ot (B\ot q_{H}^{B})=q_{H}^{B}\co\mu_B
\end{equation}
holds. 

A morphism between two Hopf quasigroup projections $(B,f,g)$ and
$(B^{\prime},f^{\prime},g^{\prime})$ over $H$  is a Hopf quasigroup morphism  $h:B\rightarrow B^{\prime}$ such that $h\co f=f^{\prime}$, $g^{\prime}\co h=g$.  Hopf quasigroup projections over
$H$ and morphisms of  Hopf quasigroup projections with the obvious composition form a category,
denoted by ${\mathcal P}roj(H)$.

\end{definition}

If $(B,f,g)$ is a Hopf quasigroup projection over $H$, 

$$
\setlength{\unitlength}{1mm}
\begin{picture}(101.00,10.00)
\put(20.00,8.00){\vector(1,0){25.00}}
\put(20.00,4.00){\vector(1,0){25.00}}
\put(55.00,6.00){\vector(1,0){21.00}}
\put(32.00,11.00){\makebox(0,0)[cc]{$\mu_{B}\co(B\ot f)$ }}
\put(33.00,0.00){\makebox(0,0)[cc]{$B\otimes
\varepsilon_H  $ }}
\put(65.00,9.00){\makebox(0,0)[cc]{$p_{H}^{B} $ }}
\put(13.00,6.00){\makebox(0,0)[cc]{$ B\ot H$ }}
\put(50.00,6.00){\makebox(0,0)[cc]{$ B$ }}
\put(83.00,6.00){\makebox(0,0)[cc]{$B_{H}$}}
\end{picture}
$$
is a coequalizer diagram. Moreover, the triple  $(B_H, e_{B_H}, \Delta_{B_H})$ is a comonoid, where $e_{B_H}$ and $\Delta_{B_H}$ are the factorizations,
through the coequalizer $p_{H}^{B}$, of the morphisms $\varepsilon_{B}$ and $(p_{H}^{B}\ot p_{H}^{B})\co \delta_{B}$, respectively. Moreover, the equalities 
\begin{equation}
\label{coproduct-bh}
e_{B_{H}}=\varepsilon_{B}\co i_{H}^{B}, \;\;\; \Delta_{B_{H}}= (p_{H}^{B}\ot p_{H}^{B})\co \delta_{B}\co i_{H}^{B},
\end{equation}
hold (see Proposition 2.3 of \cite{our1}).

\begin{definition}
 Let $H$ be a Hopf quasigroup. We say that a Hopf quasigroup projection $(B,f,g)$ over $H$ is strong if it satisfies 
\begin{equation}
 \label{condicionuno}
 p_{H}^{B}\co \mu_B\co (B\ot \mu_B)\co (i_{H}^{B}\ot f\ot i_{H}^{B})=p_{H}^{B}\co \mu_B\co (\mu_B\ot B)\co (i_{H}^{B}\ot f\ot i_{H}^{B}),
 \end{equation}
 \begin{equation}
 \label{condiciondos}
  p_{H}^{B}\co \mu_B\co (B\ot \mu_B)\co (f\ot i_{H}^{B}\ot i_{H}^{B})=p_{H}^{B}\co \mu_B\co (\mu_B\ot B)\co (f\ot i_{H}^{B}\ot i_{H}^{B}),
 \end{equation}
  \begin{equation}
 \label{condiciontres}
  p_{H}^{B}\co \mu_B\co (B\ot \mu_B)\co (f\ot f\ot i_{H}^{B})=p_{H}^{B}\co \mu_B\co (\mu_B\ot B)\co (f\ot f\ot i_{H}^{B}).
 \end{equation}
 
Note that, by the factorization of $q_{H}^{B}$, we have that (\ref{condicionuno}), (\ref{condiciondos}), and (\ref{condiciontres}) are equivalent to 
\begin{equation}
 \label{condicionuno-q}
 q_{H}^{B}\co \mu_B\co (B\ot \mu_B)\co (i_{H}^{B}\ot f\ot i_{H}^{B})=q_{H}^{B}\co \mu_B\co (\mu_B\ot B)\co (i_{H}^{B}\ot f\ot i_{H}^{B}),
 \end{equation}
 \begin{equation}
 \label{condiciondos-q}
  q_{H}^{B}\co \mu_B\co (B\ot \mu_B)\co (f\ot i_{H}^{B}\ot i_{H}^{B})=q_{H}^{B}\co \mu_B\co (\mu_B\ot B)\co (f\ot i_{H}^{B}\ot i_{H}^{B}),
 \end{equation}
  \begin{equation}
 \label{condiciontres-q}
  q_{H}^{B}\co \mu_B\co (B\ot \mu_B)\co (f\ot f\ot i_{H}^{B})=q_{H}^{B}\co \mu_B\co (\mu_B\ot B)\co (f\ot f\ot i_{H}^{B}).
 \end{equation}

 We will denote by  ${\mathcal {SP}}roj(H)$ the category of strong Hopf quasigroup projections over $H$. The morphisms of ${\mathcal {SP}}roj(H)$ are the morphisms of ${\mathcal {P}}roj(H)$.
 
\end{definition}

Let $H$ be a Hopf quasigroup with invertible antipode. By Proposition 2.7 of  \cite{our1}, if $D$ is a Hopf quasigroup in $^{H}_{H}{\mathcal Y}{\mathcal D}$,  the triple $(D\rtimes H, f=\eta_D\ot H, g=\varepsilon_D\ot H)$ is a strong Hopf quasigroup projection over $H$. In this case $q_{H}^{D\rtimes H}=D\ot \eta_{H}\ot \varepsilon_{H}$. As a consequence, we can choose $p_{H}^{D\rtimes H}=D\ot \varepsilon_{H}$ and $i_{H}^{D\rtimes H}=D\ot \eta_{H}$ and then $(D\rtimes H)_{H}=D$.

On the other hand, by Corollary 2.10 and Proposition 2.5 of \cite{our1}, we can assure that, if $(B,f,g)$ is a strong Hopf quasigroup projection over $H$, the triple  $(B_H, \varphi_{B_H}, \varrho_{B_H})$ is a Hopf quasigroup in $^{H}_{H}{\mathcal Y}{\mathcal D}$, where the magma-comonoid structure is defined by (\ref{product-bh-2}) and (\ref{coproduct-bh}), 
\begin{equation}
\label{act-coat-bh}
\varphi_{B_H}=p_{H}^{B}\co \mu_B\co (f\ot i_{H}^{B}),\;\;\; \rho_{B_H}=(g\ot p_{H}^{B})\co \delta_B\co i_{H}^{B}, 
\end{equation}
and 
\begin{equation}
\label{antipo-bh}
s_{B_H}=p_{H}^{B}\co ((f\circ g)\ast \lambda_{B})\co i_{H}^{B}. 
\end{equation}

Moreover, $w=\mu_B\co (i_{H}^{B}\ot f):B_H\rtimes H\rightarrow B$ is an isomorphism of Hopf quasigroups in ${\mathcal C}$ with inverse $w^{-1}=(p_{H}^{B}\ot g)\co \delta_B$ (see Propositions 2.8 and 2.9 of \cite{our1}). Therefore there exists an equivalence between the categories ${\mathcal {SP}}roj(H)$ and the category of Hopf quasigroups in $^{H}_{H}{\mathcal Y}{\mathcal D}$ (see Theorem 2.11 of \cite{our1}).

In the final part of the paper we will prove that we can construct examples of strong projections, and then of Hopf quasigroups in a braided setting, by working with quasitriangular structures and skew pairings. First we will introduce the notion of quasitriangular Hopf quasigroup.

\begin{definition}
\label{quasitri}
Let $H$ be a Hopf quasigroup. We will say that $H$ is quasitriangular if there exists a morphism $R:K\rightarrow H\ot H$ such that: 
\begin{itemize}
\item[(d1)] $(\delta_{H}\ot H)\co R=(H\ot H\ot \mu_{H})\co (H\ot c_{H,H}\ot H)\co (R\ot R),$
\item[(d2)] $(H\ot \delta_{H})\co R=(\mu_{H}\ot c_{H,H})\co (H\ot c_{H,H}\ot H)\co (R\ot R),$
\item[(d3)] $\mu_{H\ot H}\co ((c_{H,H}\co \delta_{H})\ot R)=\mu_{H\ot H}\co (R\ot \delta_{H}),$
\item[(d4)] $(\varepsilon_{H}\ot H)\co R=(H\ot \varepsilon_{H})\co R=\eta_{H}.$
\end{itemize}

In the Hopf algebra setting, the morphism $R$ is invertible for the convolution with inverse $R^{-1}=(\lambda_{H}\ot H)\co R$ and $R=(\lambda_{H}\ot \lambda_{H})\co R$. In our non-associative context we have that if $S=(\lambda_{H}\ot H)\co R$ and $T=(\lambda_{H}\ot \lambda_{H})\co R$,  the following identities hold: 
\begin{equation}
\label{quasi-1}
R\ast S=S\ast R=\eta_{H\ot H}, 
\end{equation}
\begin{equation}
\label{quasi-2}
S\ast T=T\ast S=\eta_{H\ot H}.
\end{equation}
Indeed:
\begin{itemize}
\item[ ]$\hspace{0.38cm} R\ast S$

\item[ ]$=((\mu_{H}\co (H\ot \lambda_{H}))\ot \mu_{H})\co (H\ot c_{H,H}\ot H)\co (R\ot R)$ {\scriptsize ({\blue by naturality of $c$})}

\item[ ]$=((id_{H}\ast \lambda_{H})\ot H)\co R$
{\scriptsize ({\blue by (d1) of Definition \ref{quasitri}})}

\item[ ]$=((\varepsilon_{H}\ot \eta_{H})\ot H)\co R$
{\scriptsize ({\blue by (\ref{new-h-2delta})})}

\item[ ]$=\eta_{H\ot H}$
{\scriptsize ({\blue by (d4) of Definition \ref{quasitri}}).}

\end{itemize}

Similarly, we prove $S\ast R=\eta_{H\ot H}$ using (\ref{primeradealpha}) instead of (\ref{new-h-2delta}). On the other hand 

\begin{itemize}
\item[ ]$\hspace{0.38cm} S\ast T$

\item[ ]$=((\mu_{H}\co (\lambda_{H}\ot \lambda_{H}))\ot (\mu_{H}\co (H\ot \lambda_{H})))\co (H\ot c_{H,H}\ot H)\co (R\ot R)$ {\scriptsize ({\blue by naturality of $c$})}

\item[ ]$=((\lambda_{H}\co\mu_{H}\co c_{H,H})\ot (\mu_{H}\co (H\ot \lambda_{H})))\co (H\ot c_{H,H}\ot H)\co (R\ot R)$
{\scriptsize ({\blue by (\ref{lambda-anti})})}

\item[ ]$=((\lambda_{H}\co\mu_{H})\ot (\mu_{H}\co (H\ot \lambda_{H})\co c_{H,H}))\co (H\ot c_{H,H}\ot H)\co (R\ot R)$ {\scriptsize ({\blue by naturality of $c$})}

\item[ ]$=(\lambda_{H}\ot (id_{H}\ast \lambda_{H}))\co R$
{\scriptsize ({\blue by (d2) of Definition \ref{quasitri}})}

\item[ ]$=(\lambda_{H}\ot (\varepsilon_{H}\ot \eta_{H}))\co R$
{\scriptsize ({\blue by (\ref{new-h-2delta})})}

\item[ ]$=\eta_{H\ot H}$
{\scriptsize ({\blue by (d4) of Definition \ref{quasitri} and (\ref{lambda-eta})}).}

\end{itemize}

The proof  for $T\ast S=\eta_{H\ot H}$ is similar using (\ref{primeradealpha}) instead of (\ref{new-h-2delta}). 

Note that, by the lack of associativity, we can not assure that the morphism $S$ be unique.

Finally, the identity
\begin{equation}
\label{quasi-3}
(\mu_{H}\ot H\ot (\mu_{H}\co c_{H,H}))\co (H\ot H\ot (c_{H,H}\co (H\ot \mu_{H}))\ot H)\co (H\ot c_{H,H}\ot c_{H,H}\ot H)\co (R\ot R\ot R)
\end{equation}
$$=(\mu_{H}\ot \mu_{H}\ot \mu_{H})\co (H\ot c_{H,H}\ot c_{H,H}\ot H)\co (R\ot R\ot R) $$
holds because 
\begin{itemize}
\item[ ]$\hspace{0.38cm} (\mu_{H}\ot H\ot (\mu_{H}\co c_{H,H}))\co (H\ot H\ot (c_{H,H}\co (H\ot \mu_{H}))\ot H)\co (H\ot c_{H,H}\ot c_{H,H}\ot H)\co (R\ot R\ot R)$

\item[ ]$= (H\ot (((\mu_{H}\co c_{H,H})\ot (\mu_{H}\co c_{H,H}))\co (H\ot c_{H,H}\ot H)))\co 
(((\mu_{H}\ot c_{H,H})\co (H\ot c_{H,H}\ot H)$
\item[ ]$\hspace{0.38cm}\co (R\ot R))\ot R)$ {\scriptsize ({\blue by naturality of $c$})}

\item[ ]$=(H\ot (((\mu_{H}\co c_{H,H})\ot (\mu_{H}\co c_{H,H}))\co (H\ot c_{H,H}\ot H)))\co 
(((H\ot \delta_{H})\co R)\ot R)$
{\scriptsize ({\blue by (d2) of Definition }}
\item[ ]$\hspace{0.38cm}$ {\scriptsize {\blue  \ref{quasitri}})}

\item[ ]$=(H\ot (\mu_{H\ot H}\co c_{H\ot H, H\ot H}\co (\delta_{H}\ot R)))\co R$ {\scriptsize ({\blue by $c^2=id$})}

\item[ ]$=(H\ot (\mu_{H\ot H}\co (R\ot \delta_{H}))) \co R$
{\scriptsize ({\blue by naturality of $c$})}

\item[ ]$=(H\ot (\mu_{H\ot H}\co ((c_{H,H}\co\delta_{H})\ot R)))\co R$
{\scriptsize ({\blue by (d3) of Definition \ref{quasitri}})}

\item[ ]$=(H\ot \mu_{H\ot H})\co (\mu_{H}\ot (c_{H,H}\co c_{H,H})\ot R)\co (H\ot c_{H,H}\ot H)\co (R\ot R)$
{\scriptsize ({\blue by (d2) of Definition \ref{quasitri}})}

\item[ ]$=(\mu_{H}\ot \mu_{H}\ot \mu_{H})\co (H\ot c_{H,H}\ot c_{H,H}\ot H)\co (R\ot R\ot R)$
{\scriptsize ({\blue by naturality of $c$}).}

\end{itemize}

\end{definition}

\begin{proposition}
\label{projec-1}
Let $A$, $H$ be  Hopf quasigroups and let $\tau:A\ot H\rightarrow K$ be a skew pairing. Assume that $H$ is quasitriangular with morphism $R$. Let $A\bowtie_{\tau} H$ be the Hopf quasigroup defined in Corollary \ref{FangTorrecillas-q}. Define the morphism 
$g:A\bowtie_{\tau} H\rightarrow H$ by 
$$g=(\tau\ot \mu_{H})\co (A\ot R\ot H).$$

Then, $g$ is a morphism of unital magmas if and only if the following equalities hold:
\begin{equation}
\label{rt-1}
\mu_{H}\co (g\ot H)=g\co (A\ot \mu_{H}),
\end{equation}
\begin{equation}
\label{rt-2}
\mu_{H}\co (H\ot g)=\mu_{H}\co (\mu_{H}\ot H)\co (H\ot ((\tau\ot H)\co (A\ot R))\ot H).
\end{equation}

\end{proposition}

\begin{proof} Assume that $g=(\tau\ot \mu_{H})\co (A\ot R\ot H)$ is a magma morphism. Then, 
\begin{equation}
\label{magma-mor}
g\co \mu_{A\bowtie_{\tau} H}=\mu_{H}\co (g\ot g)
\end{equation}
holds.  Moreover,  
\begin{itemize}
\item[ ]$\hspace{0.38cm}g\co (A\ot \mu_{H}) $

\item[ ]$=g\co \mu_{A\bowtie_{\tau} H}\co  (A\ot H\ot \eta_{A}\ot H)$ {\scriptsize ({\blue by naturality of $c$, (\ref{delta-eta}), (a3) of Definition \ref{skewpairing}, (\ref{skewpairing-1-r}), and unit and counit }}
\item[ ]$\hspace{0.38cm}${\scriptsize {\blue  properties})}

\item[ ]$=\mu_{H}\co (g\ot g)\co  (A\ot H\ot \eta_{A}\ot H)$
{\scriptsize ({\blue by (\ref{magma-mor})})}

\item[ ]$=\mu_{H}\co (g\ot H)$
{\scriptsize ({\blue by (a3) of Definition \ref{skewpairing}, (d4) of Definition \ref{quasitri}, and unit properties}).}

\end{itemize}

Therefore, (\ref{rt-1}) holds. On the other hand, the proof for (\ref{rt-2}) is the following:

\begin{itemize}
\item[ ]$\hspace{0.38cm}\mu_{H}\co (H\ot g) $

\item[ ]$=\mu_{H}\co ((g\co (\eta_{A}\ot H))\ot g)$ {\scriptsize ({\blue by (a3) of Definition \ref{skewpairing} and unit properties})}

\item[ ]$=g\co \mu_{A\bowtie_{\tau} H}\co (\eta_{A}\ot H\ot A\ot H) $
{\scriptsize ({\blue by (\ref{magma-mor})})}

\item[ ]$= (\tau\ot ( g\co (A\ot \mu_{H})))\co (((\delta_{A\ot H}\ot \tau^{-1})\co \delta_{A\ot H}\co c_{H,A})\ot H) $ {\scriptsize ({\blue by unit properties})}

\item[ ]$= (\tau\ot (\mu_{H}\co (g\ot H)))\co (((\delta_{A\ot H}\ot \tau^{-1})\co \delta_{A\ot H}\co c_{H,A})\ot H) $ {\scriptsize ({\blue by (\ref{rt-1})})}

\item[ ]$=((\tau\co (A\ot (\mu_{H}\co c_{H,H})))\ot (\mu_{H}\co (\mu_{H}\ot H)))\co (A\ot ((H\ot R\ot H)\co \delta_{H})\ot \tau^{-1}\ot H)\co ((\delta_{A\ot H}\co c_{H,A})\ot H) $ 
\item[ ]$\hspace{0.38cm}$ {\scriptsize ({\blue by (\ref{c2'-new})})}

\item[ ]$= (\tau\ot \mu_{H})\co (((A\ot (\mu_{H\ot H}\co (R\ot \delta_{H}))\ot \tau^{-1})\co \delta_{A\ot H}\co c_{H,A})\ot H)$
{\scriptsize ({\blue by naturality of $c$})}

\item[ ]$= (\tau\ot \mu_{H})\co (((A\ot (\mu_{H\ot H}\co ((c_{H,H}\co \delta_{H})\ot (R\co \tau^{-1}))))\co  \delta_{A\ot H}\co c_{H,A})\ot H)$ {\scriptsize ({\blue by (d3) of Definition \ref{quasitri}})}

\item[ ]$= \mu_{H}\co (((\tau\ot \tau\ot \mu_{H})\co (A\ot c_{H,H}\ot c_{H,H}\ot H)\co ((c_{A,A}\co \delta_{A})\ot (c_{H,H}\co \delta_{H})\ot (R\co \tau^{-1}))\co \delta_{A\ot H}\co c_{H,A})\ot H)$ 
\item[ ]$\hspace{0.38cm}$ {\scriptsize ({\blue by (a2') of Definition \ref{skewpairing}})}

\item[ ]$= \mu_{H}\co (((\tau\ot (\mu_{H}\co c_{H,H})\ot (\tau\ast \tau^{-1}))\co (A\ot R\ot H\ot A\ot H)\co \delta_{A\ot H}\co c_{H,A})\ot H)$ {\scriptsize ({\blue by naturality of $c$})}

\item[ ]$=(\tau\ot (\mu_{H}\co ((\mu_{H}\co c_{H,H})\ot H)))\co (A\ot R\ot H\ot H)\co (c_{H,A}\ot H)  $
{\scriptsize ({\blue by naturality of $c$, invertibility of $\tau$}}
\item[ ]$\hspace{0.38cm}$ {\scriptsize {\blue  and counit properties})}

\item[ ]$=\mu_{H}\co (\mu_{H}\ot H)\co (H\ot ((\tau\ot H)\co (A\ot R))\ot H) $ {\scriptsize ({\blue by naturality of $c$}).}

\end{itemize}

Conversely, assume that (\ref{rt-1}) and (\ref{rt-2}) hold. Firstly note that $g\co \eta_{A\bowtie_{\tau} H}=\eta_{H}$ follows by (a4) of Definition \ref{skewpairing}, (d4) of Definition \ref{quasitri}, and the unit properties. Secondly, 

\begin{itemize}
\item[ ]$\hspace{0.38cm}g\co \mu_{A\bowtie_{\tau} H} $

\item[ ]$= \mu_{H}\co ((g\co (\mu_{A}\ot H)\co (A\ot ((((\tau\ot A\ot H)\co \delta_{A\ot H})\ot \tau^{-1})\co \delta_{A\ot H}\co c_{H,A})))\ot H) $ 
 {\scriptsize ({\blue by (\ref{rt-1})})}

\item[ ]$=(((\tau\ot \tau)\co (A\ot c_{A,H}\ot H)\co (A\ot A\ot \delta_{H}))\ot (\mu_{H}\co (\mu_{H}\ot H)))
\co (A\ot \tau\ot A\ot R\ot H\ot \tau^{-1}\ot H)$
\item[ ]$\hspace{0.38cm} \co (A\ot ((\delta_{A\ot H}\ot A\ot H)\co \delta_{A\ot H}\co c_{H,A})\ot H)$ {\scriptsize ({\blue by (a1) of Definition \ref{skewpairing}})}

\item[ ]$= \mu_{H}\co (((\tau\ot \tau\ot (\mu_{H}\co (\mu_{H}\ot H))\ot H)\co (A\ot \tau\ot c_{A,H}\ot c_{H,H}\ot H\ot H\ot H)$
\item[ ]$\hspace{0.38cm}\co (A\ot A\ot c_{A,H}\ot \ot A\ot R\ot R\ot H\ot H)\co (A\ot \delta_{A}\ot \delta_{H}\ot \tau^{-1}\ot H)\co (A\ot  (\delta_{A\ot H}\co c_{H,A})\ot H)$ 
\item[ ]$\hspace{0.38cm}$ {\scriptsize ({\blue by (d1) of Definition \ref{quasitri}})}

\item[ ]$= \mu_{H}\co ( (\mu_{H}\co (((\tau\ot H)\co (A\ot R))\ot ((((\tau\ot \tau)\co (A\ot c_{A,H}\ot H)\co ((c_{A,A}\co \delta_{A})\ot c_{H,H}))\ot \mu_{H})$
\item[ ]$\hspace{0.38cm}\co (A\ot ((H\ot R\ot H)\co \delta_{H})))))\ot \tau^{-1}\ot H)\co (A\ot  (\delta_{A\ot H}\co c_{H,A})\ot H)$ {\scriptsize ({\blue by naturality of $c$, $c^2=id$, and}}
\item[ ]$\hspace{0.38cm}${\scriptsize {\blue (\ref{rt-2})})}

\item[ ]$=\mu_{H}\co ( (\mu_{H}\co (((\tau\ot H)\co (A\ot R))\ot ((\tau\ot H)\co (A\ot ((((\mu_{H}\co c_{H,H})\ot \mu_{H}))\co (H\ot R\ot H)$
\item[ ]$\hspace{0.38cm}\co \delta_{H})))))\ot \tau^{-1}\ot H)\co (A\ot  (\delta_{A\ot H}\co c_{H,A})\ot H)$  {\scriptsize ({\blue by (a2') of Definition \ref{skewpairing}})}

\item[ ]$=\mu_{H}\co ((\mu_{H}\co (((\tau\ot H)\co (A\ot R))\ot ((\tau\ot H)\co (A\ot ((\mu_{H\ot H}\co (R\ot \delta_{H})))))))\ot \tau^{-1}\ot H) $
\item[ ]$\hspace{0.38cm}\co (A\ot  (\delta_{A\ot H}\co c_{H,A})\ot H)$ {\scriptsize ({\blue  by naturality of $c$ and $c^{2}=id$})}

\item[ ]$=\mu_{H}\co ((\mu_{H}\co (((\tau\ot H)\co (A\ot R))\ot ((\tau\ot H)\co (A\ot ((\mu_{H\ot H}\co ((c_{H,H}\co \delta_{H})\ot R)))))))\ot \tau^{-1}\ot H) $
\item[ ]$\hspace{0.38cm}\co (A\ot  (\delta_{A\ot H}\co c_{H,A})\ot H)$ {\scriptsize ({\blue  by (d3) of Definition \ref{quasitri}})}

\item[ ]$=\mu_{H}\co ((\mu_{H}\co (((\tau\ot H)\co (A\ot R))\ot ((\tau\ot \tau\ot \mu_{H})\co (A\ot c_{A,H}\ot c_{H,H}\ot H)\co ((c_{A,A}\co \delta_{A})\ot (c_{H,H}$
\item[ ]$\hspace{0.38cm}\co \delta_{H})\ot R))))\ot \tau^{-1}\ot H) \co (A\ot  (\delta_{A\ot H}\co c_{H,A})\ot H)$ {\scriptsize ({\blue  by (a2') of Definition \ref{skewpairing}})}

\item[ ]$= \mu_{H}\co ((\mu_{H}\co (((\tau\ot H)\co (A\ot R))\ot ((\tau\ot (\mu_{H}\co c_{H,H}))\co (A\ot R\ot H))))\ot (\tau\ast\tau^{-1})\ot H)$
\item[ ]$\hspace{0.38cm}\co (A\ot  (\delta_{A\ot H}\co c_{H,A})\ot H) $
 {\scriptsize ({\blue by naturality of $c$})}

\item[ ]$= \mu_{H}\co ((g\co (A\ot \mu_{H}))\ot H)\co (A\ot H\ot ((\tau\ot H)\co (A\ot R))\ot H) $
{\scriptsize ({\blue by  invertibility of $\tau$ and counit}}
\item[ ]$\hspace{0.38cm}$ {\scriptsize {\blue  properties})}

\item[ ]$= \mu_{H}\co ((\mu_{H}\co (g\ot (((\tau\ot H)\co (A\ot R)))))\ot H)$
 {\scriptsize ({\blue by (\ref{rt-1})})}

\item[ ]$=\mu_{H}\co (g\ot g) $
 {\scriptsize ({\blue by (\ref{rt-2})}),}

\end{itemize}

and then, $g$ is morphism of unital magmas.

\end{proof}

\begin{remark}
In the previous Proposition, note that, if $H$ is a Hopf algebra, the equalities (\ref{rt-1}) and (\ref{rt-2}) always hold.

\end{remark}

\begin{theorem}
\label{theend}
Let $A$, $H$ be  Hopf quasigroups and let $\tau:A\ot H\rightarrow K$ be a skew pairing. Assume that $H$ is quasitriangular with morphism $R$. Let $A\bowtie_{\tau} H$ be the Hopf quasigroup defined in Corollary \ref{FangTorrecillas-q} and let $g:A\bowtie_{\tau} H\rightarrow H$ be the morphism introduced in Proposition \ref{projec-1}. Define the morphism 
$f:H\rightarrow A\bowtie_{\tau} H$ by $f=\eta_{A}\ot H.$ Then, if (\ref{rt-1}) and (\ref{rt-2}) hold, the triple $(A\bowtie_{\tau} H,f,g)$ is a strong Hopf quasigroup projection  over $H$.
\end{theorem}

\begin{proof} By Proposition \ref{projec-1} we know that $g$ is a morphism of unital magmas. Also, by (\ref{mu-eps}), (d4) of Definition \ref{quasitri}, and (a3) of Definition \ref{skewpairing}, we obtain that 
$\varepsilon_{H}\co g=\varepsilon_{A\bowtie_{\tau} H}$. Moreover, 
\begin{itemize}
\item[ ]$\hspace{0.38cm} \delta_{H}\co g$

\item[ ]$= (\tau \ot (\mu_{H\ot H}\co (\delta_{H}\ot \delta_{H})))\co (A\ot R\ot H)$ {\scriptsize ({\blue by (\ref{delta-mu})})}

\item[ ]$=(\tau\ot \tau\ot \mu_{H\ot H})\co (A\ot c_{A,H}\ot H\ot c_{H,H}\ot H\ot H)\co ((c_{A,A}\co \delta_{A})\ot ((H\ot c_{H,H}\ot H)\co (R\ot R))\ot \delta_{H})$ 
\item[ ]$\hspace{0.38cm}${\scriptsize ({\blue by (d2) of Definition \ref{quasitri}, and (a2') of Definition \ref{skewpairing}})}

\item[ ]$=(g\ot g)\co \delta_{A\bowtie_{\tau} H}$
{\scriptsize ({\blue by naturality of $c$ and $c^{2}=id$}).}

\end{itemize}

Therefore $g$ is a comonoid morphism. On the other hand, trivially $f\co \eta_{H}=\eta_{A\bowtie_{\tau} H}$ and  $\mu_{A\bowtie_{\tau} H}\co (f\ot f)=f\co \mu_{H}$ follows easily by  naturality of $c$, (\ref{delta-eta}), (a4) of Definition \ref{skewpairing}, (\ref{skewpairing-1}), and unit and counit properties. By (\ref{eta-eps}) it is clear that $\varepsilon_{A\bowtie_{\tau} H}\co f=\varepsilon_{H}$, and the identity $\delta_{A\bowtie_{\tau} H}\co f=(f\ot f)\co \delta_{H}$ can be proved using (\ref{delta-eta}) and the naturality of $c$. As a consequence, $f$ is a morphism of unital magmas and comonoids.  By  (a4) of Definition \ref{skewpairing} and (d4) of Definition \ref{quasitri} an easy computation shows that $g\co f=id_{H}$.

Our next goal is to obtain a simple expression for the idempotent morphism 
$$q_{H}^{A\bowtie_{\tau} H}=id_{A\bowtie_{\tau} H}\ast (f\co\lambda_{H}\co g):A\bowtie_{\tau} H\rightarrow A\bowtie_{\tau} H.$$ 

Indeed, the equality 
\begin{equation}
\label{q-complex}
q_{H}^{A\bowtie_{\tau} H}=(A\ot \tau\ot \lambda_{H})\co (\delta_{A}\ot R\ot\varepsilon_{H})
\end{equation}
holds because 

\begin{itemize}
\item[ ]$\hspace{0.38cm} q_{H}^{A\bowtie_{\tau} H}$

\item[ ]$= (A\ot (\mu_{H}\co (H\ot (\lambda_{H}\co g))))\co \delta_{A\ot H}$ {\scriptsize ({\blue by naturality of $c$, (\ref{delta-eta}), (a4) of Definition \ref{skewpairing}, (\ref{skewpairing-1}), and unit and counit}}
\item[ ]$\hspace{0.38cm}${\scriptsize {\blue   properties})}

\item[ ]$=(A\ot ( \mu_{H}\co  (H\ot \tau\ot (\mu_{H}\co c_{H,H}\co (\lambda_{H}\ot \lambda_{H})))\co (H\ot A\ot R\ot H)))\co \delta_{A\ot H} $ {\scriptsize ({\blue by (\ref{lambda-anti})})}

\item[ ]$=(A\ot (\mu_H\circ (H\ot \mu_H)\circ (H\ot \lambda_H\ot H)\circ (\delta_H\ot H)\co c_{H,H}\co (\tau\ot \lambda_{H}\ot H)))\co (\delta_{A}\ot R\ot H) $ {\scriptsize ({\blue by }}
\item[ ]$\hspace{0.38cm}${\scriptsize {\blue naturality of $c$})}

\item[ ]$=(A\ot \tau\ot \lambda_{H})\co (\delta_{A}\ot R\ot\varepsilon_{H})$ {\scriptsize ({\blue by (\ref{leftHqg})}).}

\end{itemize}

Then, using (\ref{q-complex}), we can prove that  $(A\bowtie_{\tau} H,f,g)$ is a Hopf quasigroup projection  over $H$. Indeed: 

\begin{itemize}
\item[ ]$\hspace{0.38cm} q_{H}^{A\bowtie_{\tau} H}\co \mu_{A\bowtie_{\tau} H}\co (A\ot H\ot q_{H}^{A\bowtie_{\tau} H})$

\item[ ]$= (A\ot \tau\ot H)\co ((\delta_{A}\co \mu_{A})\ot ((H\ot \lambda_{H})\co R))\co (A\ot ((\tau\ot A)\co (A\ot c_{A,H}\ot \tau^{-1})\co (\delta_{A}\ot H\ot A\ot H)$
\item[ ]$\hspace{0.38cm}\co \delta_{A\ot H}\co c_{H,A})\ot \varepsilon_{H})$ 
 {\scriptsize ({\blue by (\ref{mu-eps}), (\ref{lambda-vareps}), (d4) of Definition \ref{quasitri}, (a3) of Definition \ref{skewpairing},  and  counit properties})}

\item[ ]$=q_{H}^{A\bowtie_{\tau} H}\co \mu_{A\bowtie_{\tau} H}$ {\scriptsize ({\blue by (\ref{mu-eps}), and  counit properties}).}

\end{itemize}

Note that, by (\ref{delta-eta}), the naturality of $c$, (\ref{skewpairing-1}), (a4) of Definition \ref{skewpairing}, and unit and counit properties, we have the equality 
\begin{equation}
\label{aux-1-q}
\mu_{A\bowtie_{\tau} H}\co (A\ot H\ot \eta_{A}\ot H)=\mu_{H}, 
\end{equation}
and, by unit and counit properties, and (\ref{mu-eps}), the identity 
\begin{equation}
\label{aux-2-q}
q_{H}^{A\bowtie_{\tau} H}\co \mu_{A\bowtie_{\tau} H}\co (A\ot H\ot A\ot (\eta_{H}\co \varepsilon_{H}))=q_{H}^{A\bowtie_{\tau} H}\co \mu_{A\bowtie_{\tau} H}
\end{equation}
holds.

To finish the proof it is sufficient to show that the projection is strong, i.e., (\ref{condicionuno-q}), (\ref{condiciondos-q}), and (\ref{condiciontres-q}) hold.  Let us first prove (\ref{condicionuno-q}): 

\begin{itemize}
\item[ ]$\hspace{0.38cm} q_{H}^{A\bowtie_{\tau} H}\co \mu_{A\bowtie_{\tau} H}\co (A\ot H\ot \mu_{A\bowtie_{\tau} H})\co (i_{H}^{A\bowtie_{\tau} H}\ot f\ot i_{H}^{A\bowtie_{\tau} H})$

\item[ ]$= (A\ot \tau\ot \lambda_{H})\co ((\delta_{A}\co \mu_{A})\ot (R\co \tau^{-1}))\co (A\ot ((((\tau\ot H)\co (A\ot c_{A,H})\co (\delta_{A}\ot H))\ot A\ot H)\co \delta_{A\ot H}))$
\item[ ]$\hspace{0.38cm}\co (A\ot c_{H,A})\co (i_{H}^{A\bowtie_{\tau} H}\ot ((((\tau\ot H)\co (A\ot c_{A,H})\co (\delta_{A}\ot H))\ot \tau^{-1})\co \delta_{A\ot H}\co (c_{H,A}\ot \varepsilon_{H})$
\item[ ]$\hspace{0.38cm}\co (H\ot i_{H}^{A\bowtie_{\tau} H})))$ {\scriptsize ({\blue by (\ref{mu-eps}), (\ref{aux-2-q}) and  unit and counit properties})}

\item[ ]$=(A\ot \tau\ot \lambda_{H})\co ((\delta_{A}\co \mu_{A})\ot (R\co ((\tau\ot \tau)\co (A\ot c_{A,H}\ot H)\co ((c_{A,A}\co \delta_{A})\ot H\ot H))))$
\item[ ]$\hspace{0.38cm}\co (A\ot (((c_{A,A}\co \delta_{A})\ot H\ot \tau^{-1})\co \delta_{A\ot H}\co c_{H,A})\ot H\ot \tau^{-1})\co (i_{H}^{A\bowtie_{\tau} H}\ot (((\delta_{A\ot H}\co c_{H,A})\ot \varepsilon_{H})$
\item[ ]$\hspace{0.38cm}\co (H\ot i_{H}^{A\bowtie_{\tau} H})))$ {\scriptsize ({\blue by naturality of $c$, coassociativity, and $c^{2}=id$})} 

\item[ ]$=(A\ot \tau\ot \lambda_{H})\co ((\delta_{A}\co \mu_{A})\ot (R\co \tau\co (A\ot \mu_{H})))\co (A\ot (((c_{A,A}\co \delta_{A})\ot H\ot \tau^{-1})\co \delta_{A\ot H}\co c_{H,A})\ot H\ot \tau^{-1})$
\item[ ]$\hspace{0.38cm}\co (i_{H}^{A\bowtie_{\tau} H}\ot (((\delta_{A\ot H}\co c_{H,A})\ot \varepsilon_{H})\co (H\ot i_{H}^{A\bowtie_{\tau} H})))$ {\scriptsize ({\blue by (a2') of Definition \ref{skewpairing}})}

\item[ ]$= (A\ot ((\tau\ot \lambda_{H})\co (A\ot R))\co ((\delta_{A}\co \mu_{A})\ot ((\tau^{-1}\ot \tau^{-1})\co (A\ot c_{A,H}\ot H)\co (\delta_{A}\ot H\ot H)))\co (A\ot ((((\tau\ot \delta_{A})$
\item[ ]$\hspace{0.38cm}\co (A\ot c_{A,H})\co (\delta_{A}\ot \mu_{H}))\ot H\ot H)\co (A\ot  \delta_{H\ot H})\co (c_{H,A}\ot H)))\co  (i_{H}^{A\bowtie_{\tau} H}\ot ((c_{H,A}\ot \varepsilon_{H})$
\item[ ]$\hspace{0.38cm}\co (H\ot i_{H}^{A\bowtie_{\tau} H})))$ {\scriptsize ({\blue by naturality of $c$, coassociativity, and $c^{2}=id$})} 

\item[ ]$= (A\ot ((\tau\ot \lambda_{H})\co (A\ot R))\co ((\delta_{A}\co \mu_{A})\ot \tau^{-1})\co 
(A\ot ((((\tau\ot \delta_{A})\co (A\ot c_{A,H}))\ot H)$
\item[ ]$\hspace{0.38cm}\co (\delta_{A}\ot ((\mu_{H}\ot \mu_{H})\co \delta_{H\ot H}))\co (c_{H,A}\ot H))) \co (i_{H}^{A\bowtie_{\tau} H}\ot ((c_{H,A}\ot \varepsilon_{H})\co (H\ot i_{H}^{A\bowtie_{\tau} H})))$ {\scriptsize ({\blue by  (\ref{skewpairing-3})})}

\item[ ]$= (A\ot ((\tau\ot \lambda_{H})\co (A\ot R))\co ((\delta_{A}\co \mu_{A})\ot \tau^{-1})\co 
(A\ot ((((\tau\ot \delta_{A})\co (A\ot c_{A,H}))\ot H)$
\item[ ]$\hspace{0.38cm}\co (\delta_{A}\ot (\delta_{H}\co \mu_{H}))\co (c_{H,A}\ot H))) \co (i_{H}^{A\bowtie_{\tau} H}\ot ((c_{H,A}\ot \varepsilon_{H})\co (H\ot i_{H}^{A\bowtie_{\tau} H})))$ {\scriptsize ({\blue by  (\ref{delta-mu})})}

\item[ ]$=(A\ot ((\tau\ot \lambda_{H})\co (A\ot R))\co ((\delta_{A}\co \mu_{A})\ot (\varepsilon_{H}\co \mu_{H}))\co (A\ot ((((\tau\ot A\ot H)\co \delta_{A\ot H})\ot \tau^{-1}) $
\item[ ]$\hspace{0.38cm}\co \delta_{A\ot H}\co c_{H,A})\ot H)\co (((A\ot \mu_{H})\co (i_{H}^{A\bowtie_{\tau} H}\ot H))\ot i_{H}^{A\bowtie_{\tau} H})$ {\scriptsize ({\blue by (\ref{mu-eps}), and  counit properties})}

\item[ ]$=q_{H}^{A\bowtie_{\tau} H}\co \mu_{A\bowtie_{\tau} H}\co (\mu_{A\bowtie_{\tau} H}\ot A\ot H)\co (i_{H}^{A\bowtie_{\tau} H}\ot f\ot i_{H}^{A\bowtie_{\tau} H})$ {\scriptsize ({\blue by (\ref{aux-1-q})}).}

\end{itemize}

Secondly, we will prove (\ref{condiciondos-q}): 

\begin{itemize}
\item[ ]$\hspace{0.38cm}q_{H}^{A\bowtie_{\tau} H}\co \mu_{A\bowtie_{\tau} H}\co (\mu_{A\bowtie_{\tau} H}\ot A\ot H)\co (f\ot i_{H}^{A\bowtie_{\tau} H}\ot i_{H}^{A\bowtie_{\tau} H})$

\item[ ]$= (A\ot \tau\ot \lambda_{H})\co ((\delta_{A}\co \mu_{A})\ot R)\co (A\ot ((((\tau\ot A)\co (A\ot c_{A,H})\co (\delta_{A}\ot H))\ot \tau^{-1})\co \delta_{A\ot H}\co c_{H,A}))$
\item[ ]$\hspace{0.38cm}\co (((\tau\ot A\ot \mu_{H})\co (((\delta_{A\ot H}\ot \tau^{-1})\co \delta_{A\ot H}\co c_{H,A})\ot H))\ot A)\co (H\ot i_{H}^{A\bowtie_{\tau} H}
\ot ((A\ot \varepsilon_{H})\co i_{H}^{A\bowtie_{\tau} H}))$ {\scriptsize ({\blue by}}
\item[ ]$\hspace{0.38cm}$ {\scriptsize {\blue  (\ref{mu-eps}), and  unit and counit properties})}

\item[ ]$= (A\ot \tau\ot \lambda_{H})\co ((\delta_{A}\co \mu_{A})\ot R)\co (A\ot ((((\tau\ot A)\co (A\ot c_{A,H})\co (\delta_{A}\ot H))\ot \tau^{-1})\co (A\ot c_{A,H}\ot H)$
\item[ ]$\hspace{0.38cm}\co (\delta_{A}\ot (\delta_{H}\co \mu_{H}))\co (c_{H,A}\ot H)))\co (((\tau\ot A\ot H)\co (((\delta_{A\ot H}\ot \tau^{-1})\co 
\delta_{A\ot H}\co c_{H,A})))\ot c_{H,A})$
\item[ ]$\hspace{0.38cm}\co (H\ot i_{H}^{A\bowtie_{\tau} H}
\ot ((A\ot \varepsilon_{H})\co i_{H}^{A\bowtie_{\tau} H}))$ {\scriptsize ({\blue by naturality of $c$})} 

\item[ ]$= (A\ot \tau\ot \lambda_{H})\co ((\delta_{A}\co \mu_{A})\ot R)\co (A\ot ((((\tau\ot A)\co (A\ot c_{A,H})\co (\delta_{A}\ot H))\ot \tau^{-1})\co (A\ot c_{A,H}\ot H)$
\item[ ]$\hspace{0.38cm}\co (\delta_{A}\ot ((\mu_{H}\ot \mu_{H})\co \delta_{H\ot H}))\co (c_{H,A}\ot H)))\co (((\tau\ot A\ot H)\co (((\delta_{A\ot H}\ot \tau^{-1})\co 
\delta_{A\ot H}\co c_{H,A})))\ot c_{H,A})$
\item[ ]$\hspace{0.38cm}\co (H\ot i_{H}^{A\bowtie_{\tau} H}
\ot ((A\ot \varepsilon_{H})\co i_{H}^{A\bowtie_{\tau} H}))$ {\scriptsize ({\blue by (\ref{delta-mu})})} 

\item[ ]$=(A\ot \tau\ot \lambda_{H})\co ((\delta_{A}\co \mu_{A})\ot (R\co \tau^{-1}\co (A\ot \mu_{H})))\co 
(A\ot (((\tau\co (A\ot \mu_{H}))\ot A\ot A\ot H\ot H)$
\item[ ]$\hspace{0.38cm}\co (A\ot c_{A\ot A,H\ot H}\ot H\ot H)\co (((\delta_{A}\ot A)\co \delta_{A})\ot \delta_{H\ot H})\co (c_{H,A}\ot H)))\co (((\tau\ot A\ot H)\co (((\delta_{A\ot H}\ot \tau^{-1})$
\item[ ]$\hspace{0.38cm}\co  \delta_{A\ot H}\co c_{H,A})))\ot c_{H,A})\co (H\ot i_{H}^{A\bowtie_{\tau} H}
\ot ((A\ot \varepsilon_{H})\co i_{H}^{A\bowtie_{\tau} H}))$ {\scriptsize ({\blue by naturality of $c$})} 

\item[ ]$=(A\ot \tau\ot \lambda_{H})\co ((\delta_{A}\co \mu_{A})\ot (R\co \tau^{-1}\co (A\ot \mu_{H})))\co 
(A\ot ((((\tau\ot \tau)\co (A\ot c_{A,H}\ot H)\co ((c_{A,A}\co \delta_{A})$
\item[ ]$\hspace{0.38cm}\ot H\ot H))\ot A\ot A\ot H\ot H)
\co (A\ot c_{A\ot A,H\ot H}\ot H\ot H)\co (((\delta_{A}\ot A)\co \delta_{A})\ot \delta_{H\ot H})\co (c_{H,A}\ot H)))$
\item[ ]$\hspace{0.38cm}\co (((\tau\ot A\ot H)\co (((\delta_{A\ot H}\ot \tau^{-1})
\co  \delta_{A\ot H}\co c_{H,A})))\ot c_{H,A})\co (H\ot i_{H}^{A\bowtie_{\tau} H}
\ot ((A\ot \varepsilon_{H})\co i_{H}^{A\bowtie_{\tau} H}))$ 
\item[ ]$\hspace{0.38cm}$ {\scriptsize ({\blue  by (a2') of Definition \ref{skewpairing}})}

\item[ ]$=(((\tau\ot \tau)\co (A\ot c_{A,H}\ot H)\co (A\ot A\ot \delta_{H}))\ot ((A\ot \tau\ot \lambda_{H})\co (\delta_{A}\ot R)))\co (A\ot ((A\ot c_{A,H})$
\item[ ]$\hspace{0.38cm}\co (((A\ot \mu_{A})\co (c_{A,A}\ot A)\co (A\ot \delta_{A}))\ot H)\co (A\ot c_{H,A}))\ot \tau\ot (\tau^{-1}\co (A\ot \mu_{H})))$
\item[ ]$\hspace{0.38cm}\co (\delta_{A}\ot H\ot (c_{A,A}\co \delta_{A})\ot c_{A,H}\ot H\ot H)\co (A\ot H\ot \delta_{A}\ot c_{H,H}\ot H)\co (A\ot H\ot c_{H,A}\ot \delta_{H})$
\item[ ]$\hspace{0.38cm}\co 
(((A\ot \delta_{H}\ot \tau^{-1})\co \delta_{A\ot H}\co c_{H,A})\ot c_{H,A})\co (H\ot i_{H}^{A\bowtie_{\tau} H}
\ot ((A\ot \varepsilon_{H})\co i_{H}^{A\bowtie_{\tau} H}))$ {\scriptsize ({\blue by naturality of $c$,}}
\item[ ]$\hspace{0.38cm}$ {\scriptsize {\blue  coassociativity, and $c^{2}=id$})}

\item[ ]$= (\tau\ot  ((A\ot \tau\ot \lambda_{H})\co (\delta_{A}\ot R)))\co (A\ot c_{A,H})\co (((\mu_{A}\ot \mu_{A})\co \delta_{A\ot A})\ot H)$
\item[ ]$\hspace{0.38cm}\co (A\ot c_{H,A}\ot \tau\ot (\tau^{-1}\co (A\ot \mu_{H})))\co  (A\ot H\ot (c_{A,A}\co \delta_{A})\ot c_{A,H}\ot H\ot H)\co  (A\ot H\ot \delta_{A}\ot c_{H,H}\ot H)$
\item[ ]$\hspace{0.38cm}\co (A\ot H\ot c_{H,A}\ot \delta_{H})\co (((A\ot \delta_{H}\ot \tau^{-1})\co \delta_{A\ot H}\co c_{H,A})\ot c_{H,A})\co (H\ot i_{H}^{A\bowtie_{\tau} H}
\ot ((A\ot \varepsilon_{H})\co i_{H}^{A\bowtie_{\tau} H}))$
\item[ ]$\hspace{0.38cm}$ {\scriptsize ({\blue by  (a1) of Definition \ref{skewpairing}})}

\item[ ]$= (\tau\ot  ((A\ot \tau\ot \lambda_{H})\co (\delta_{A}\ot R)))\co (A\ot c_{A,H})\co (((\delta_{A}\co \mu_{A})\ot H)$
\item[ ]$\hspace{0.38cm}\co (A\ot c_{H,A}\ot \tau\ot (\tau^{-1}\co (A\ot \mu_{H})))\co  (A\ot H\ot (c_{A,A}\co \delta_{A})\ot c_{A,H}\ot H\ot H)\co  (A\ot H\ot \delta_{A}\ot c_{H,H}\ot H)$
\item[ ]$\hspace{0.38cm}\co (A\ot H\ot c_{H,A}\ot \delta_{H})\co (((A\ot \delta_{H}\ot \tau^{-1})\co \delta_{A\ot H}\co c_{H,A})\ot c_{H,A})\co (H\ot i_{H}^{A\bowtie_{\tau} H}
\ot ((A\ot \varepsilon_{H})\co i_{H}^{A\bowtie_{\tau} H}))$
\item[ ]$\hspace{0.38cm}$ {\scriptsize ({\blue by  (\ref{delta-mu})})}

\item[ ]$= (\tau\ot  ((A\ot \tau\ot \lambda_{H})\co (\delta_{A}\ot R)))\co (A\ot c_{A,H})\co (((\delta_{A}\co \mu_{A})\ot H)$
\item[ ]$\hspace{0.38cm}\co (A\ot c_{H,A}\ot \tau\ot ((\tau^{-1}\ot \tau^{-1})\co (A\ot c_{A,H}\ot H)\co (\delta_{A}\ot H\ot H)))\co  (A\ot H\ot (c_{A,A}\co \delta_{A})\ot c_{A,H}\ot H\ot H)$
\item[ ]$\hspace{0.38cm}\co  (A\ot H\ot \delta_{A}\ot c_{H,H}\ot H) \co (A\ot H\ot c_{H,A}\ot \delta_{H})\co (((A\ot \delta_{H}\ot \tau^{-1})\co \delta_{A\ot H}\co c_{H,A})\ot c_{H,A})$
\item[ ]$\hspace{0.38cm} \co (H\ot i_{H}^{A\bowtie_{\tau} H}
\ot ((A\ot \varepsilon_{H})\co i_{H}^{A\bowtie_{\tau} H}))$ {\scriptsize ({\blue by  (\ref{skewpairing-3})})}

\item[ ]$=   (A\ot \tau\ot \lambda_{H})\co (\delta_{A}\ot R)\co (((\tau\ot A)\co (A\ot c_{A,H})\co ((\delta_{A}\co \mu_{A})\ot H))\ot ((\tau^{-1}\ot \tau^{-1})\co (A\ot c_{A,H}\ot H)$
\item[ ]$\hspace{0.38cm}\co (A\ot A\ot (c_{H,H}\co \delta_{H}))))\co \delta_{A\ot A\ot H}\co (A\ot c_{H,A})\co (c_{H,A}\ot A)\co (H\ot A\ot ((A\ot\tau\ot \tau^{-1})$
\item[ ]$\hspace{0.38cm}\co ((c_{A,A}\co \delta_{A})\ot H\ot A\ot H)\co \delta_{A\ot H}\co c_{H,A}))\co (H\ot i_{H}^{A\bowtie_{\tau} H}
\ot ((A\ot \varepsilon_{H})\co i_{H}^{A\bowtie_{\tau} H}))$ {\scriptsize ({\blue by naturality}} 
\item[ ]$\hspace{0.38cm}${\scriptsize {\blue  of $c$, coassociativity, and $c^{2}=id$})}

\item[ ]$=   (A\ot \tau\ot \lambda_{H})\co (\delta_{A}\ot R)\co (((\tau\ot A)\co (A\ot c_{A,H})\co ((\delta_{A}\co \mu_{A})\ot H))\ot (\tau^{-1}\co (\mu_{A}\ot H))$
\item[ ]$\hspace{0.38cm}\co \delta_{A\ot A\ot H}
\co (A\ot c_{H,A})\co (c_{H,A}\ot A)\co (H\ot A\ot ((A\ot\tau\ot \tau^{-1})\co ((c_{A,A}\co \delta_{A})\ot H\ot A\ot H)\co \delta_{A\ot H}\co c_{H,A}))$
\item[ ]$\hspace{0.38cm}
\co (H\ot i_{H}^{A\bowtie_{\tau} H}
\ot ((A\ot \varepsilon_{H})\co i_{H}^{A\bowtie_{\tau} H}))$ {\scriptsize ({\blue by (\ref{skewpairing-4})})}

\item[ ]$=   (A\ot \tau\ot \lambda_{H})\co (\delta_{A}\ot R)\co (((\tau\ot A)\co (A\ot c_{A,H})\co (\delta_{A}\ot H))\ot \tau^{-1})\co (A\ot c_{A,H}\ot H)\co (((\mu_{A}\ot \mu_{A})$
\item[ ]$\hspace{0.38cm}\co \delta_{A\ot A})\ot \delta_{H})
\co (A\ot c_{H,A})\co (c_{H,A}\ot A)\co \co (H\ot A\ot ((A\ot\tau\ot \tau^{-1})\co ((c_{A,A}\co \delta_{A})\ot H\ot A\ot H)$
\item[ ]$\hspace{0.38cm}\co \delta_{A\ot H}\co c_{H,A}))
\co (H\ot i_{H}^{A\bowtie_{\tau} H}
\ot ((A\ot \varepsilon_{H})\co i_{H}^{A\bowtie_{\tau} H}))$ {\scriptsize ({\blue by naturality of $c$})} 

\item[ ]$=   (A\ot \tau\ot \lambda_{H})\co (\delta_{A}\ot R)\co (((\tau\ot A)\co (A\ot c_{A,H})\co (\delta_{A}\ot H))\ot \tau^{-1})\co (A\ot c_{A,H}\ot H)\co ((\delta_{A}$
\item[ ]$\hspace{0.38cm}\co \mu_{A})\ot \delta_{H})
\co (A\ot c_{H,A})\co (c_{H,A}\ot A)\co  (H\ot A\ot ((A\ot\tau\ot \tau^{-1})\co ((c_{A,A}\co \delta_{A})\ot H\ot A\ot H)$
\item[ ]$\hspace{0.38cm}\co \delta_{A\ot H}\co c_{H,A}))
\co (H\ot i_{H}^{A\bowtie_{\tau} H}
\ot ((A\ot \varepsilon_{H})\co i_{H}^{A\bowtie_{\tau} H}))$ {\scriptsize ({\blue by (\ref{delta-mu})})} 

\item[ ]$= q_{H}^{A\bowtie_{\tau} H}\co \mu_{A\bowtie_{\tau} H}\co (A\ot A\ot (\eta_{H}\co \varepsilon_{H}))\co (A\ot H\ot \mu_{A\bowtie_{\tau} H})\co (f\ot i_{H}^{A\bowtie_{\tau} H}\ot i_{H}^{A\bowtie_{\tau} H})$ {\scriptsize ({\blue by (\ref{mu-eps}),  (\ref{eta-eps}),}}
\item[ ]$\hspace{0.38cm}$ {\scriptsize {\blue naturality of $c$, and  unit and counit properties})}

\item[ ]$= q_{H}^{A\bowtie_{\tau} H}\co \mu_{A\bowtie_{\tau} H}\co (A\ot H\ot \mu_{A\bowtie_{\tau} H})\co (f\ot i_{H}^{A\bowtie_{\tau} H}\ot i_{H}^{A\bowtie_{\tau} H})$ {\scriptsize ({\blue by (\ref{aux-2-q})}).}

\end{itemize}

Finally, we prove   (\ref{condiciontres-q}):

\begin{itemize}
\item[ ]$\hspace{0.38cm} q_{H}^{A\bowtie_{\tau} H}\co \mu_{A\bowtie_{\tau} H}\co (A\ot H\ot \mu_{A\bowtie_{\tau} H})\co (f\ot f\ot i_{H}^{A\bowtie_{\tau} H})$

\item[ ]$=(A\ot \tau\ot \lambda_{H})\co (\delta_{A}\ot R)\co (((\tau\ot A)\co (A\ot c_{A,H})\co (\delta_{A}\ot H))\ot \tau^{-1})\co  \delta_{A\ot H}\co c_{H,A}\co (H\ot  ((((\tau\ot A)$
\item[ ]$\hspace{0.38cm}\co (A\ot c_{A,H})\co (\delta_{A}\ot H))\ot \tau^{-1})\co  \delta_{A\ot H}\co c_{H,A}))\co (H\ot H\ot ((A\ot \varepsilon_{H})\co i_{H}^{A\bowtie_{\tau} H})) $ {\scriptsize ({\blue   by (\ref{mu-eps}), and  unit}}
\item[ ]$\hspace{0.38cm}${\scriptsize {\blue  and counit properties})}

\item[ ]$= (((\tau\ot \tau)\co (A\ot c_{A,H}\ot H)\co ((c_{A,A}\co \delta_{A})\ot H\ot H))\ot ((A\ot \tau\ot \lambda_{H})\co (\delta_{A}\ot (R\co \tau^{-1}))))$
\item[ ]$\hspace{0.38cm}\co (A\ot H\ot c_{A,H}\ot A\ot H)\co (((((A\ot c_{A,H})\co (\delta_{A}\ot H))\ot H\ot \tau^{-1})\co \delta_{A\ot H}\co c_{H,A})\ot H\ot A\ot H)$
\item[ ]$\hspace{0.38cm}\co (H\ot (\delta_{A\ot H}\co c_{H,A}))\co (H\ot H\ot ((A\ot \varepsilon_{H})\co i_{H}^{A\bowtie_{\tau} H})) $ {\scriptsize ({\blue by naturality of $c$, coassociativity, and $c^{2}=id$})} 

\item[ ]$= ((\tau\co (A\ot \mu_{H}))\ot ((A\ot \tau\ot \lambda_{H})\co (\delta_{A}\ot (R\co \tau^{-1}))))$
\item[ ]$\hspace{0.38cm}\co (A\ot H\ot c_{A,H}\ot A\ot H)\co (((((A\ot c_{A,H})\co (\delta_{A}\ot H))\ot H\ot \tau^{-1})\co \delta_{A\ot H}\co c_{H,A})\ot H\ot A\ot H)$
\item[ ]$\hspace{0.38cm}\co (H\ot (\delta_{A\ot H}\co c_{H,A}))\co (H\ot H\ot ((A\ot \varepsilon_{H})\co i_{H}^{A\bowtie_{\tau} H})) $ {\scriptsize ({\blue by (a2') of Definition \ref{skewpairing}})}

\item[ ]$= ((\tau\co (A\ot \mu_{H}))\ot ((A\ot \tau\ot \lambda_{H})\co (\delta_{A}\ot (R\co((\tau^{-1}\ot \tau^{-1})\co (A\ot c_{A,H}\ot H)\co (\delta_{A}\ot H\ot H))))))$
\item[ ]$\hspace{0.38cm}\co (A\ot c_{A\ot A,H\ot H}\ot H\ot H)\co 
(((\delta_{A}\ot A)\co \delta_{A})\ot \delta_{H\ot H})\co (c_{H,A}\ot H)\co (H\ot c_{H,A})$
\item[ ]$\hspace{0.38cm}\co  (H\ot H\ot ((A\ot \varepsilon_{H})\co i_{H}^{A\bowtie_{\tau} H}))$ {\scriptsize ({\blue by naturality of $c$, coassociativity, and $c^{2}=id$})}

\item[ ]$= ((\tau\co (A\ot \mu_{H}))\ot ((A\ot \tau\ot \lambda_{H})\co (\delta_{A}\ot (R\co(\tau^{-1}\co (A\ot \mu_{H}))))))$
\item[ ]$\hspace{0.38cm}\co (A\ot c_{A\ot A,H\ot H}\ot H\ot H)\co 
(((\delta_{A}\ot A)\co \delta_{A})\ot \delta_{H\ot H})\co (c_{H,A}\ot H)\co (H\ot c_{H,A})$
\item[ ]$\hspace{0.38cm}\co  (H\ot H\ot ((A\ot \varepsilon_{H})\co i_{H}^{A\bowtie_{\tau} H}))$ {\scriptsize ({\blue by (\ref{skewpairing-3})})}

\item[ ]$= (A\ot \tau\ot \lambda_{H})\co (\delta_{A}\ot (R\co \tau^{-1}))\co (((\tau\ot A)\co (A\ot c_{A,H})\co (\delta_{A}\ot H))\ot A\ot H)\co (A\ot c_{A,H}\ot H)$
\item[ ]$\hspace{0.38cm}\co (\delta_{A}\ot ((\mu_{H}\ot \mu_{H})\co \delta_{H\ot H}))\co (c_{H,A}\ot H)\co (H\ot c_{H,A})\co  (H\ot H\ot ((A\ot \varepsilon_{H})\co i_{H}^{A\bowtie_{\tau} H}))$ {\scriptsize ({\blue  by}}
\item[ ]$\hspace{0.38cm}$ {\scriptsize {\blue naturality of $c$})}

\item[ ]$= (A\ot \tau\ot \lambda_{H})\co (\delta_{A}\ot (R\co \tau^{-1}))\co (((\tau\ot A)\co (A\ot c_{A,H})\co (\delta_{A}\ot H))\ot A\ot H)\co (A\ot c_{A,H}\ot H)$
\item[ ]$\hspace{0.38cm}\co (\delta_{A}\ot ((\delta_{H}\co \mu_{H}))\co (c_{H,A}\ot H)\co (H\ot c_{H,A})\co  (H\ot H\ot ((A\ot \varepsilon_{H})\co i_{H}^{A\bowtie_{\tau} H}))$ {\scriptsize ({\blue  by (\ref{delta-mu})})}

\item[ ]$=(\tau \ot ((A\ot \tau\ot \lambda_{H})\co (\delta_{A}\ot (R\co \varepsilon_{H}\co \mu_{H}))))\co 
(((\delta_{A\ot H}\ot \tau^{-1})\co \delta_{A\ot H}\co c_{H,A}) \ot H)\co (\mu_{H}\ot i_{H}^{A\bowtie_{\tau} H})$ 
\item[ ]$\hspace{0.38cm}$ {\scriptsize ({\blue  by naturality of $c$, (\ref{mu-eps}), and counit properties})}

\item[ ]$= q_{H}^{A\bowtie_{\tau} H}\co \mu_{A\bowtie_{\tau} H}\co (\mu_{A\bowtie_{\tau} H}\ot A\ot H)\co (f\ot f\ot i_{H}^{A\bowtie_{\tau} H})$ {\scriptsize ({\blue by (\ref{mu-eps}), (\ref{aux-1-q}), (\ref{aux-2-q}) and  unit and counit properties}).}

\end{itemize}

\end{proof}

\begin{corollary}
\label{thelastone}
Let $A$, $H$ be  Hopf quasigroups and let $\tau:A\ot H\rightarrow K$ be a skew pairing. Assume that $H$ is quasitriangular with morphism $R$. Then if (\ref{rt-1}) and (\ref{rt-2}) hold, there exist an action $\varphi_{A}$ and a coaction $\rho_{A}$ such that $(A, \varphi_{A}, \rho_{A})$ is a Hopf quasigroup in $^{H}_{H}{\mathcal Y}{\mathcal D}$. Moreover, $A\rtimes H$ and $A\bowtie_{\tau} H$ are isomorphic Hopf quasigroups in ${\mathcal C}$.

\end{corollary}

\begin{proof} By the proof of the previous theorem, we know that $(A\bowtie_{\tau} H,f=\eta_{A}\ot H,g=(\tau\ot \mu_{H})\co (A\ot R\ot H))$ is a strong projection over $H$ and (\ref{q-complex}) holds. Put 
$$p_{H}^{A\bowtie_{\tau} H}= A \ot\varepsilon_{H}, \;\; i_{H}^{A\bowtie_{\tau} H}= (A\ot \tau\ot \lambda_{H})\co (\delta_{A}\ot R).
$$
Then, $q_{H}^{A\bowtie_{\tau} H}=i_{H}^{A\bowtie_{\tau} H}\co p_{H}^{A\bowtie_{\tau} H}$ and 
$p_{H}^{A\bowtie_{\tau} H}\co i_{H}^{A\bowtie_{\tau} H}=id_{A}$ because 

\begin{itemize}
\item[ ]$\hspace{0.38cm} p_{H}^{A\bowtie_{\tau} H}\co i_{H}^{A\bowtie_{\tau} H}$

\item[ ]$= (A\ot \tau \ot \varepsilon_{H})\co (\delta_{A}\ot R)$ {\scriptsize ({\blue by (\ref{lambda-vareps})})}

\item[ ]$= (A\ot \tau)\co (\delta_{A}\ot \eta_{H})$ {\scriptsize ({\blue by (d4) of Definition \ref{quasitri}})}

\item[ ]$=id_{A} $ {\scriptsize ({\blue by (a3) of Definition \ref{skewpairing},  and counit properties}).}

\end{itemize}

Therefore, we can choose $A=(A\bowtie_{\tau} H)_{H}$, 

$$
\setlength{\unitlength}{3mm}
\begin{picture}(30,4)
\put(3,2){\vector(1,0){4}} 
\put(14,2.5){\vector(1,0){10}}
\put(14,1.5){\vector(1,0){10}}
 \put(1,2){\makebox(0,0){$A$}}
\put(10,2){\makebox(0,0){$A\bowtie_{\tau} H$}} 
\put(30,2){\makebox(0,0){$A\bowtie_{\tau} H\otimes H$}} 
\put(5.5,3){\makebox(0,0){$i_{H}^{A\bowtie_{\tau} H}$}}
\put(19,3.5){\makebox(0,0){$(A\bowtie_{\tau} H\ot g)\co\delta_{A\bowtie_{\tau} H}$}}
\put(19,0.5){\makebox(0,0){$A\bowtie_{\tau} H\ot\eta_H$}}
\end{picture}
$$
is an equalizer diagram and 

$$
\setlength{\unitlength}{1mm}
\begin{picture}(101.00,10.00)
\put(16.00,8.00){\vector(1,0){30.00}}
\put(16.00,4.00){\vector(1,0){30.00}}
\put(67.00,6.00){\vector(1,0){21.00}}
\put(31.00,11.00){\makebox(0,0)[cc]{$\mu_{A\bowtie_{\tau} H}\co(A\bowtie_{\tau} H\ot f)$ }}
\put(31.00,0.00){\makebox(0,0)[cc]{$A\bowtie_{\tau} H\otimes
\varepsilon_H  $ }}
\put(79.00,9.00){\makebox(0,0)[cc]{$p_{H}^{A\bowtie_{\tau} H} $ }}
\put(1.00,6.00){\makebox(0,0)[cc]{$ A\bowtie_{\tau} H\ot H$ }}
\put(58.00,6.00){\makebox(0,0)[cc]{$ A\bowtie_{\tau} H$ }}
\put(94.00,6.00){\makebox(0,0)[cc]{$A$}}
\end{picture}
$$

is a coequalizer diagram.  Moreover, by the general theory developed in \cite{our1} (see (\ref{act-coat-bh})), we know that $(A, \varphi_{A}, \varrho_{A})$ is a Hopf quasigroup in $^{H}_{H}{\mathcal Y}{\mathcal D}$, where $\varphi_{A}$  is the action defined in Proposition \ref{doublecrossskewpairing}, and 
$$\rho_{A}=(\tau \ot c_{A,H})\co (A\ot c_{A,H}\ot \mu_{H})\co (\delta_{A}\ot (R\co \tau)\ot \lambda_{H})\co (\delta_{A}\ot R).$$

By (\ref{product-bh-2}) and (\ref{coproduct-bh}), the new magma-comonoid structure of $A$ is 
$$u_{A}=\eta_{A}, \;\;\; m_{A}=\mu_{A}\co (A\ot \varphi_{A})\co ( i_{H}^{A\bowtie_{\tau} H}\ot A),$$
$$e_{A}=\varepsilon_{A}, \;\;\;\Delta_{A}=\delta_{A},$$
and the antipode $s_{A}$  (see (\ref{antipo-bh})) admits the following expression
$$s_{A}=(\tau\ot \varphi_{A})\co (A\ot R\ot \lambda_{A})\co \delta_{A}.$$

Finally, the isomorphism of Hopf quasigroups 
$$w=\mu_{A\bowtie_{\tau} H}\co (i_{H}^{A\bowtie_{\tau} H}\ot f):A\rtimes H\rightarrow A\bowtie_{\tau} H$$
is 
$$w=(A\ot \mu_{H})\co (i_{H}^{A\bowtie_{\tau} H}\ot H).$$

\end{proof}

\begin{example}
\label{thelastone}
Let $H_{4}$ be the $4$-dimensional Taft Hopf algebra and consider the Hopf quasigroup  $A\bowtie_{\tau} H_{4}$ constructed in Example \ref{prin-ej}. By \cite{Rad-93}, we know that  $H_{4}$ has a one-parameter family of quasitriangular structures $R_{\alpha}$ defined by 
$$R_{\alpha}=\frac{1}{2}(1\ot 1+1\ot x+x\ot 1-x\ot x)+\frac{\alpha}{2}(y\ot y-y\ot w+w\ot y+w\ot w).$$

Therefore, we are in the conditions of the previous corollary and, as a consequence, $A$ admits a structure of Hopf quasigroup in the category $^{H_{4}}_{H_{4}}{\mathcal Y}{\mathcal D}$. Moreover, $A\bowtie_{\tau} H_{4} \backsimeq A\rtimes H_{4}$.
\end{example}

\section*{Acknowledgements}
The  authors were supported by  Ministerio de Econom\'{\i}a y Competitividad of Spain. Grant MTM2016-79661-P: Homolog\'{\i}a, homotop\'{\i}a e invariantes categ\'oricos en grupos y \'algebras no asociativas.

\end{document}